\documentclass[reqno,12pt]{amsart}
\newcommand{\RR}{{\mathbb R}}
\newcommand{\CC}{{\mathbb C}}

\newcommand{\NN}{{\mathbb N}}

\newcommand{\mcS}{\mathcal{S}}
\newcommand{\Rp}{\mathrm{Re}}
\def\bege{\begin{equation}} \def\ende{\end{equation}}
\def\begr{\begin{eqnarray}} \def\endr{\end{eqnarray}}
\usepackage{amsmath}
\usepackage{amssymb}
\usepackage{cite}

\usepackage{bbm}
\usepackage{amsmath}
\usepackage{txfonts}
\usepackage{stmaryrd}
\usepackage{amssymb}
\usepackage{amsfonts}
\usepackage{times}
\usepackage{mathrsfs}

\usepackage{amsthm}
\usepackage{enumerate}

\usepackage{framed}
\usepackage{lipsum}
\usepackage{color}

\def\BB{ \mathbb{B}}
\def\SS{ \mathbb{S}}
\def\CC{ \mathbb{C}}
\newcommand{\DD}{{\mathbb D}}
\def\B{\mathcal{B}}
\def\R{\mathcal{R}}
\def\I{\mathcal{I}}
\def\T{\mathcal{T}}
\def\HT{\mathscr{T}}
\def\D{\mathbb{D}}
\def\N{\mathbb N}

\def\hD{\hat{\mathcal{D}}}
\def\dD{\mathcal{D}}
\def\a{\alpha}

\def\vp{\varphi}
\def\om{\omega}

\def\p{{\prime}}
\def\Waw{{W_{\alpha}^\om}}
\def\Wav{\Waw}
\def\begr{\begin{eqnarray}} \def\endr{\end{eqnarray}}

\def\msk{\medskip}
\def\ol{\overline}
\newtheorem{Lemma}{Lemma}
\newtheorem{Theorem}{Theorem}

\newtheorem{Proposition}{Proposition}

\begin{document}
\title[  Toeplitz operators on $A_\om^2(\BB)$ with $\om\in\hD$ ]{  Toeplitz operators on Bergman spaces induced by  doubling weights on the unit ball of $\mathbb{C}^n$}

 \author{ Juntao Du and    Songxiao Li$\dagger$ }

 \address{Juntao Du\\ Faculty of Information Technology, Macau University of Science and Technology, Avenida Wai Long, Taipa, Macau.}
 \email{jtdu007@163.com  }

 \address{Songxiao Li\\ Institute of Fundamental and Frontier Sciences, University of Electronic Science and Technology of China,
 610054, Chengdu, Sichuan, P.R. China. } \email{jyulsx@163.com}

 \begin{abstract}
 The boundedness and compactness of Toeplitz operator from $A_\omega^p$ to $A_\omega^q$ with  doubling weights $\omega$ are studied in this paper.   The characterizations of Schatten class  Toeplitz operators and Volterra operators on $A_\omega^2$  are also investigated.
 \thanks{$\dagger$ Corresponding author.}
 \end{abstract}
 \maketitle

\section{Introduction}
Let $\BB$ be the open unit ball of $\CC^n$ and $\SS$   the boundary of $\BB$. When $n=1$,  $\BB$ is  the open unit disk in the complex plane $\mathbb{C}$ and always denoted by $\D$. Let $H(\BB)$ denote the space of all holomorphic functions on $\BB$. For any two points
$z=(z_1,z_2,\cdots,z_n)\,\mbox{ and } \, w=(w_1,w_2,\cdots,w_n)$ in $\CC^n$, we define $\langle z,w \rangle=z_1\overline{w_1}+\cdots+z_n\overline{w_n}$ and
$$|z|=\sqrt{\langle z,z \rangle}=\sqrt{|z_1|^2+\cdots+|z_n|^2}.$$

 Let $d\sigma$  and $dV$ be the normalized Lebesgue surface and volume measures on $\SS$ and $\BB$, respectively.
For $0<p< \infty$, the  Hardy space $H^p(\BB) $(or $H^p$) is the space consisting of all functions $f\in H(\BB)$ such that
$$\|f\|_{H^p}=\sup_{0<r<1}M_p(r,f)<\infty,$$ where
$$M_p(r,f)=\left(\int_{\SS}|f(r\xi)|^pd\sigma(\xi)\right)^\frac{1}{p}, \,\,\mbox{ when }\,\,0<p<\infty.$$
 $H^\infty$ is the space consisting of all $f\in H(\BB)$ such that $\|f\|_{H^\infty}=\sup_{z\in\BB}|f(z)|<\infty.$

For any $f\in H(\BB)$, let $\Re f(z)$ be the radial derivative of $f$,  that is,
$$\Re f(z)=\sum_{k=1}^n z_k\frac{\partial f}{\partial z_k}(z),\,\,\,\,z=(z_1,z_2,\cdots,z_n)\in\BB.$$
The Bloch space $\B(\BB)$ consists of all $f\in H(\BB)$ such that
$$ \|f\|_{\B(\BB)}=|f(0)|+\sup_{z\in\BB}(1-|z|^2)|\Re f(z)|<\infty.$$
When $n=1$, $\|\cdot\|_{\B(\D)}$ is a little different from the norm defined in   classical way, see \cite{zhu} for example, but they are equivalent.
We  keep $\B$ as the abbreviation of  $\B(\BB)$.

Suppose $\om$ is a radial weight ( i.e., $\om$ is a positive, measurable and  integrable  function on $[0,1)$ and $\om(z)=\om(|z|)$ for all $z\in\BB$).  Let $\hat{\om}(r)=\int_r^1\om(t)dt$.
We say that
\begin{itemize}
  \item $\om$ is a doubling weight, denoted by $\om\in \hD$,  if there exists   $C>0$ such that
$$\hat{\om}(r)<C\hat{\om}(\frac{1+r}{2}) ,\,\,\mbox{ when } 0\leq r<1;$$
  \item  $\om$ is a regular weight,  denoted by  $\om\in    \R $, if there exists $C>0$ and $\delta\in(0,1)$ such that
$$\frac{1}{C}<\frac{\hat{\om}(r)}{(1-r)\om(r)}<C,\,\mbox{ when }\, r\in(\delta,1);$$
 \item  $\om$ is a rapidly increasing weight, denoted by  $\om\in\I$, if
$$\lim_{r\to 1} \frac{\hat{\om}(r)}{(1-r)\om(r)}=\infty;$$
  \item $\om$ is a reverse doubling weight, denoted by $\om\in\check{\dD}$, if
 there exist $K>1$ and $C=C(\om)>1$ such that
\begin{align}\label{0420-1}
\hat{\om}(r)\geq C\hat{\om}(1-\frac{1-r}{K}),\,\,\,\,r\in(0,1).
\end{align}
\end{itemize}

The regular weight is a natural extension of the classical weight $(1-r^2)^\alpha(\alpha>-1)$.
The rapidly increasing weight was introduced by Pel\'aez and R\"atty\"a in \cite{PjaRj2014book}.
 The doubling weight, which was introduced in  \cite{Pja2015},  is  the extension of the regular weight and the rapidly increasing weight.
   See  \cite{Pja2015,PjaRj2014book}  for more details about $\I,\R$ and  $\hD$.
Let $\dD=\hD\cap\check{\dD}$. It is easy to check that $\R\subset\dD$.
If $\om\in\dD$, let $K_\om$ be the infimum  of the $K$ such that (\ref{0420-1}) holds.
By Lemma 1.1 in \cite{PjaRj2014book},  $K_\om=1$ if $\om$ is continuous and regular.
More information about $\check{\dD}$ and $\dD$ can be seen in \cite{PjRj2019arxiv,KtPjRj2018arxiv}.

Suppose $\mu$ is  a positive Borel measure on $\BB$ and $0<p<\infty$. The Lebesgue space $L^p(\BB,d\mu)$ (or $L_\mu^p$, for brief) consists of all measurable
complex functions $f$ on $\BB$ such that $|f|^p$ is integrable with respect to $\mu$, that is, $f\in L^p(\BB,d\mu)$ if and only if
$$\|f\|_{L^p(\BB,d\mu)}=\left(\int_\BB |f(z)|^pd\mu(z)\right)^\frac{1}{p}<\infty.$$
$L^\infty(\BB,d\mu)$ (or $L_\mu^\infty$)  consists of all measurable complex functions $f$ on $\BB$ such that $f$ is essential bounded, that is, $f\in L^\infty(\BB,d\mu)$ if and only if
$$\|f\|_{L^\infty(\BB,d\mu)}=\inf_{E\subset\BB,\mu(E)=0}\sup_{z\in\BB\backslash E}|f(z)|<\infty.$$
More details about $L^p(\BB,d\mu)$ can be seen in \cite{Rw1980,Zk2005}. If  $\om\in \hD$,  letting $d\mu(z)=\om(z)dV(z)$,  $\mu$ is a Borel measure on $\BB$.
 Then, we will write $L^p(\BB,d\mu)$ as $L^p(\BB,\om dV)$ or $L_\om^p$.
When $n=1$ and $z\in\D$,  $dV(z)$ is  the normalized Lebesgue area measure on $\D$, i.e., $dV(z)=\frac{1}{\pi}dA(z)$. Then we can define
the corresponding Lebesgue spaces on the unit disk in the same way.

In \cite{PjaRj2014book}, J. Pel\'aez and J. R\"atty\"a introduced a new class function spaces $A_\om^p(\D)$, the weighted Bergman spaces induced by   rapidly increasing weights $\om$ in $\D$. That is
$$A_\om^p(\DD)=L^p(\D,\om dA)\cap H(\D),\,\,0<p<\infty.$$
See \cite{Pja2015,PjaRj2014book,PjaRj2015,PjaRj2016,PjRj2016jmpa,PjaRjSk2015mz,PjaRjSk2018jga} for more results on $A_\om^p(\D)$ with $\om\in\hD$.  In \cite{DjLsLxSy2019arxiv}, we extended the Bergman space $A_\om^p(\D)$  with $\om\in\hD$ to the unit ball $\BB$ of $\CC^n$.   That is
$$A_\om^p(\BB)=L^p(\BB,\om dV)\cap H(\BB),\,\,0<p<\infty.$$
For brief, let $A_\om^p=A_\om^p(\BB)$. As a subspace of $L^p(\BB,\om dV)$, the norm on $A_\om^p$ will be written as $\|\cdot\|_{A_\om^p}$.
 It is easy to check that $A_\om^p$ is a Banach space when $p\geq 1$ and a complete metric space with the distance $\rho(f,g)=\|f-g\|_{A_\om^p}^p$ when $0<p<1$.

When $\alpha>-1$  and $c_\a=\frac{\Gamma(n+\a+1)}{\Gamma(n+1)\Gamma(\a+1)}$, if $\om(z)=c_\alpha(1-|z|^2)^\alpha$, the space $A_\om^p$ becomes the classical weighted Bergman space $A_\alpha^p$, and we write $dV_\alpha(z)=c_\alpha(1-|z|^2)^\alpha dV(z)$.  When $\alpha=0$, $A^p_0=A^p$ is the standard Bergman space.
See \cite{Rw1980,Zk2005} for the theory of $H^p$ and $A_\alpha^p$.

Let  $\om_s=\int_0^1 r^s\om(r)dr$ and
$$B_z^\om(w)=\frac{1}{2n!}\sum_{k=0}^\infty \frac{(n-1+k)!}{k!\om_{2n+2k-1}} \langle w,z\rangle^k.$$
In \cite{DjLsLxSy2019arxiv2}, we proved that,  for any $f\in L^1(\BB,\om dV)$,
\begin{align*}
f(z)=\langle f, B_z^\om \rangle_{A_\om^2}=\int_\BB f(w)\ol{B_z^\om(w)}\om(w)dV(w).
\end{align*}
So, $B_z^\om$ is called the reproducing kernel of $A_\om^2$. More results about $B_z^\om$ can be seen in \cite{DjLsLxSy2019arxiv2}.

Assume that $\mu$ is a positive Borel measure on $\BB$. The Toeplitz operator  associated with $\mu$ is defined by
$$\T_\mu f (z)=\int_\BB f(\xi)\ol{B_z^\om(\xi)}d\mu(\xi),$$
and the Berezin transform  of $\T_\mu$ is defined by
$$\widetilde{\T_\mu}(z)=\frac{\langle \T_\mu B_z^\om, B_z^\om \rangle_{A_\om^2}}{\|B_z^\om\|_{A_\om^2}^2}.$$
Generally, the Berezin transform of a linear operator $T:A_\om^p\to A_\om^q$ is
$$\widetilde{T}(z)=\frac{\langle T B_z^\om, B_z^\om \rangle_{A_\om^2}}{\|B_z^\om\|_{A_\om^2}^2}.$$

Since 1970s, there are a lot of works focused on Toeplitz operators, see  \cite{Cla1973iumj,MgSc1979iumj,Zk1988jop}.
On the unit disk,  in \cite{PjaRj2014book, PjaRj2016}, Pel\'aez and R\"atty\"a completely characterized the Schatten class Toeplitz on $A_\om^2$ with $\om\in\hD$.
In \cite{PjaRjSk2018jga},  the authors investigate  the boundedness and compactness of $\T_\mu: A_\om^p\to A_\om^q$ with $\om\in\R$.
On the unit ball, Zhu studied the Schatten class Toeplitz operators on $A_\alpha^2$ in \cite{Zk2007nyjm}. Paul and Zhao described the boundedness and compactness of $\T_\mu:A_\alpha^p\to A_\beta^q$ in \cite{PjZr2015mmj}.

Motivated by \cite{PjaRj2014book, PjaRjSk2018jga}, we study the Toeplitz operator on Bergman spaces induced by doubling weights in the unit ball of $\CC^n$.
The paper is organized as follows. In section 2, we give some lemmas which   will be used later. In section 3, we study the boundedness and compactness of  $\T_\mu:A_\om^p\to A_\om^q$ with $\om\in\dD$. In section 4, we introduce a new kind of Dirichlet spaces induced by doubling weights  and investigate the Schatten class  Toeplitz operators on these Dirichlet spaces.   As an application, we get two characterizations of Schatten class  Toeplitz operators on $A_\om^2$.
In section 5, using the  characterizations of Schatten class  Toeplitz operators on $A_\om^2$, we describe the  Schatten class  Volterra integral operator $T_g$ on $A_\om^2$.

Throughout this paper, the letter $C$ will denote  constants and may differ from one occurrence to the other.
The notation $A \lesssim B$ means that there is a positive constant C such that $A\leq CB$.
The notation $A \approx B$ means $A\lesssim B$ and $B\lesssim A$.\msk

 \section{ Preliminary results}
 In this section, we   introduce some notations and some results obtained  in \cite{DjLsLxSy2019arxiv,DjLsLxSy2019arxiv2}.
For any $\xi,\tau\in\overline{\BB}$, let $d(\xi,\tau)=|1-\langle \xi,\tau\rangle|^\frac{1}{2}$.
 Then $d(\cdot,\cdot)$ is a nonisotropic metric.
For $r>0$   and $\xi\in\SS$, let
  $$Q(\xi,r)=\{\eta\in \SS: d(\xi,\eta) \leq r \}.$$
 $Q(\xi,r)$  is a nonisotropic metric ball in $\SS$ for all $\xi\in \SS$ and $r\in(0,1)$.
More information about $d(\cdot,\cdot)$ and $Q(\xi,r)$ can be found in \cite{Rw1980,Zk2005}.

For any $a\in\BB\backslash\{0\}$, let $Q_a=Q({a}/{|a|},\sqrt{1-|a|})$ and
$$S_a=S(Q_{a})=\left\{z\in\BB:\frac{z}{|z|}\in Q_{a},|a|<|z|<1\right\}.$$
For convenience, if $a=0$, let $Q_a=\SS$ and $S_a=\BB$.  We call $S_a$ the Carleson block.  As usual, for a measurable set $E\subset\BB$, $\om(E)=\int_E \om(z)dV(z)$.

For any radial weight $\om$,  its associated weight $\om^*$  is defined by
\begin{align*}\om^*(z)=\int_{|z|}^1 \om(s)\log\frac{s}{|z|}sds, \,\,z\in\D\backslash\{0\}.\end{align*}

Now we state some lemmas which will be used in this paper.

\begin{Lemma}\label{0507-1}
Suppose $\om$ is a radial weight.
\begin{enumerate}[(i)]
  \item The following statements are equivalent.
  \begin{enumerate}[(a)]
    \item  $\om\in\hD$;
    \item  $\om^*(r)\approx (1-r)\int_r^1 \om(t)dt$ as $r\to 1$;
    \item for all $x\geq 1$, $\int_0^1 s^x\om(s)ds\approx \hat{\om}(1-\frac{1}{x})$;
    \item there is a constant $b>0$ such that $\frac{\hat{\om}(t)}{(1-t)^b}$ is essentially increasing.
  \end{enumerate}
  \item $\om\in\check{\dD}$ if and only if there is a constant $a>0$ such that $\frac{\hat{\om}(t)}{(1-t)^a}$ is essentially decreasing.
  \item If $\om$ is continuous, then $\om\in \R$ if and only if  there are  $-1<a<b<+\infty$ and $\delta\in [0,1)$, such that
  \begin{align}\label{0515-1}\frac{{\om}(t)}{(1-t)^b} \nearrow\infty,
\,\,\mbox{ and }\,\,\frac{{\om}(t)}{(1-t)^a}\searrow 0,\,\,\mbox{ when }\,\,\delta\leq t<1 .
\end{align}
\end{enumerate}
\end{Lemma}

\begin{proof}
By Lemmas A and B in \cite{PjaRjSk2018jga}, {\it (i)} and {\it (ii)} holds. By observation {\it (v)} of Lemma 1.1 in \cite{PjaRj2014book}, {\it (iii)} holds.
\end{proof}

\begin{Lemma}\label{1210-3}
Assume that $\om\in\hD$. Then  the following statements hold.
\begin{enumerate}[(i)]
    \item For any $\alpha>-2$, $(1-t)^\alpha\om^*(t)\in\R$;
\item  $\om(S_a)\approx (1-|a|)^n\int_{|a|}^1 \om(r)dr$;
\item $\hat{\om}(z)\approx\hat{\om}(a)$, if $1-|z|\approx 1-|a|$.
\end{enumerate}
\end{Lemma}

\begin{proof}
By Lemma 1.7 in \cite{PjaRj2014book}, {\it (i)} holds. By Lemma 2 in \cite{DjLsLxSy2019arxiv}, {\it (ii)} holds.
Using Lemma \ref{0507-1}, it is easy to check that {\it (iii)} holds.
\end{proof}

The following two lemmas are Lemmas 3 and 4 in \cite{DjLsLxSy2019arxiv2}, respectively.

\begin{Lemma}\label{0313-1} Suppose $\om\in\hD$. Then,
$$\|B_z^\om\|_{\B}\approx\frac{1}{\om(S_z)}\approx \|B_z^\om\|_{H^\infty},\,\,z\in\BB.$$
\end{Lemma}

\begin{Lemma}\label{0313-2}
Let $0<p<\infty$ and  $\om\in\hD$. Then the following assertions hold.
\begin{enumerate}[(i)]
  \item When $|rz|>\frac{1}{4}$, then
     $$M_p^p(r, B_z^\om) \approx \int_{0}^{r|z|} \frac{1}{\hat{\om}(t)^p (1-t)^{np-n+1}}dt,$$
     and
     $$M_p^p(r, \Re B_z^\om) \approx \int_{0}^{r|z|} \frac{1}{\hat{\om}(t)^p (1-t)^{(n+1)p-n+1}}dt.$$
  \item If $\upsilon\in\hD$, when $|z|>\frac{6}{7}$, then
  $$\|B_z^\om\|_{A_\upsilon^p}^p \approx   \int_{0}^{|z|} \frac{\hat{\upsilon}(t)}{\hat{\om}(t)^p (1-t)^{np-n+1}}dt,$$
  and
  $$  \|\Re B_z^\om\|_{A_\upsilon^p}^p\approx  \int_0^{|z|}  \frac{\hat{\upsilon}(t)}{\hat{\om}(t)^p (1-t)^{(n+1)p-n+1}}dt.$$
\end{enumerate}
\end{Lemma}

To study the compactness of a linear operator, we need the following  lemma which can be obtained in a standard way.
\begin{Lemma}\label{1210-2}
Suppose that $0<p, q<\infty, \om\in \hD$ and $\mu$ is a positive Borel measure on $\BB$. If $T:A_\om^p\to L_\mu^q$ is linear and bounded, then T is compact if and only if
whenever $\{f_k\}$  is bounded in $A_\om^p$ and $f_k\to 0$ uniformly on compact subsets of $\BB$,
$\lim\limits_{k\to\infty}\|Tf_k\|_{L_\mu^q} =0.$
\end{Lemma}

For a Banach space or a complete metric space $X$  and a positive Borel measure $\mu$ on $\BB$, $\mu$ is a  $q-$Carleson measure (vanish $q-$Carleson measure) for $X$ means that the identity operator $Id:X\to L_\mu^q$ is bounded (compact).
When $0<p\leq q<\infty$ and $\om\in\hD$, the characterizations of $q-$Carleson measure for $A_\om^p$ was obtained in \cite{DjLsLxSy2019arxiv}.
\msk

\noindent{\bf Theorem A. }{\it Let $0<p\leq q<\infty$, $\om\in \hD$ and $\mu$ be a positive Borel measure on $\D$. Then the following statements hold:
\begin{enumerate}[(i)]
  \item  $\mu$ is a $q$-Carleson measure for $A_\om^p$ if and only if
  \begin{align}\label{0109-8}
  \sup_{a\in\BB} \frac{\mu(S_a)}{(\om(S_a))^{\frac{q}{p}}}<\infty.
  \end{align}
  Moreover, if $\mu$  is a $q$-Carleson measure for $A_\om^p$, then
  $$\|Id\|_{A_\om^p\to L_\mu^q}^q \approx \sup_{a\in\BB} \frac{\mu(S_a)}{(\om(S_a))^{\frac{q}{p}}}.$$
  \item  $\mu$ is a vanish $q$-Carleson measure for $A_\om^p$ if and only if
  $$\lim_{|a|\to 1}   \frac{\mu(S_a)}{(\om(S_a))^{\frac{q}{p}}}=0.$$
\end{enumerate}
}

Recall that, for any $f\in L^1_\om$,  the Bergman projection $P_\om $ is defined by
\begin{align*}
P_\om f(z)=\int_\BB f(\xi)\ol{B_z^\om(\xi)}\om(\xi)dV(\xi),
\end{align*}
and the  maximal Bergman projection $P_\om^+$ is defined by $$P_\om^+(f)(z)=\int_\BB f(\xi)\left|B_z^\om(\xi)\right|\om(\xi)dV(\xi).$$

The following Theorems B and C are main results in \cite{DjLsLxSy2019arxiv2}.\msk

\noindent{\bf Theorem B. }{\it When  $\om\in\dD$,  $P_\om:L^\infty\to\B$ is bounded and  onto.}\msk

\noindent{\bf Theorem C. }{\it Suppose $1<p<\infty$ and $\om,\upsilon\in\dD$. Let $q=\frac{p}{p-1}$. Then the following statements are equivalent:
\begin{enumerate}[(i)]
  \item $P_\om^+: L_\upsilon^p\to  L_\upsilon^p$ is bounded;
  \item $P_\om: L_\upsilon^p\to  L_\upsilon^p$ is bounded;
  \item $M=\sup\limits_{0\leq r<1} \frac{\hat{\upsilon}(r)^{\frac{1}{p}}}{\hat{\om}(r)}
  \left(\int_r^1 \frac{\om(s)^q}{\upsilon(s)^{q-1}}s^{2n-1}ds\right)^\frac{1}{q}<\infty;$
  \item $N=\sup\limits_{0\leq r <1} \left(\int_0^r \frac{\upsilon(s)}{\hat{\om}(s)^p}s^{2n-1}ds+1\right)^\frac{1}{p}\left(\int_r^1 \frac{\om(s)^q}{\upsilon(s)^{q-1}}s^{2n-1}ds\right)^\frac{1}{q}<\infty.$
\end{enumerate}}
\msk

Let $P_z$ be the orthogonal projection of $\CC^n$ onto the one dimensional subspace $[z]=\{\lambda z:\lambda\in\CC\}$ generated by $z$, and $P_z^{\perp}$ be the orthogonal projection from $\CC^n$ onto $\CC^n\ominus[z]$. Thus $P_0(w)=0$,  $P_0^\perp(w)=w$ and
$$P_z(w)=\frac{\langle w,z\rangle}{|z|^2}z, ~~~~\,~~~\,~~P_z^\perp(w)=w-\frac{\langle w,z\rangle}{|z|^2}z,~~~~\mbox{~~~when~~}z\neq0.$$
For $z,w\in\BB$, the pseudo-hyperbolic distance between $z$ and $w$ is defined by
$$\rho(z,w)=\left|\frac{z-P_z(w)-\sqrt{1-|z|^2}P_z^\perp(w)}{1-\langle w,z\rangle}\right|.$$
The pseudo-hyperbolic ball at $z\in\BB$ with radius $r\in(0,1)$ is given by
 $$\Delta(z,r)=\{w\in\BB:\rho(z,w)<r\}.$$
Let $\beta(\cdot,\cdot)$ be the Bergman metric, that is $$\beta(z,w)=\frac{1}{2}\log\frac{1+\rho(z,w)}{1-\rho(z,w)}.$$
$D(z,r)$ means a Bergman metric ball at $z$ with radius $r>0$.
As we know, every Bergman metric ball is a pseudo-hyperbolic ball, and for all $a\in\BB$ and $z\in D(a,r)$, we  have $1-|z|\approx 1-|a|$.

The pseudo-hyperbolic balls and Bergman metric balls play very important roles in the theory of operators on the $A_\alpha^p$.  But when $\om\in\hD$, the roles of pseudo-hyperbolic balls and Bergman metric balls are substituted by Carleson block in the unit ball. If $\om\in\dD$, we   compare $\om(\Delta(a,r))$ with $\om(S_a)$
 as follows.

\begin{Proposition}\label{0421-1}
Let $0<r<1$ and $\om\in\dD$ such that $\frac{1}{K_\om}+\frac{2r}{1+r^2}>1$.
Then, for all $z\in\BB$  and $w\in\Delta(z,r)$,
\begin{align}\label{0508-1}
\om(\Delta(z,r))\approx \om(\Delta(w,r))\approx \om(S_z)\approx \om(S_w).
\end{align}
Moreover, if $\om\in\R$, for any fixed $r\in(0,1)$, (\ref{0508-1}) holds.
\end{Proposition}
\begin{proof}
 For any $z\neq 0$, $\Delta(z,r)$ is an ellipsoid consisting of all $w\in \BB$ such that
$$\frac{|P_z(w)-c|^2}{r^2t^2}+\frac{|P_z^\perp(w)|^2}{r^2t}<1,$$
where
$$ c=\frac{(1-r^2)z}{1-r^2|z|^2}~~,\mbox{~~~~}~~t=\frac{1-|z|^2}{1-r^2|z|^2}.$$
As $|z|\to 1$, we have $|c|\to 1$ and $t\to 0$.
Without loss of generality, we can assume $z=(|z|,0,0,\cdots,0)$.
Then $w=(w_1,w_2,\cdots,w_n)\in \Delta(z,r)$ if and only if
$$\frac{|w_1 -|c||^2}{r^2t^2}+\frac{|w|^2-|w_1|^2}{r^2t}<1.$$

Let  $\delta,k\in(0,1)$ and
\begin{align*}
E_z=&\left\{w\in\BB:|c|<|w|<|c|+\delta rt,\,\,\mbox{ and }\,\, \left|1-\langle \frac{w}{|w|},\frac{c}{|c|} \rangle\right|<k(1-|c|)\right\}\\
=&\bigg\{w\in\BB:|c|<|w|<|c|+\delta rt,\,\,\mbox{ and }\,\, \left||w|-w_1\right|<k|w|(1-|c|)\bigg\}.
\end{align*}
After a calculation, for all  $w\in E_z$,  there is a $C=C(r)>0$ such that
\begin{align*}
\frac{|w_1 -|c||^2}{r^2t^2}+\frac{|w|^2-|w_1|^2}{r^2t}  &\leq \frac{(|w_1 -|w|+|w|-|c||)^2}{r^2t^2}+\frac{2(|w|-|w_1|)}{r^2t}  \\
&\leq \delta^2+Ck.
\end{align*}
Let $\delta, k\in (0,1)$ such that
$$\frac{1}{K_\om}+\frac{2r\delta}{1+r^2}>1\,\,\mbox{ and }\,\,\delta^2+Ck<1.$$
Then we have $E_z\subset \Delta(z,r)$. Moreover, we can get a  $K>K_\om$ such that
$$\frac{\delta r(1+|z|)}{1+r^2|z|}>1-\frac{1}{K},~~~\,~~\mbox{as}~~ |z|\rightarrow 1. $$
  Then
\begin{align*}
\int_{|c|}^{|c|+\delta rt}\om(s)ds&=\int_{|c|}^{|c|+\frac{\delta r(1+|z|)}{1+r^2|z|}(1-|c|)} \om(s)ds \geq \int_{|c|}^{|c|+(1-\frac{1}{K})(1-|c|)} \om(s)ds \\
&=\int_{|c|}^{1-\frac{1-|c|}{K}}\om(s)ds\gtrsim  \hat\om(|c|)\approx \hat{\om}(|z|).
\end{align*}
Hence
\begin{align*}
\om(\Delta(z,r))\geq \om(E_z)\approx (1-|c|)^n\int_{|c|}^{|c|+\delta rt}\om(s)ds\gtrsim\om(S_z).
\end{align*}
By the proof of Lemma 8 in \cite{DjLsLxSy2019arxiv}, we have $\om(\Delta(z,r))\lesssim\om(S_z)$ as $|z|\rightarrow 1$.
For any fixed $\tau\in(0,1)$, when $|z|\leq \tau$, it is obvious that $\om(\Delta(z,r))\approx 1\approx \om(S_z)$.
By Lemma \ref{1210-3}, (\ref{0508-1}) holds.

If $\om\in\R$, using $K_\om=1$,  (\ref{0508-1}) holds for any fixed $r\in(0,1)$. The proof is complete.
\end{proof}

\section{Boundedness and compactness of $\T_\mu:A_\om^p\to A_\om^q$ with $\om\in\dD$}
In this section, we will discuss the boundedness and compactness of $\T_\mu:A_\om^p\to A_\om^q$ with $\om\in\dD$.

For a $f\in H(\BB)$, the Taylor series of $f$ at origin, which converges absolutely and uniformly on each compact subset of $\BB$,  is
$$f(z)=\sum_{m}\hat{f}_mz^m, \,\,z\in\BB.$$ Here the summation is over all multi-index $m=(m_1,m_2,\cdots,m_n),$
where each $m_k(k=1,2,\cdots,n)$ is a nonnegative integer and $z^m=z_1^{m_1}z_2^{m_2}\cdots z_{n}^{m_n}.$
Let
$$|m|=m_1+m_2+\cdots+m_n, \,\, m!=m_1!m_2!\cdots m_n!.$$

Suppose $\om\in\hD$. The space $A_\om^2$ is a Hilbert space with the inner product as follows.
$$\langle f,g \rangle_{A_\om^2}=\int_\BB f(z)\ol{g(z)}\om(z)dV(z), \,\,\mbox{ for all }\,\,f,g\in A_\om^2.$$

\begin{Lemma}\label{0413-1}
Suppose $\om\in\hD$. Then there exist constants $c=c(\om)>0$ and $\delta=\delta(\om)\in(0,1)$ such that
$$|B_a^\om(z)|\geq \frac{c}{\om(S_a)},\,\,z\in S_{a_\delta},\,\,a\in\BB\backslash\{0\},$$
where $a_\delta=(1-\delta(1-|a|))\frac{a}{|a|}$.
\end{Lemma}
\begin{proof}By Lemma \ref{0313-2}, when $1<p<\infty$, for all $|z|>\frac{6}{7}$,  we have
\begin{align*}
\|B_z^\om\|_{A_\om^p}^p &\lesssim \frac{1}{\hat{\om}(z)^{p-1}}\int_0^{|z|}\frac{dt}{(1-t)^{np-n+1}}
\lesssim \frac{1}{\hat{\om}(z)^{p-1} (1-|z|)^{np-n}},
\end{align*}
and
\begin{align*}
\|B_z^\om\|_{A_\om^p}^p
&\gtrsim  \int_{2|z|-1}^{|z|}\frac{dt}{\hat{\om}(t)^{p-1}(1-t)^{np-n+1}}
\approx \frac{1}{\hat{\om}(z)^{p-1} (1-|z|)^{np-n}}.
\end{align*}
So, when $1<p<\infty$, by Lemma \ref{1210-3}, we have
\begin{align*}
\|B_z^\om\|_{A_\om^p} \approx \frac{1}{\hat{\om}(z)^{1-\frac{1}{p}} (1-|z|)^{n-\frac{n}{p}}}
\approx \frac{1}{\om(S_z)^{1-\frac{1}{p}}},\,\,\mbox{ for all }\,\, |z|>\frac{6}{7}.
\end{align*}
When $0<|z|\leq \frac{6}{7}$, there exist $r_0\in(0,1)$ and $\varepsilon>0$ such that
$$|B_z^\om(w)|\geq \varepsilon,\,\,\mbox{ for all }\,\,|w|\leq r_0.$$
Therefore, when $0<|z|\leq \frac{6}{7}$, we obtain
\begin{align*}
\|B_z^\om\|_{A_\om^p}^p \geq \int_{|w|\leq r_0}|B_z^\om(w)|^p\om(w)dV(w)\geq C(\varepsilon,r_0,\om)
\approx \frac{1}{\om(S_z)^{p-1}},
\end{align*}
and
\begin{align*}
\|B_z^\om\|_{A_\om^p}^p  &= 2n \int_0^1 r^{2n-1}\om(r) M_p^p (r,B_z^\om)dr = 2n \int_0^1 r^{2n-1}\om(r) M_p^p (\frac{7|z|r}{6},B_{\frac{6z}{7|z|}}^\om)dr \\
&\leq 2n \int_0^1 r^{2n-1}\om(r) M_p^p (r,B_{\frac{6z}{7|z|}}^\om)dr  =\|B_{\frac{6z}{7|z|}}^\om\|_{A_\om^p}^p\approx 1\approx \frac{1}{\om(S_z)^{p-1}}.
\end{align*}
The case of $z=0$ is trivial. So, we have
\begin{align}\label{0412-1}
\|B_z^\om\|_{A_\om^p} \approx \frac{1}{\hat{\om}(z)^{1-\frac{1}{p}} (1-|z|)^{n-\frac{n}{p}}}
\approx \frac{1}{\om(S_z)^{1-\frac{1}{p}}},\,\,\mbox{ for all }\,\, z\in\BB .
\end{align}
In Particular, when $p=2$, we have $\|B_a^\om\|_{A_\om^2}^2 \approx \frac{1}{\om(S_a)}$. So, there exists a constant $C_1=C_1(\om)>0$, such that
$\|B_a^\om\|_{A_\om^2}^2 \geq  \frac{C_1}{\om(S_a)}$ for all $a\in \BB\backslash\{0\}$.
For any fixed $\delta\in (0,1]$, let $$t=t_{\delta,a}=\sqrt{\frac{1-\delta(1-|a|)}{|a|}}.$$
Then for all $z\in\BB$,  we have
\begin{align*}
|B_a^\om(z)|
&\geq |B_a^\om(a_\delta)|-|B_a^\om(a_\delta)-B_a^\om(z)|=|B_{ta}^\om(ta)|-|B_a^\om(a_\delta)-B_a^\om(z)|   \\
&=\|B_{ta}^\om\|_{A_\om^2}^2-|B_a^\om(a_\delta)-B_a^\om(z)|\geq  \frac{C_1}{\om(S_{ta})}-|B_a^\om(a_\delta)-B_a^\om(z)|  \\
& \geq  \frac{C_1}{\om(S_{a})}-|B_a^\om(a_\delta)-B_a^\om(z)|.
\end{align*}

Let $I=[\langle a_\delta,a\rangle, \langle z,a \rangle]$ be the line segment in $\D$. For all $\eta\in I$, we have $|\eta|\leq |a|$.
Let $|I|$ be the length of $I$. If $z\in S_{a_\delta}$, we have
\begin{align*}
|I|&=|\langle a_\delta,a\rangle- \langle z,a \rangle| =||a_\delta||a|-\langle z,a\rangle|  \\
&\leq |a_\delta|\left|1-\langle \frac{z}{|z|}, \frac{a}{|a|}\rangle+\langle \frac{z}{|z|}, \frac{a}{|a|}\rangle -\langle \frac{z}{|a_\delta|}, \frac{a}{|a|}\rangle \right|  \\
&\leq  |a_\delta|\left( 1-|a_\delta|+\frac{|z|-|a_\delta|}{|z||a||a_\delta|}|\langle z,a\rangle|\right)\\
&\leq 2\delta(1-|a|).
\end{align*}
By the proof of Lemma \ref{0313-1}, there exists $C=C(\om)$ such that
\begin{align}|B_a^\om(a_\delta)-B_a^\om(z)|
&=\frac{1}{2n!}\left|\sum_{k=0}^\infty \frac{(n-1+k)!}{k!\om_{2n+2k-1}} \left( \langle a_\delta,a\rangle^k  - \langle z,a\rangle^k  \right) \right|  \nonumber\\
&=\frac{1}{2n!}\left|\sum_{k=1}^\infty \frac{(n-1+k)!}{(k-1)!\om_{2n+2k-1}} \int_I \eta^{k-1}d\eta \right|  \nonumber\\
&\leq \frac{2\delta(1-|a|)}{2n!|a|}\sum_{k=1}^\infty \frac{(n-1+k)!}{(k-1)!\om_{2n+2k-1}}|a|^k  \label{0413-2}\\
&\leq \frac{C\delta}{|a|\om(S_{\sqrt{|a|}})}.  \nonumber
\end{align}
Here, $S_{\sqrt{|a|}}$ means some Carleson block $S_\eta$ with $|\eta|=\sqrt{|a|}$.
Since  $1-|a|\approx 1-\sqrt{|a|}$, by Lemma \ref{1210-3}, there exists $C=C(\om)$ such that
$$|B_a^\om(a_\delta)-B_a^\om(z)| \leq  \frac{C\delta}{\om(S_{a})}, \,\,\mbox{ when }\,\,|a|>\frac{1}{2}.$$
When $0<|a|\leq \frac{1}{2}$, by (\ref{0413-2}), there exists $C=C(\om)$ such that
\begin{align*}|B_a^\om(a_\delta)-B_a^\om(z)|
&\leq \frac{4\delta(1-\frac{1}{2})}{2n! (\frac{1}{2})}\sum_{k=1}^\infty \frac{(n-1+k)!}{(k-1)!\om_{2n+2k-1}}(\frac{1}{2})^k\leq   \frac{C\delta}{\om(S_a)}.
\end{align*}
Therefore,
\begin{align*}
|B_a^\om(z)| \geq  \frac{C_1}{\om(S_{a})}- \frac{\delta C}{\om(S_{a})}, \,\,\mbox{ when }\,\,a\in\BB\backslash\{0\}.
\end{align*}
We get the desired result by choosing a $\delta$ small enough. The proof is complete.
\end{proof}

\begin{Lemma}\label{0422-1}
Let $\om\in\hD$. Then there exists  $r=r(\om)>0$ such that $|B_z^\om(a)|\approx B_a^\om(a)$ for all $a\in\BB$ and $z\in D(a,r)$.
\end{Lemma}
\begin{proof} For any fixed $r>0$, by Cauchy-Schwarz's inequality, (\ref{0412-1}) and Lemma \ref{1210-3}, we have
\begin{align}
|B_a^\om(z)|
&=\left|\langle B_a^\om,B_z^\om \rangle_{A_\om^2}\right|\leq \|B_a^\om\|_{A_\om^2}\|B_z^\om\|_{A_\om^2}     \nonumber \\
&\approx \left(\frac{1}{\om(S_z)}\frac{1}{\om(S_a)}\right)^\frac{1}{2}
\approx \frac{1}{\om(S_a)} \approx \|B_a^\om\|_{A_\om^2}^2=B_a^\om(a).  \label{0422-3}
\end{align}

Let $z\in D(a,r)$ and $I=[\langle a,a\rangle, \langle z,a \rangle]$ be the line segment in $\D$. We claim that  there exists a constant $\delta(r)$ such that
\begin{align}\label{0829-4}
|I|\leq \delta(r)(1-|a|)\,\mbox{ and  }\,\lim\limits_{r\to 0}\delta(r)=0.
\end{align}
Taking this for granted for a moment.
By the proof of Lemma \ref{0413-1}, there exists $C_3=C_3(\om)$ such that
$$|B_a^\om(a)-B_a^\om(z)|\leq \frac{C_3\delta(r)}{\om(S_a)}\,\,\mbox{for all}\,\,z\in D(a,r).$$
Therefore, by (\ref{0422-3}), there exists $C_4=C_4(\om)$ such that
\begin{align*}
|B_a^\om(z)|  &\geq |B_a^\om(a)|-|B_a^\om(a)-B_a^\om(z)|\geq  \frac{C_4}{\om(S_{a})}-\frac{C_3\delta(r)}{\om(S_a)}.
\end{align*}
We get the desired result by  choosing   $r$ small enough.

Now we only need to prove that (\ref{0829-4}) holds. Let $\tanh r=\frac{e^{2r}-1}{e^{2r}+1}$. Without loss of generality, suppose $a=(|a|,0,0,\cdots,0)$.
Then, $z\in D(a,r)$ if and only if
$$\frac{|z_1-c_1|^2}{(\tanh r)^2t^2}+\frac{|z_2|^2+|z_3|^2+\cdots+|z_n|^2}{(\tanh r)^2t}<1,$$
where
$$ c=\frac{(1-(\tanh r)^2)a}{1-(\tanh r)^2|a|^2}~~,\mbox{~~~~}~~t=\frac{1-|a|^2}{1-(\tanh r)^2|a|^2}.$$
Therefore, using $a=(|a|,0,0,\cdots,0)$, we have
\begin{align*}
|I|&=|\langle a,a\rangle -\langle z,a \rangle|\leq |z_1-a_1|\leq |z_1-c_1|+|c_1-a_1|  \\
&\leq 2t\tanh r  \leq \frac{4\tanh r}{1-(\tanh r)^2}(1-|a|),
\end{align*}
which implies the desired result. The proof is complete.
\end{proof}

\begin{Lemma}\label{0413-7}
Suppose $1<p<\infty$ and $\om\in\dD$. Let $q=\frac{p}{p-1}$. Then $(A_\om^p)^*\simeq A_\om^q$ with equivalent norms, under the pairing
$$\langle f,g \rangle_{A_\om^2}=\int_\BB f(z)\ol{g(z)}\om(z)dV(z).$$
\end{Lemma}
\begin{proof}

Suppose $T\in (A_\om^p)^*$. By Hahn-Banach's Theorem, we can extend $T$ as a linear functional on $L_\om^p$ without increasing the norm of $T$.
Thus, there exists $h\in L_\om^q$, such that for all $f\in L_\om^p$,   $Tf=\langle f,h \rangle_{L_\om^2}$.
For all $f\in A_\om^p$, we have
\begin{align*}
Tf=\langle f,h \rangle_{L_\om^2}= \langle P_\om f,h \rangle_{L_\om^2}=\langle f, P_\om h \rangle_{A_\om^2}.
\end{align*}
By Theorem C, we have $g=P_\om h\in A_\om^q$ and $\|g\|_{A_\om^q}\leq \|P_\om\|_{L_\om^q\to A_\om^q}\|h\|_{L_\om^q}$. Therefore,
$$\|T\|_{A_\om^p\to\CC}\leq \|g\|_{A_\om^q},$$
and
$$\|T\|_{A_\om^p\to\CC}=\|T\|_{L_\om^p\to \CC}\approx \|h\|_{L_\om^q}\gtrsim \|g\|_{A_\om^q}.$$
Thus, $\|T\|_{A_\om^p\to\CC}\approx \|g\|_{A_\om^q}.$

For any fixed $g_i \in A_\om^q(i=1,2)$, let
$$T_i f=\langle f,g_i\rangle_{A_\om^2},\,\,\mbox{ for all }\,\,f\in A_\om^p.$$
Then $T_i:A_\om^p\to \CC$ is bounded. If $g_1\neq g_2$,  let $g_1-g_2=\sum_m a_mz^m$. Then there exists a $k$ such that $a_k\neq 0$. So we have
$(T_1-T_2)(z^k)\neq 0$ and hence $T_1\neq T_2$. The proof is complete.
\end{proof}

\begin{Lemma}\label{0413-6} Suppose $\om\in\hD$  and $\mu$ is a finite positive Borel measure on $\BB$.
If $f\in L_\mu^1$ and $\sum\limits_m |\hat{g}_m|<\infty$, then
$$\langle \T_\mu f,g\rangle_{A_\om^2}=\langle f,g \rangle_{L_\mu^2}.$$
\end{Lemma}

\begin{proof}Let  $0<r<1$. By  Lemmas 1.8 and 1.11  in \cite{Zk2005} and
\begin{align*}
B_z^\om(w) =\frac{1}{2 n!}\sum_{m} \frac{(n-1+|m|)!}{m!  \om_{2n+2|m|-1}}\ol{z}^mw^m,
\end{align*}
we have
\begin{align*}
\left|\int_{|z|<r} B_z^\om(w)g(z) \om(z)dV(z)\right|
=&\frac{1}{2 n!}\left|\sum_{m} \frac{(n-1+|m|)!}{m!  \om_{2n+2|m|-1}}\hat{g}_m w^m\int_{|z|<r} z^m\ol{z}^m\om(z)dV(z)   \right|\\
\leq &\frac{1}{2 n!}\sum_{m} \frac{(n-1+|m|)!}{m!  \om_{2n+2|m|-1}} |\hat{g}_m| \int_{\BB} z^m\ol{z}^m\om(z)dV(z) \\
=&\sum_m |\hat{g}_m|<\infty.
\end{align*}
Then, Fubini's theorem and the dominated convergence theorem yield
\begin{align*}
\langle \T_\mu f,g\rangle_{A_\om^2} &=\lim_{s\to 1} \int_{|z|<s}\left(\int_\BB f(w)\ol{B_z^\om(w)}d\mu(w)\right)\ol{g(z)}\om(z)dV(z)  \\
&=\lim_{s\to 1} \int_{\BB} f(w)\ol{\int_{|z|<s} B_z^\om(w)g(z) \om(z)dV(z)}  d\mu(w)\\
&=\int_{\BB} f(w)\ol{\int_{\BB} B_z^\om(w)g(z) \om(z)dV(z)}  d\mu(w)\\
&=\langle f,g\rangle_{L_\mu^2}.
\end{align*}
The proof is complete.
\end{proof}

\begin{Theorem}\label{0414-2}
Let $1<p\leq q<\infty$, $\om\in\dD$ and $\mu$ be a positive Borel measure on $\BB$. Then the following statements are equivalent.
\begin{enumerate}[(i)]
  \item $\T_\mu:A_\om^p\to A_\om^q$ is bounded;
  \item $\frac{\widetilde{\T_\mu}(z)}{\om(S_z)^{\frac{1}{p}-\frac{1}{q}}}\in L^\infty$;
  \item $\mu$ is a $s\left(\frac{1}{p}-\frac{1}{q}+1\right)$-Carleson measure for $A_\om^s$, for some (equivalently for all ) $0<s <\infty$;
  \item $\frac{\mu(S_z)}{\om(S_z)^{\frac{1}{p}-\frac{1}{q}+1}}\in L^\infty$.
\end{enumerate}
\end{Theorem}

\begin{proof} Since $1<p\leq q<\infty$, $\frac{1}{p}-\frac{1}{q}\geq 0$. By Theorem A, we have {\it (iii)}$\Leftrightarrow${\it (iv)} and
$$\|Id\|_{A_\om^s\to L_\mu^{s\left(\frac{1}{p}-\frac{1}{q}+1\right)}}^{s\left(\frac{1}{p}-\frac{1}{q}+1\right)}
\approx \left\|\frac{\mu(S_z)}{\om(S_z)^{\frac{1}{p}-\frac{1}{q}+1}}\right\|_{ L^\infty}.$$

{\it (i)}$\Rightarrow${\it (ii)}. Suppose that {\it (i)} holds. Let $q^\p=\frac{q}{q-1}$.  Since $\T_\mu:A_\om^p\to A_\om^q$ is bounded,   by H\"older's inequality and (\ref{0412-1}),
\begin{align*}
|\widetilde{\T_\mu}(z)| &=\frac{|\langle \T_\mu B_z^\om, B_z^\om \rangle_{A_\om^2}|}{\|B_z^\om\|_{A_\om^2}^2}
 \leq \frac{ \|\T_\mu\|_{A_\om^p\to A_\om^q} \|B_z^\om\|_{A_\om^p}  \|B_z^\om\|_{A_\om^{q^\p}} }{\|B_z^\om\|_{A_\om^2}^2}
 \approx \frac{\|\T_\mu\|_{A_\om^p\to A_\om^q}}{\om(S_z)^{\frac{1}{q}-\frac{1}{p}}},
\end{align*}
which implies that {\it (ii)} holds.

{\it (ii)}$\Rightarrow${\it (iv)}.  Suppose that {\it (ii)} holds, that is, $\frac{\widetilde{\T_\mu}(z)}{\om(S_z)^{\frac{1}{p}-\frac{1}{q}}}\in L^\infty$.
Let $\delta$ and $c$ be those in Lemma \ref{0413-1}. When $z\in\BB\backslash\{0\}$, by (\ref{0412-1}), Lemmas \ref{0413-1} and \ref{0413-6}, we have
\begin{align}
\frac{\widetilde{\T_\mu}(z)}{\om(S_z)} \approx \|B_z^\om\|_{A_\om^2}^2\widetilde{\T_\mu}(z)
&=\langle \T_\mu B_z^\om, B_z^\om \rangle_{A_\om^2}=\|B_z^\om\|_{L_\mu^2}^2  \label{0413-3}\\
&\geq \int_{S_{z_\delta}} |B_z^\om(\zeta)|^2d\mu(\zeta)\geq c^2\frac{\mu(S_{z_\delta})}{(\om(S_z))^2}.  \nonumber
\end{align}
Here $z_\delta=(1-\delta(1-|a|))\frac{a}{|a|}$. Since $1-|z_\delta|=\delta(1-|z|)$, by Lemma \ref{1210-3} we obtain
\begin{align}\label{0414-1}
\frac{\mu(S_{z_\delta})}{\om(S_{z_\delta})^{\frac{1}{p}-\frac{1}{q}+1}}
\lesssim \frac{\widetilde{\T_\mu}(z)\om(S_z)}{\om(S_{z_\delta})^{\frac{1}{p}-\frac{1}{q}+1}}
\approx \frac{\widetilde{\T_\mu}(z)}{\om(S_{z})^{\frac{1}{p}-\frac{1}{q}}}.
\end{align}
Then we get that {\it (iv)} holds.

{\it (iii)}$\Rightarrow${\it (i)}. Suppose {\it (iii)} holds. Then $\mu$ is a $1$-Carleson measure for $A_\om^{\frac{pq}{q-p+pq}}$.
For any $f\in A_\om^p$, we have $f\in A_\om^\frac{pq}{q-p+pq}$ and  $f\in L_\mu^1$.
If $g$ is a  polynomial, by Lemma \ref{0413-6}, Theorem A and H\"older's inequality, we have
\begin{align*}
|\langle \T_\mu f,g \rangle_{A_\om^2}|  &\leq \int_\BB |f(z)g(z)|d\mu(z)        \leq \|Id\|_{A_\om^{\frac{pq}{q-p+pq}}  \to L_\mu^1}\|fg\|_{A_\om^{\frac{pq}{q-p+pq}}}   \\
&\leq \|Id\|_{A_\om^{\frac{pq}{q-p+pq}}\to L_\mu^1} \|f\|_{A_\om^p} \|g\|_{A_\om^\frac{q}{q-1}}.
\end{align*}
Since polynomials are dense in  $A_\om^\frac{q}{q-1}$, by Lemma \ref{0413-7}, $\T_\mu: A_\om^p\to A_\om^q$ is bounded. The proof is complete.
\end{proof}

\begin{Theorem} Let $1<p\leq q<\infty$, $\om\in\dD$ and $\mu$ be a positive Borel measure on $\BB$. Then the following statements are equivalent.
\begin{enumerate}[(i)]
  \item $\T_\mu:A_\om^p\to A_\om^q$ is compact;
  \item $\lim\limits_{|z|\to 1}\frac{\widetilde{\T_\mu}(z)}{\om(S_z)^{\frac{1}{p}-\frac{1}{q}}}=0$;
  \item $\mu$ is a vanish $s\left(\frac{1}{p}-\frac{1}{q}+1\right)$-Carleson measure for $A_\om^s$, for some (equivalently for all ) $0<s <\infty$;
  \item $\lim\limits_{|z|\to 1}\frac{\mu(S_z)}{\om(S_z)^{\frac{1}{p}-\frac{1}{q}+1}}=0$.
\end{enumerate}
\end{Theorem}
\begin{proof}
By Theorem A, we have {\it (iii)}$\Leftrightarrow${\it (iv)}. By (\ref{0414-1}), we have {\it (ii)}$\Rightarrow${\it (iv)}.

{\it (i)}$\Rightarrow${\it (ii)}. Suppose that {\it (i)} holds.  Let $b_z^{\om,p}(w)=\frac{B_z^\om(w)}{\|B_z^\om\|_{A_\om^p}}$. When $|w|\leq r<1$, by (\ref{0412-1}) and Lemma \ref{0313-1}, we have
$$|b_z^{\om,p}(w)|\lesssim \om(S_z)^{1-\frac{1}{p}}\|B_{rz}^\om\|_{H^\infty}\approx \frac{\om(S_z)^{1-\frac{1}{p}}}{\om(S_{rz})}.$$
Then $\{b_z^{\om,p}\}$ is bounded in $A_\om^p$ and converges to 0 uniformly on compact subsets of $\BB$ as $|z|\rightarrow1$.
By Lemma \ref{1210-2}, we have $\lim\limits_{|z|\to 1} \|\T_\mu b_z^{\om,p}\|_{A_\om^q}=0$.
By (\ref{0412-1}) and H\"older's inequality, we have
\begin{align*}
\frac{\left|\widetilde{\T_\mu}(z)\right|}{\om(S_z)^{\frac{1}{p}-\frac{1}{q}}}
=\frac{|\langle \T_\mu B_z^\om, B_z^\om \rangle_{A_\om^2}|}{\om(S_z)^{\frac{1}{p}-\frac{1}{q}}\|B_z^\om\|_{A_\om^2}^2}
&\approx \frac{\left|\langle \T_\mu b_z^{\om,p}, B_z^{\om}  \rangle_{A_\om^2} \right|}{\om(S_z)^{-\frac{1}{q}}}  \\
&\leq \frac{\|\T_\mu b_z^{\om,p}\|_{A_\om^q}  \|B_z^\om\|_{A_\om^\frac{q}{q-1}}}{\om(S_z)^{-\frac{1}{q}}}
\approx \|\T_\mu b_z^{\om,p}\|_{A_\om^q},
\end{align*}
which implies that {\it (ii)} holds.

{\it (iii)} $\Rightarrow$ {\it (i)}.  Suppose that {\it (iii)} holds.
Let  $\{f_k\}_{k=1}^\infty$ be bounded in $A_\om^p$ and converge to 0 uniformly on compact subsets of $\BB$.
If $0<r<1$, let $d\mu_r=\chi_{r\leq |z|<1}d\mu$. By (12) in \cite{DjLsLxSy2019arxiv}, we have
 $$\lim_{r\to 1}\sup_{z\in\BB}\frac{\mu_r(S_z)}{\om(S_z)^{\frac{1}{p}-\frac{1}{q}+1}}=0.$$
 By the proof of {\it (iii)}$\Rightarrow${\it (i)} of Theorem \ref{0414-2}, we have
$\lim\limits_{r\to 1}\|\T_{\mu_r}\|_{A_\om^p\to A_\om^q}=0.$  For any fixed $0<r<1$, by Lemma \ref{0313-1} we have
 \begin{align*}
 \left|\T_{\mu-\mu_r}f_k (z)\right| &=\left| \int_{r\BB} f_k(\xi)\ol{B_z^\om(\xi)}d\mu(\xi) \right|  \\
 &\leq \mu(\BB) \|B_{rz}\|_{H^\infty}\sup_{|\xi|<r}|f_k(\xi)| \lesssim \frac{\mu(\BB)}{\om(S_{rz})}\sup_{|\xi|<r}|f_k(\xi)|.
 \end{align*}
Therefore, $\lim\limits_{k\to\infty}\|\T_{\mu-\mu_r} f_k\|_{A_\om^q}=0$ for any fixed $0<r<1$.

For any given $\varepsilon$, we can choose a fixed $r_\varepsilon>0$ such that $\|\T_{\mu_{r_\varepsilon}}\|_{A_\om^p\to A_\om^q}<\varepsilon$. Then we have
\begin{align*}
\lim_{k\to\infty}\|\T_\mu f_k\|_{A_\om^q}
&\leq \lim_{k\to\infty}\|\T_{\mu_{r_\varepsilon}}f_k  \|_{A_\om^q}   +\lim_{k\to\infty}\|\T_{\mu-\mu_{r_\varepsilon}} f_k\|_{A_\om^q}
\leq \varepsilon \sup_{k>0}\|f_k\|_{A_\om^p}.
\end{align*}
Since $\varepsilon$ is arbitrary, by Lemma \ref{1210-2}, $\T_\mu:A_\om^p\to A_\om^q$ is compact. The proof is complete.
\end{proof}

Recall that $\beta(\cdot,\cdot)$ is the Bergman metric and  $D(a,r)$  is the  Bergman metric ball at $a$ with radius $r>0$.
Then for all $a\in\BB$ and $z\in D(a,r)$, we  have $1-|z|\approx 1-|a|$.
Suppose  $\{a_k\}$ is a sequence in  $\BB$. It is  $\gamma-$separated if $\inf\limits_{k\neq j}\beta(a_k,a_j)>\gamma>0$.
 It is  a $\delta$-lattice if it is $\frac{\delta}{5}$-separated and   $\BB=\cup_{k=1}^\infty  D(a_k,5\delta)$.

Define the Rademacher function $r_j$ by $$r_j(t)=\textrm{sgn}(\sin(2^j\pi t)),\,\,j=1,2,\cdots.$$ The Khintchine's inequality is the following.\msk

\noindent{\bf  Khintchine's inequality.}  {\it For $0<p<\infty$, there exist constants $0<A_p\leq B_p<\infty$ such that, for all natural numbers $k$ and all complex numbers $c_1, c_2, \cdots , c_k,$ we have
$$A_p\left(\sum_{j=1}^k |c_j|^2\right)^{\frac{p}{2}}\leq\int_0^1\left|\sum_{j=1}^k c_jr_j(t)\right|^pdt\leq B_p\left(\sum_{j=1}^k |c_j|^2\right)^{\frac{p}{2}}.$$}

\begin{Lemma}\label{0424-1} Suppose $0<q<p<\infty$, $0<r<\infty$, $\om\in\R$ and $\mu$ is a positive Borel measure on $\BB$.
Let  $\widehat{\mu_r}(z)=\frac{\mu(D(z,r))}{\om(S_z)}$. Then $\mu$ is a $q-$Carleson measure for $A_\om^p$ if and only if
$$\int_\BB \widehat{\mu_r}(z)^\frac{p}{p-q}\om(z)dV(z)<\infty.$$
Moreover,
$$\|Id\|_{A_\om^p\to L_\mu^q}^q \approx \|\widehat{\mu_r}\|_{L_\om^s},\,\,\mbox{ where }\,\,\frac{1}{s}+\frac{q}{p}=1.$$
\end{Lemma}

\begin{proof}
Theorem 4 in \cite{hu3} shows that $\mu$ is a $q-$Carleson measure for $A_\om^p$ if and only if $ \widehat{\mu_1}(z)\in L_\om^\frac{p}{p-q}$, and $\|\widehat{\mu_1}\|_{L_\om^{\frac{p}{p-q}}}\lesssim \|Id\|_{A_\om^p\to L_\mu^q}^q.$
An analogue of (2.3) in \cite{Ld1986pems} shows that $\|\widehat{\mu_1}\|_{L_\om^{\frac{p}{p-q}}}\gtrsim \|Id\|_{A_\om^p\to L_\mu^q}^q.$
So, we only need to prove that, for $0<s<t<\infty$,
$\|\widehat{\mu_s}\|_{L_\om^\frac{p}{p-q}}\approx \|\widehat{\mu_t}\|_{L_\om^\frac{p}{p-q}}$.

It is obvious that $\|\widehat{\mu_s}\|_{L_\om^\frac{p}{p-q}}\leq \|\widehat{\mu_t}\|_{L_\om^\frac{p}{p-q}}$.
For any $f\in L_\om^1$, define
$$\mathcal{E} f(z)=\frac{\int_{D(z,t)}f(\xi)\om(\xi)dV(\xi)}{\om(D(z,t))}.$$
By Fubini's theorem and Proposition \ref{0421-1}, $\mathcal{E} :L_\om^1 \to L_\om^1$ is bounded. Obviously, $\mathcal{E} :L_\om^\infty \to L_\om^\infty$ is bounded.
By Theorems 1.32 and 1.33 in \cite{Zk2005}, $\mathcal{E} :L_\om^\frac{p}{p-q} \to L_\om^\frac{p}{p-q}$ is bounded.

There exists $\varepsilon=\varepsilon(s,t)$ such that, for all $z\in\BB$ and $\eta\in D(z,t)$,
$$D(\tau,\varepsilon)\subset D(z,t) \cap D(\eta,s)$$
for some  $\tau\in D(z,t)$. Using Proposition \ref{0421-1}, we have
\begin{align*}
\mathcal{E} \widehat{\mu_s}(z)&=\frac{1}{\om(D(z,t))}\int_{D(z,t)}\frac{\int_{D(\xi,s)}d\mu(\eta)}{\om(S_\xi)}  \om(\xi)dV(\xi)  \\
&\approx\frac{1}{\om(D(z,t))}\int_{\BB}\int_\BB\frac{\chi_{D(z,t)}(\xi)\chi_{D(\xi,s)}(\eta) \om(\xi)}{\om(D(\xi,s))}  dV(\xi)d\mu(\eta)\\
&\approx \frac{1}{\om(D(z,t))}\int_\BB\int_{D(z,t) \cap D(\eta,s)}\frac{ \om(\xi)}{\om(D(\eta,s))}  dV(\xi)d\mu(\eta)\\
&\geq  \frac{1}{\om(D(z,t))}\int_{D(z,t)}\int_{D(z,t) \cap D(\eta,s)}\frac{ \om(\xi)}{\om(D(\eta,s))}  dV(\xi)d\mu(\eta)\\
&\approx \frac{1}{(\om(D(z,t)))^2}\int_{D(z,t)}\int_{D(z,t) \cap D(\eta,s)}\om(\xi)  dV(\xi)d\mu(\eta)\\
&\gtrsim \frac{\mu(D(z,t))}{\om(D(z,t))}\approx \widehat{\mu_t}(z).
\end{align*}
Therefore,
$$\|\widehat{\mu_t}\|_{L_\om^\frac{p}{p-q}}\lesssim \|\mathcal{E} \widehat{\mu_s}\|_{L_\om^\frac{p}{p-q}}\leq \|\mathcal{E}  \|_{{L_\om^\frac{p}{p-q}}\to {L_\om^\frac{p}{p-q}}}
 \|\widehat{\mu_s}\|_{L_\om^\frac{p}{p-q}}.$$
 The proof is complete.
\end{proof}

\begin{Theorem}
Suppose $1<q<p<\infty$,  $\om\in\dD$, $r>0$ and $\mu$ be a positive Borel measure on $\BB$. Let $W_1(t)=\frac{\hat{\om}(t)}{1-t}$.
Then the following statements are equivalent.
\begin{enumerate}[(i)]
  \item $\T_\mu:A_\om^p\to A_\om^q$ is compact;
  \item $\T_\mu:A_\om^p\to A_\om^q$ is bounded;
  \item $\widehat{\mu_r}(z)=\frac{\mu(D(z,r))}{\om(S_z)}\in L_{W_1}^\frac{pq}{p-q}$;
  \item $Id:A_\om^p\to L_\mu^\frac{pq+q-p}{q}$ is bounded;
  \item $Id:A_\om^p\to L_\mu^\frac{pq+q-p}{q}$ is compact.
\end{enumerate}
Moreover,
\begin{align}\label{0429-1}
\|\T_\mu\|_{A_\om^p\to A_\om^q}\approx \|\widehat{\mu_r}\|_{L_{W_1}^{\frac{pq}{p-q}}}
\approx \|Id\|_{A_\om^p\to L_\mu^\frac{pq+q-p}{q}}^\frac{pq+q-p}{q}.
\end{align}
If $\om\in \R$, then  $\|\T_\mu\|_{A_\om^p\to A_\om^q}\approx \|\widetilde{T_\mu}\|_{L_{\om}^\frac{pq}{p-q}}. $
\end{Theorem}
\begin{proof} {\it (i)}$\Rightarrow${\it (ii)} and {\it (v)}$\Rightarrow${\it (iv)} are obvious.
In the next proof, let $p^\p$ and $q^\p$ be the conjugate of $p$ and $q$, respectively, that is, $\frac{1}{p}+\frac{1}{p^\p}=1$ and $\frac{1}{q}+\frac{1}{q^\p}=1$.

{\it (iii)}$\Leftrightarrow${\it (iv)}.  Since ${W_1}(t)=\frac{\hat{\om}(t)}{1-t}$, by Proposition 5 in \cite{PjaRjSk2018jga}, we have ${W_1}\in\R$ and $\hat{{W_1}}=\hat{\om}$.
Using Theorem A and Lemma \ref{1210-3}, we obtain $\|\cdot\|_{A_\om^{p}}\approx \|\cdot\|_{A_{W_1}^{p}}$ for any $p>0$.
By Proposition \ref{0421-1} and Lemma \ref{1210-3}, we have
\begin{align}\label{0422-2}
{W_1}(D(z,r))\approx (1-|z|)^{n+1}{W_1}(z)\approx {W_1}(S_z)\approx (1-|z|)^n \hat{{W_1}}(z)\approx \om(S_z).
\end{align}
Therefore, by Lemma \ref{0424-1}, we see that  {\it (iii)}$\Leftrightarrow${\it (iv)} holds and
$$\|Id\|_{A_\om^p\to L_\mu^\frac{pq+q-p}{q}}^\frac{pq+q-p}{q}\approx
\|Id\|_{A_{W_1}^p\to L_\mu^\frac{pq+q-p}{q}}^\frac{pq+q-p}{q}\approx
 \|\widehat{\mu_r}\|_{L_{W_1}^{\frac{pq}{p-q}}}.$$

 {\it (ii)}$\Rightarrow${\it (iii)}. Let $r_0=r_0(\om)$ such that  Lemma \ref{0422-1} holds.  Suppose that $\{a_k\}_{k=1}^\infty$ is a $\delta-$lattice with $5\delta\leq r_0$.
By Lemmas 1.23 and  2.20 in \cite{Zk2005}, if $s>0$ is fixed,  for all $\eta\in D(z,s)$ and $z\in \BB$, we have
\begin{align}\label{0423-1}
1-|\eta|\approx 1-|z|  ~~~~\,\,\mbox{ and }\,\,V(D(\eta,s))\approx (1-|\eta|^2)^{n+1}.
\end{align}
 Then
 $$N_*=N_*(s,\delta)=\sup_{z\in\BB,\eta\in D(z,s)}\frac{V(D(z,s+\frac{\delta}{10})}{V(D(\eta,\frac{\delta}{10}))}<\infty.$$
 So, there are at most $N_*$ elements of $\{a_k\}$ contained in $D(z,s)$ for any $z\in\BB$.
 Moreover, by (\ref{0422-2})  we have
$$\sum_{k=1}^\infty W_1(D(a_k,s))\approx \sum_{k=1}^\infty W_1(D(a_k,\frac{\delta}{10}))\leq W_1(\BB).$$
 These facts are very important for our proof and we will use them repeatedly.

For any $c=\{c_k\}_{k=1}^\infty\in l^p$, let
\begin{align*}
F(w)=\sum\limits_{k=1}^\infty c_k b_{a_k}^{\om,p}(w)~~~\,\,\mbox{ and }\,\,F_t(w)=\sum\limits_{k=1}^\infty c_k r_k(t) b_{a_k}^{\om,p}(w),
\end{align*}
where $b_a^{\om,p}(w)=\frac{B_a^\om(w)}{\|B_a^\om\|_{A_\om^p}}$ and  $r_k(t)$ denotes the $k$th Rademacher function.
For any given $w\in\BB$, by H\"older's inequality, (\ref{0412-1}), (\ref{0422-2}) and Lemma \ref{0507-1}, we have
\begin{align*}
|F(w)|=\left| \sum_{k=1}^\infty c_k b_{a_k}^{\om,p} (w) \right|
&\leq \|c\|_{l^p} \left(\sum_{k=1}^\infty  \om(S_{a_k})|B_{a_k}^\om(w)|^\frac{p}{p-1}\right)^\frac{p-1}{p} \\
&\lesssim \|c\|_{l^p} \left(\sum_{k=1}^\infty  {W_1}(D(a_k,\frac{\delta}{10}))\|B_w^\om\|_{H^\infty}^\frac{p}{p-1}\right)^\frac{p-1}{p}
<\infty.
\end{align*}
So, $F\in H(\BB)$.
For all $g\in A_\om^{p^\p}$, by H\"older's inequality, (\ref{0412-1}) , the subharmonicity of $|g|^{p^\p}$ and Lemma \ref{0507-1}, we have
\begin{align*}
\left|\langle g,F \rangle_{A_\om^2}\right|
&=\left|\sum_{k=1}^\infty \ol{c_k} \frac{g(a_k)}{\|B_{a_k}^\om\|_{A_\om^p}} \right|
\lesssim \|c\|_{l^p}\left(\sum_{k=1}^\infty \frac{|g(a_k)|^{p^\p}}{\|B_{a_k}^\om\|_{A_{W_1}^p}^{p^\p}}\right)^\frac{1}{p^\p}\\
&\lesssim \|c\|_{l^p}\left(\sum_{k=1}^\infty \int_{D(a_k,\frac{\delta}{10})}|g(z)|^{p^\p}{W_1}(z)dV(z)\right)^\frac{1}{p^\p}\\
&\lesssim  \|c\|_{l^p}\|g\|_{A_{W_1}^{p^\p}} \approx \|c\|_{l^p}\|g\|_{A_\om^{p^\p}}.
\end{align*}
By Lemma \ref{0413-7}, we have $F\in A_\om^p$ and $\|F\|_{A_\om^p}\lesssim \|c\|_{l^p}$.
Thus, $\|F_t\|_{A_\om^p}\lesssim \|\{c_kr_k(t)\}_{k=1}^\infty\|_{l^p}$.

For brief, let $\|\T _\mu\|=\|\T _\mu\|_{A_\om^p\to A_\om^q}$ and $\chi_E$ be the characterization function of a Borel set $E$ in $\BB$.
Then for all $z\in \BB$, we have $\sum_{k=1}^\infty \chi_{D(a_k,s)}(z)\leq N_*$.
By  Fubini's theorem and Khintchine's inequality, we get
\begin{align*}
\|\T _\mu\|^q\|c\|_{l^p}^q
&\geq \int_0^1 \|\T _\mu\|^q\|\{c_k r_k(t)\}_{k=1}^\infty\|_{l^p}^q dt
 \geq \int_0^1 \|\T _\mu\|^q \|F_t\|_{A_\om^p}^q dt  \\
 &\geq\int_0^1  \|\T _\mu F_t\|_{A_\om^q}^q dt
 \approx\int_\BB \int_0^1 \left| \sum_{k=1}^\infty c_k r_{k}(t) \T_\mu b_{a_k}^{\om,p}(z)      \right|^q   dt W_1(z)dV(z)   \\
 &\gtrsim \int_\BB  \left| \sum_{k=1}^\infty |c_k|^2 |\T_\mu b_{a_k}^{\om,p}(z)|^2      \right|^\frac{q}{2}   W_1(z)dV(z)  \\
 &\geq \int_\BB  \left| \sum_{k=1}^\infty |c_k|^2 |\T_\mu b_{a_k}^{\om,p}(z)|^2\chi_{D(a_k,s)}(z)      \right|^\frac{q}{2}   W_1(z)dV(z)  \\
 &\gtrsim \sum_{k=1}^\infty  |c_k|^q \int_{D(a_k,s)} |\T_\mu b_{a_k}^{\om,p}(z)|^q W_1(z)dV(z).
\end{align*}

If $s\leq r_0$, by the subharmonicity of  $|\T_\mu b_{a_k}^{\om,p}|^q$, Lemma \ref{0422-1},  (\ref{0412-1}), (\ref{0422-3}), and (\ref{0422-2}), we have
\begin{align*}
\int_{D(a_k,s)}|\T_\mu b_{a_k}^{\om,p}(z)|^q W_1(z)dV(z)
&\gtrsim W_1(D(a_k,s))|\T_\mu (b_{a_k}^{\om,p})(a_k)|^q  \\
&\approx \frac{\om(S_{a_k})}{\|B_{a_k}^\om\|_{A_\om^p}^q}\left(\int_\BB |B_{a_k}^\om(\xi)|^2d\mu(\xi)\right)^q   \\
&\geq  \frac{\om(S_{a_k})}{\|B_{a_k}^\om\|_{A_\om^p}^q}\left(\int_{D(a_k,s)} |B_{a_k}^\om(\xi)|^2d\mu(\xi)\right)^q   \\
&\gtrsim \frac{\om(S_{a_k})}{\|B_{a_k}^\om\|_{A_\om^p}^q}|B_{a_k}^\om(a_k)|^{2q}(\mu(D(a_k,s)))^q  \\
&\approx \left(\frac{\mu(D(a_k,s))}{\om(S_{a_k})^{1+\frac{1}{p}-\frac{1}{q}}} \right)^q.
\end{align*}
Therefore, if $0<s\leq r_0$, for all $c\in l^p$, we have
\begin{align}\label{0423-2}
\sum_{k=1}^\infty  |c_k|^q \left(\frac{\mu(D(a_k,s))}{\om(S_{a_k})^{1+\frac{1}{p}-\frac{1}{q}}}\right)^q \lesssim \|\T _\mu\|^q\|c\|_{l^p}^q.
\end{align}

If $s>r_0$, by assumption, $\{a_k\}$ is  a $\delta-$lattice with $5\delta \leq r_0$.
For any $a_k$, we can find $N_k$ (maybe $N_k=\infty$) elements of $\{a_k\}$, write as $a_{k,1},a_{k,2},\cdots,a_{k,N_k}$ such that,
$$D(a_k,s)\subset\cup_{j=1}^{N_k}D(a_{k,j},5\delta),$$
and
$$D(a_k,s)\cap D(a_{k,j}5\delta)\neq \O\,\,\mbox{ for }\,\,j=1,2,\cdots,N_k.$$
Thus, $a_{k,j}\in D(a_k,s+5\delta)$ and $D(a_{k,j},\frac{\delta}{10})\subset D(a_k,s+6\delta)$.
By (\ref{0422-2}) and (\ref{0423-1}), we have
$$N_\#=\sup\limits_{k\geq 1}{N_k}<\infty,\,\,\mbox{and}\,\,\om(S_{a_k})\approx \om(S_{a_{k,j}}).$$
At the same time, we should note that, for $k_1\neq k_2$, we may have
$$\{a_{k_1,j}\}_{j=1}^{N_{k_1}}\cap\{a_{k_2,j}\}_{j=1}^{N_{k_2}}\neq \O.$$
If $a_{k_0}$ appears  in some $\{a_{k,j}\}_{j=1}^\infty$,
for convenience, let $a_{k_0}=a_{k_1,1}=\cdots=a_{k_t,1}$ with $k_1<k_2<\cdots<k_t$.
Then for $j=1,2,\cdots,t$, we have
$$a_{k_0}=a_{k_j,1}\in D(a_{k_j},s+5\delta),\,\,\mbox{i,e.}\,\, a_{k_j}\in D(a_{k_0},s+5\delta).$$
Since $\{a_k\}$ is a $\delta$-lattice, there exists $N_\dag<\infty$, we always have $t\leq N_\dag$.

For convenience, if $N_k<N_\#$, let $\frac{\mu(D(a_{k,j},s))}{\om(S_{a_{k,j}})^{1+\frac{1}{p}-\frac{1}{q}}}=0$, for $j=N_k+1,\cdots,N_\#$.
By (\ref{0423-2}), we have
\begin{align*}
\sum_{k=1}^\infty  |c_k|^q \left(\frac{\mu(D(a_k,s))}{\om(S_{a_k})^{1+\frac{1}{p}-\frac{1}{q}}}\right)^q
&\lesssim \sum_{k=1}^\infty\sum_{j=1}^{N_\#}  |c_k|^q \left(\frac{\mu(D(a_{k,j},r_0))}{\om(S_{a_{k,j}})^{1+\frac{1}{p}-\frac{1}{q}}}\right)^q   \\
&=\sum_{j=1}^{N_\#} \sum_{k=1}^\infty  |c_k|^q\left(\frac{\mu(D(a_{k,j},r_0))}{\om(S_{a_{k,j}})^{1+\frac{1}{p}-\frac{1}{q}}}\right)^q
\lesssim \|\T _\mu\|^q\|c\|_{l^p}^q.
\end{align*}
So, (\ref{0423-2}) holds  for all  $\{c_k\}\in l^p$, $\delta$-lattice $\{a_k\}$  with $5\delta\leq r_0$ and any fixed $s>0$.
Using the facts $\|\{c_k\}_{k=1}^\infty\|_{l^p}^q=\|\{c_k^q\}_{k=1}^\infty\|_{l^\frac{p}{q}}$
 and  $(l^{\frac{p}{q}})^*=l^\frac{p}{p-q}$, we get
 \begin{align*}
  \sum_{k=1}^\infty  \left(\frac{\mu(D(a_k,s))}{\om(S_{a_k})^{1+\frac{1}{p}-\frac{1}{q}}}\right)^\frac{pq}{p-q}
\leq C(s,\delta) \|\T _\mu\|^\frac{pq}{p-q}.
\end{align*}

For any fixed $r>0$, choose $s=s(r,\delta)$ such that $D(z,r)\subset D(a_j,s)$ for all $z\in D(a_j,5\delta)$ and $j\in\N$.
By (\ref{0423-1}),  we have
\begin{align*}
\|\widehat{\mu_r}\|_{L_{W_1}^\frac{pq}{p-q}}^\frac{pq}{p-q}
&\leq \sum_{j=1}^\infty \int_{D(a_j,5\delta)} \left(\frac{\mu(D(z,r))}{\om(S_z)}\right)^\frac{pq}{p-q}W_1(z)dV(z)  \\
&\approx\sum_{j=1}^\infty \frac{W_1(a_j)}{\left((1-|a_j|)^{n+1}W_1(a_j)\right)^\frac{pq}{p-q}}   \int_{D(a_j,5\delta)} \mu(D(z,r))^\frac{pq}{p-q}dV(z)  \\
&\lesssim \sum_{j=1}^\infty \frac{\mu(D(a_j,s))^\frac{pq}{p-q}}{\left((1-|a_j|)^{n+1}W_1(a_j)\right)^{\frac{pq}{p-q}-1}}
\approx \sum_{j=1}^\infty \left(\frac{\mu(D(a_j,s))^\frac{pq}{p-q}}{\om(S_{a_j})^{1+\frac{1}{p}-\frac{1}{q}}}\right)^{\frac{pq}{p-q}}
\lesssim \|\T_\mu\|^\frac{pq}{p-q}.
\end{align*}
Therefore, {\it (iii)} holds.

 {\it (iv)}$\Rightarrow${\it (ii)}. Suppose that $Id:A_\om^p\to L_\mu^\frac{pq+q-p}{q}$ is bounded.
For any $f\in A_\om^p$, we have $f\in L_\mu^1$. Let $x=\frac{pq+q-p}{q}$.
 Keep $\|\cdot\|_{A_\om^p}\approx \|\cdot\|_{A_{W_1}^p}$ and $\|\cdot\|_{A_\om^{q^\p}}\approx \|\cdot\|_{A_{W_1}^{q^\p}}$ in mind.
Then for all polynomial $g$,
by Lemma \ref{0413-6}, H\"older's inequality, Lemma \ref{0424-1} together with the equality $\frac{x}{p}=\frac{x^\p}{q^\p}$, we have
\begin{align*}
\left|\langle \T_\mu f,g\rangle_{A_\om^2}\right|
&\leq \int_\BB |f(z)g(z)|d\mu(z)\leq \|f\|_{L_\mu^x}\|g\|_{L_\mu^{x^\p}}  \\
&\leq \|Id\|_{A_{W_1}^p\to L_\mu^x}\|Id\|_{A_{W_1}^{q^\p}\to L_\mu^{x^\p}}\|f\|_{A_{W_1}^p}\|g\|_{A_{W_1}^{q^\p}}  \\
&\approx \|Id\|_{A_{W_1}^p\to L_\mu^x}^x \|f\|_{A_{W_1}^p}\|g\|_{A_{W_1}^{q^\p}}
\approx \|Id\|_{A_\om^p\to L_\mu^x}^x \|f\|_{A_\om^p}\|g\|_{A_\om^{q^\p}}.
\end{align*}
Since polynomials are dense in $A_\om^{q^\p}$ and $(A_\om^q)^*\simeq   A_\om^{q^\p}$,
we see that $\T_\mu:A_\om^p\to A_\om^q$ is bounded and $\|\T_\mu\|_{A_\om^p\to L_\mu^q}\lesssim
 \|Id\|_{A_\om^p\to L_\mu^x}^x$.

{\it (iii)}$\Rightarrow${\it (v)}. Suppose that {\it (iii)} holds, that is, $\widehat{\mu_r}(z)=\frac{\mu(D(z,r))}{\om(S_z)}\in L_{W_1}^\frac{pq}{p-q}$.
For brief, let  $x=\frac{pq+q-p}{q}$ and $s\BB=\{z\in\BB:|z|<s\}$ for $0<s<1$.
Since $r>0$ is fixed, there exists a $t=t(s)\in(0,1)$, such that, for all $\xi\in \BB\backslash s\BB$ and $\eta\in D(\xi,r)$, we have $|\eta|>t$.
Moreover, we can assume $\lim\limits_{s\to 1} t(s)=1$.
For every $\varepsilon>0$, by {\it (iii)} and $\lim\limits_{s\to 1} t(s)=1$,  there exists a $s>0$ such that
\begin{align}\label{0429-2}
\left(\int_{\BB\backslash t\BB}  \left(\frac{\mu(D(\eta,r))}{W_1(D(\eta,r))}\right)^{\frac{pq}{p-q}} W_1(\eta)dV(\eta)\right)^{\frac{p-q}{pq}}<\varepsilon.
\end{align}

 Suppose $\{g_k\}_{k=1}^\infty$ is bounded in $A_{\om}^p$ and   converges to 0 uniformly on compact subset of $\BB$.
By the subharmonicity of $|g_k|^x$, Fubini's theorem, H\"older's inequality, we have
\begin{align*}
\int_{\BB\backslash s\BB} |g_k(\xi)|^x d\mu(\xi)
&\lesssim \int_{\BB\backslash s\BB} \frac{1}{W_1(D(\xi,r))}\int_{D(\xi,r)}|g_k(\eta)|^x W_1(\eta)dV(\eta) d\mu(\xi) \\
&\leq \int_{\BB\backslash t\BB}  |g_k(\eta)|^x W_1(\eta)dV(\eta)\int_{D(\eta,r)} \frac{1}{W_1(D(\xi,r))}  d\mu(\xi)  \\
&\approx \int_{\BB\backslash t\BB}  |g_k(\eta)|^x  \frac{\mu(D(\eta,r))}{W_1(D(\eta,r))} W_1(\eta)dV(\eta)\\
&\leq \|g_k\|_{A_{W_1}^p}^x\left(\int_{\BB\backslash t\BB}  \left(\frac{\mu(D(\eta,r))}{W_1(D(\eta,r))}\right)^{\frac{pq}{p-q}} W_1(\eta)dV(\eta)\right)^{\frac{p-q}{pq}}.
\end{align*}
 So,
 \begin{align*}\lim_{k\to\infty}\|g_k\|_{L_\mu^x}^x
 &=\left(\int_{s\BB}+\int_{\BB\backslash s\BB}\right) |g_k(\xi)|^x d\mu(\xi)  \\
 &\leq \lim_{k\to\infty}\sup_{z\in s\BB}|g_k(z)|^x\mu(\BB)+\varepsilon \|g_k\|_{A_{W_1}^p}^x
 \leq \varepsilon \|g_k\|_{A_{W_1}^p}^x
 \lesssim \varepsilon \|g_k\|_{A_{\om}^p}^x.
 \end{align*}
 By Lemma \ref{1210-2}, we get that $Id:A_\om^p\to L_\mu^\frac{pq+q-p}{q}$ is compact.

   {\it (iii)}$\Rightarrow${\it (i)}.
Let $d\mu_t=\chi_{\BB\backslash t\BB}d\mu$ and $\widehat{(\mu_t)}_{r}(z)=\frac{\mu_t(D(z,r))}{\om(S_z)}$.
By Lemma \ref{0313-1} and (\ref{0429-1}), we have
\begin{align*}
\|\T_\mu g_k\|_{A_\om^q}
&=\left(\int_\BB\left|\int_\BB g_k(\xi)\ol{B_z^\om(\xi)}d\mu(\xi)\right|^q\om(z)dV(z)\right)^\frac{1}{q}   \\
&= \left(\int_\BB\left|\left(\int_{t\BB}+\int_{\BB\backslash t\BB}\right) g_k(\xi)\ol{B_z^\om(\xi)}d\mu(\xi)\right|^q\om(z)dV(z)\right)^\frac{1}{q}  \\
&\lesssim \sup_{|\xi|\leq t}\frac{|g_k(\xi)|}{\om(S_\xi)} +\|\T_{\mu_t}\|_{A_\om^p\to A_\om^q}\|g_k\|_{A_\om^p}
\lesssim \sup_{|\xi|\leq t}\frac{|g_k(\xi)|}{\om(S_\xi)} +\|\widehat{(\mu_t)}_{r}\|_{L_{W_1}^{\frac{pq}{p-q}}}\|g_k\|_{A_\om^p}.
\end{align*}
By (\ref{0429-2}), we have
$$\lim_{k\to \infty}\|\T_\mu g_k\|_{A_\om^q}\lesssim \varepsilon \sup\|g_k\|_{A_\om^p}.$$
 By Lemma \ref{1210-2}, $\T_\mu:A_\om^p\to A_\om^q$ is compact.

  Suppose $\om\in\R$ and {\it (iii)} holds. Let $h$ be a positive subharmonic function in $\BB$. Then Lemma \ref{1210-3} and Fubini's theorem yield
\begin{align*}
\int_\BB h(z)d\mu(z)&\lesssim \int_\BB\frac{1}{(1-|z|^2)^{n+1}}\int_{D(z,r)}h(\xi)dV(\xi)d\mu(z)   \\
&\approx \int_\BB\int_{D(z,r)}h(\xi)\frac{\om(\xi)}{\om(S_\xi)}dV(\xi)d\mu(z)
=\int_\BB h(\xi)\widehat{\mu_r}(\xi)\om(\xi)dV(\xi).
\end{align*}
Then Lemma \ref{0313-1} and (\ref{0412-1}) yield
\begin{align*}
\widetilde{\T_\mu}(z)&=\int_\BB |b_z^\om(\xi)|^2d\mu(z)\lesssim  \int_\BB |b_z^\om(\xi)|^2\widehat{\mu_r}(\xi)\om(\xi)dV(\xi)  \\
&\lesssim \frac{\|B_z^\om\|_{H^\infty}}{\|B_z^\om\|_{A_\om^2}^2}\int_\BB |B_z^\om(\xi)| \widehat{\mu_r}(\xi)\om(\xi)dV(\xi)
\approx P_\om^+(\widehat{\mu_r})(z).
\end{align*}
By Theorem C, $P_\om^+:L_\om^\frac{pq}{p-q}\to L_\om^\frac{pq}{p-q}$ is bounded. Then
$\|\widetilde{\T_\mu}\|_{L_\om^\frac{pq}{p-q}}\lesssim \|\widehat{\mu_r}\|_{L_\om^\frac{pq}{p-q}}.$

Assume $\om\in\R$, $\widetilde{\T_\mu}\in L_\om^\frac{pq}{p-q}$ and $t\in (0,r_0)$, where $r_0=r_0(\om)$ is that of Lemma \ref{0422-1}.
By the proof of Lemma \ref{0424-1}, we have $\|\widehat{\mu_t}\|_{L_{\om}^\frac{pq}{p-q}}\approx \|\widehat{\mu_r}\|_{L_{\om}^\frac{pq}{p-q}}$.
Then (\ref{0412-1}) gives
\begin{align*}
\widetilde{\T_\mu}(z)\geq \int_{D(z,t)}|b_z^\om(\xi)|^2d\mu(\xi)\approx |b_z^\om(z)|^2\mu(D(z,t))\approx \widehat{\mu_t}(z).
\end{align*}
Since $\om\in\R$,
$$\|\widehat{\mu_r}\|_{L_{\om}^\frac{pq}{p-q}}\approx\|\widehat{\mu_t}\|_{L_{\om}^\frac{pq}{p-q}}\lesssim \|\widetilde{\T_\mu}\|_{L_\om^\frac{pq}{p-q}}.$$
By (\ref{0429-1}), $\|\widetilde{\T_\mu}\|_{L_\om^\frac{pq}{p-q}}\approx \|\T_\mu\|_{A_\om^p\to A_\om^q}$.
The proof is complete.
\end{proof}

\section{Schatten class  Toeplitz operators}

 In this section, we will define a new kind of Dirichlet spaces and investigate the Schatten class Toeplitz operators on them.
 As an application,  we will characterize the Schatten class of $\T_\mu:A_\om^2\to A_\om^2$ with $\om\in\hD$.

Suppose $\mathcal{H}$ is a Hilbert space which is separable and $T:\mathcal{H}\to \mathcal{H}$ is compact.
Let $k=0,1,2,\cdots$ and
$$\lambda_k(T)=\inf\{\|T-R\|_{\mathcal{H}\to\mathcal{H}}:\mbox{rank}(R)\leq k\}.$$
Obviously,
$$\|T\|_{\mathcal{H}\to\mathcal{H}}=\lambda_0(T)\geq \lambda_1(T)\geq \lambda_2(T)\geq \cdots\geq 0.$$
If $\{\lambda_k(T)\}_{k=0}^\infty\in l^p$ for some $p\in(0,\infty)$, we say that $T$ belongs to the Schatten $p$-class, denoted by $T\in \mathcal{S}_p(\mathcal{H})$.
With respect to the norm $|T|_p=\|\{\lambda_k(T)\}_{k=0}^\infty\|_{ l^p}$, $ \mathcal{S}_p(\mathcal{H})$ is a Banach space when $1\leq p<\infty$. More information about $\mathcal{S}_p(\mathcal{H})$ can be seen in \cite[Chapter 1]{zhu}.

Recall  that for any $\xi\in\SS$ and $0<r<\sqrt{2}$,
  $$Q(\xi,r)=\{\eta\in \SS:|1-\langle \xi,\eta\rangle|\leq r^2\}.$$

\begin{Lemma}\label{0529-1}
For any $0<r<1$, there exist   $\xi_{r,1},\xi_{r,2},\cdots,\xi_{r,N_r}$  in $\SS$ such that
\begin{enumerate}[(i)]
  \item $Q(\xi_{r,i},r)\cap Q(\xi_{r,j},r)=\O$, if $1\leq i<j\leq N_r$;
  \item $\SS=\cup_{i=1}^{N_r}Q(\xi_{r,i},2r)$;
  \item $N_r\approx r^{-2n}$.
\end{enumerate}
Moreover, there exist $Q_{r,j}(j=1,2,\cdots,N_r)$ such that
\begin{enumerate}[(i)]
  \item[(iv)] $Q(\xi_{r,j},r)\subset Q_{r,j}\subset Q(\xi_{r,j},2r)$ for all $1\leq j\leq N_r$;
  \item[(v)] $\SS=\cup_{j=1}^{N_r} Q_{r,j};$
  \item[(vi)] $Q_{r,i}\cap Q_{r,j}=\O$ when $1\leq i<j\leq N_r$.
\end{enumerate}
\end{Lemma}

\begin{proof} By Lemma 4.6 in \cite{Zk2005}, for all $\xi\in\SS$ and $r\in (0,\sqrt{2})$, $\sigma(Q(\xi,r))\approx r^{2n}$. So, there exist at most $N_r$  points $\xi_{r,1},\xi_{r,2},\cdots,\xi_{r,N_r}$ such that $N_r\lesssim r^{-2n}$ and
$$Q(\xi_{r,i},r)\cap Q(\xi_{r,j},r)=\O,\,\,\mbox{ when }\,\,1\leq i<j\leq N_r.$$
 If $\SS\neq \cup_{k=1}^{N_r}Q(\xi_{r,k},2r)$, there exist $\eta\in\SS$ such that  $d(\xi_{r,i},\eta)\geq 2r$ for  $i=1,2,\cdots, N_r$.
 Then $Q(\eta,r)\cap Q(\xi_{r,i},r)=\O$ when $i=1,2,\cdots,N_r$. It is contradict with $N_r$ is the largest number.
 Therefore,
 $$\SS=\cup_{i=1}^{N_r}Q(\xi_{r,k},2r),\,\,\mbox{ and }\,\,N_r\gtrsim (2r)^{-2n}\approx r^{-2n}.$$
Then $N_r\approx r^{-2n}.$

For $j=1,2,\cdots,N_r$, let    $$E_{r,j}=Q(\xi_{r,j},2r)-\cup_{i\neq j}Q(\xi_{r,i},r).$$
Then
$$Q(\xi_{r,j},r)\subset E_{r,j}\subset Q(\xi_{r,j},2r).$$
For every $\eta\in\SS$, if $\eta\in Q(\xi_{r,j},r)$ for some $j$, we have $\eta\in E_{r,j}$; otherwise, we have $\eta\in E_{r,k}$ for some $k$ with  $\eta\in Q(\xi_{r,k},2r)$. That is to say, $\SS=\cup_{j=1}^{N_r} E_{r,j}$.

Let $Q_{r,1}=E_{r,1}$ and
$$Q_{r,j+1}=E_{r,j+1}-\cup_{i=1}^j Q_{r,i},\,\,j=1,2,\cdots,N_r-1.$$
Then
$$Q_{r,j}\subset Q(\xi_{r,j},2r),\,\,\,\SS=\cup_{j=1}^{N_r} Q_{r,j},\,\,\,\mbox{ and }\,\,\,Q_{r,i}\cap Q_{r,j}=\O\,\,\,\mbox{ if }\,\,\,i\neq j.$$
Obviously, we have $Q(\xi_{r,1},r)\subset Q_{r,1}$. Suppose $Q(\xi_{r,j},r)\subset Q_{r,j}$ for all $1\leq j\leq k$, where $1\leq k<N_r$ is fixed.
If $\eta\in Q(\xi_{r,k+1},r)$, we have $\eta\in E_{r,k+1}$ and $\eta\not\in E_{r,j}$ with $j\neq k+1$. Thus $\eta\in Q_{r,k+1}$. So, $Q(\xi_{r,k+1},r)\subset Q_{r,k+1}$. By mathematical induction, $Q(\xi_{r,j},r)\subset Q_{r,j}$ for all $1\leq j \leq N_r$.
 The proof is complete.
\end{proof}

When $k=1,2,\cdots$, let
$$N_k=N_{\frac{1}{\sqrt{2^k}}}, \,\,\xi_{k,j}=\xi_{\frac{1}{\sqrt{2^k}},j},\,\mbox{ and }\, Q_{k,j}=Q_{\frac{1}{\sqrt{2^k}},j}.$$
Define $c_{k,j}=(1-\frac{3}{2^{k+2}})\xi_{k,j}$ and
$$R_{k,j}=\left\{z\in\BB: 1-\frac{1}{2^k}\leq |z|<1-\frac{1}{2^{k+1}},  \frac{z}{|z|}\in Q_{k,j}\right\}.$$
For convenience, let $\xi_{0,1}=(1,0,\cdots,0)$, $c_{0,1}\in\frac{1}{4}\xi_{0,1}, Q_{0,1}=\SS$ and $R_{0,1}=\frac{1}{2}\BB$.
Let
$$\Upsilon=\{R_{k,j}:k=0,1,2,\cdots, j=1,2,\cdots,N_k\}.$$
 Then, $\BB=\cup_{k=0}^\infty \cup_{j=1}^{N_k} R_{k,j}$ and $N_k\approx 2^{nk}$.

\begin{Lemma}\label{0618-1}
The following statements hold.
\begin{enumerate}[(i)]
  \item Suppose $0<r<1$ is fixed, there exists $N=N(r)$,  such that,   for any $z\in\BB$, $\Delta(z,r)$ can be covered by a subsets of $\{R_{k,j}\}$ with no more than $N$ elements.
  \item For any given $0<s<r<1$,  there exists $M=M(s,r)$, such that, if $\{a_i\}$ is  $s$ pseudo-hyperbolic separated and $\BB=\cup\Delta(a_i,r)$, each $R_{k,j}$   can be covered by  a subset of $\{\Delta(a_j,r)\}$ with no more than $M$ elements.
\end{enumerate}
\end{Lemma}
\begin{proof}
{\bf {\it (i)}. }For any $z\in\BB$, $\Delta(z,r)$ consists of  all $w\in \BB$ such that
$$\frac{|P_z(w)-c|^2}{r^2t^2}+\frac{|P_z^\perp(w)|^2}{r^2t}<1,$$
where
$$ c=\frac{(1-r^2)z}{1-r^2|z|^2}~~,\mbox{~~~~}~~t=\frac{1-|z|^2}{1-r^2|z|^2}.$$
As $|z|\to 1$, we have $|c|\to 1$ and $t\to 0$.

Without loss of generality,  assume $z=(|z|,0,0,\cdots,0)$ and $|z|\in [1-\frac{1}{2^k}, 1-\frac{1}{2^{k+1}})$ for some $k\in\NN$.
Then $w=(w_1,w_2,\cdots,w_n)\in \Delta(z,r)$ if and only if
$$\frac{|w_1 -|c||^2}{r^2t^2}+\frac{|w_2|^2+|w_3|^2+\cdots+|w_n|^n}{r^2t}<1.$$
Here,
$$ c=\frac{(1-r^2)z}{1-r^2|z|^2}~~,\mbox{~~~~}~~t=\frac{1-|z|^2}{1-r^2|z|^2}.$$
After a calculation, we have   $|c|-rt < |w|<|c|+rt$.
Let $x=\mbox{int}(\log_2\frac{1+r}{1-r})+1$.  Then
\begin{align}\label{0610-1}
1-\frac{1}{2^{k-x}}<1-\frac{1+r}{1-r}\frac{1}{2^k}<|w|<1-\frac{1-r}{1+r}\frac{1}{2^{k+1}}<1-\frac{1}{2^{k+x+1}}.
\end{align}
At the  same time, if $k$ is large enough such that $|w|>\frac{1}{2}$ always holds, then we have
\begin{align}\label{0826-1}
\left|1-\langle \frac{w}{|w|}, \frac{c}{|c|} \rangle\right|
&=\left|1-\frac{w_1}{|w|}\right|\leq \frac{|w_1-|c||+||c|-|w||}{|w|}<4t.
\end{align}

For any $\delta= k-x,k-x+1,\cdots,k+x$, there exists $Q_{\delta,j_1}, Q_{\delta,j_2},\cdots,Q_{\delta,j_{M_\delta}}$ such that
$$Q(\frac{c}{|c|},\sqrt{4t})\subset\cup_{i=1}^{M_\delta}Q_{\delta,j_i},\,\,\mbox{ and }\,\,Q(\frac{c}{|c|},\sqrt{4t})\cap Q_{\delta,j_i}\neq \O,\,\mbox{ for }\,i=1,2,\cdots, M_\delta.$$
Then we have
$$\cup_{i=1}^{M_\delta}Q_{\delta,j_i}\subset Q(\frac{c}{|c|},\sqrt{4t}+\frac{2}{\sqrt{2^\delta}}).$$
By Lemma \ref{0529-1}, \cite[Lemma 4.6]{Zk2005} and $t\approx 1-|z|\approx \frac{1}{2^k}$, we obtain
$$M_\delta \leq \frac{\sup\limits_{\eta\in\SS}\sigma(Q(\eta,\sqrt{4t}+\frac{2}{\sqrt{2^\delta}}))}
{\inf\limits_{\eta\in\SS}\sigma(Q(\eta,\frac{1}{\sqrt{2^\delta}}))}
\approx \left(\frac{\sqrt{4t}+\frac{2}{\sqrt{2^\delta}}}{\frac{1}{\sqrt{2^\delta}}}\right)^{2n}
\approx 1.$$
This and (\ref{0610-1}) deduce the desired, that is,  for any $z\in\BB$, $\Delta(z,r)$ can be covered by a subsets of $\{R_{k,j}\}$ with no more than $N(r)$ elements.

{\bf {\it (ii)}.} Suppose $\{a_i\}$ is  $s$ pseudo-hyperbolic separated and $\BB=\cup_{i=1}^\infty\Delta(a_i,r)$ for any fixed  $0<r<\frac{1}{2}$.
For any given $R_{k,j}$, without loss of generality, suppose there exists a constant $M_{k,j}\in\NN$ such that
$$R_{k,j}\cap\Delta(a_i,r)\neq\O,\,\mbox{ for all }\,i=1,2,\cdots, M_{k,j},$$
and
$$R_{k,j}\cap\Delta(a_i,r)=\O,\,\mbox{ for all }\,i=M_{k,j}+1,M_{k,j}+2,\cdots.$$
Then $R_{k,j}\subset \cup_{i=1}^{M_{k,j}}\Delta(a_i,r)$.
Let $E=\cup_{\xi\in R_{k,j}}\Delta(\xi,2r).$
Then $$\cup_{i=1}^{M_{k,j}}\Delta(a_i,\frac{s}{2})\subset \cup_{i=1}^{M_{k,j}}\Delta(a_i,r)\subset E.$$
For any $z\in E$, there exists $\xi\in R_{k,j}$ such that $z\in\Delta(\xi,2r)$. By (\ref{0610-1}) and $1-\frac{1}{2^k}\leq |\xi|<1-\frac{1}{2^{k+1}}$, we have
\begin{align*}
1-\frac{1+2r}{1-2r}\frac{1}{2^k}<|z|<1-\frac{1-2r}{1+2r}\frac{1}{2^{k+1}}.
\end{align*}
By (\ref{0826-1}), when $k$ is large enough,  we have
\begin{align*}
\left|1-\langle \frac{z}{|z|}, \frac{\xi}{|\xi|} \rangle\right|
<\frac{4(1-|\xi|^2)}{1-4r^2|\xi|^2}<\frac{8}{1-4r^2}\frac{1}{2^k}.
\end{align*}
Using the notations  defined before Lemma \ref{0618-1}, there exists a constant $0<C(r)<\infty$, such that
\begin{align*}
\left|1-\langle \frac{z}{|z|}, \frac{\xi_{k,j}}{|\xi_{k,j}|} \rangle\right|
&=\left(d(\frac{z}{|z|}, \frac{\xi_{k,j}}{|\xi_{k,j}|})\right)^2
\leq 2\left(d(\frac{z}{|z|}, \frac{\xi}{|\xi|})\right)^2+2\left(d(\frac{\xi}{|\xi|}, \frac{\xi_{k,j}}{|\xi_{k,j}|})\right)^2  \\
&=2 \left|1-\langle \frac{z}{|z|}, \frac{\xi}{|\xi|} \rangle\right|
+ 2\left|1-\langle \frac{\xi}{|\xi|},\frac{\xi_{k,j}}{|\xi_{k,j}|},  \rangle\right|
< \frac{C(r)}{2^k}.
\end{align*}
Let
$$
E^\p=\left\{z\in\BB:1-\frac{1+2r}{1-2r}\frac{1}{2^k}<|z|<1-\frac{1-2r}{1+2r}\frac{1}{2^{k+1}}, \,\,
\left|1-\langle \frac{z}{|z|}, \frac{\xi_{k,j}}{|\xi_{k,j}|} \rangle\right|< \frac{C(r)}{2^k}\right\}.
$$
Then,
$$\cup_{i=1}^{M_{k,j}}\Delta(a_i,\frac{s}{2})\subset E^\p,\,\,\,\, V(E^\p)\approx \frac{1}{2^{(n+1)k}}.$$
Meanwhile, for $i=1,2,\cdots, M_{k,j}$, since $R_{k,j}\cap \Delta(a_i,r)\neq\O$, we have $1-|a_i|\approx \frac{1}{2^k}$.
So, $V(\Delta(a_i,\frac{s}{2}))\approx \frac{1}{2^{(n+1)k}}.$
Therefore,
\begin{align*}
M_{k,j}\leq \frac{V(E^\p)}{\inf\limits_{1\leq i\leq M_{k,j}}V(\Delta(a_i,\frac{s}{2}))}\lesssim 1.
\end{align*}

When $\frac{1}{2}\leq r<1$, we can translate the pseudo hyperbolic balls to Bergman metric balls, and proved the statement in the same way. The details will be omitted.
 The proof is complete.
\end{proof}

Suppose $\om\in\hD$. For $z\in \BB\backslash\{0\}$  and $\alpha<2$, let
$$\om^{n*}(z)=\int_{|z|}^1 r^{2n-1}\log\frac{r}{|z|}\om(r)dr$$
and
$$\Waw(z)=\frac{(1-|z|)^{-\alpha}\om^{n*}(z)}{|z|^{2n}}.$$
By Lemma \ref{1210-3}, when $\alpha<2$ and $t>\frac{1}{2}$, we have
\begin{align}\label{0605-1}
\int_\BB |z|^2\Waw(z)dV(z)<\infty\,\,\mbox{ and }\,\, \frac{\widehat{\Waw}(t)}{(1-t)\Waw(t)}\approx 1.
\end{align}

So, when $\alpha<2$, we  define a function space $H(\Waw)$ consisting of all $f\in H(\BB)$ such that
$$\|f\|_{H(\Waw)}^2=|f(0)|^2 \om(\BB)+4\int_\BB |\Re f(z)|^2 \Waw(z)dV(z) <\infty.$$
Obviously, $H(\Waw)$ is a Banach space and polynomials are dense in it.
For all $f,g\in H(\Waw)$, the inner product induced by $\|\cdot\|_{H(\Waw)}$ is
$$\langle f,g \rangle_{H(\Waw)}=f(0)\ol{g(0)}\om(\BB)+4\int_\BB \Re f(z)\ol{\Re g(z)} \Waw(z)dV(z).$$

By (\ref{0605-1}), even if $\int_0^1 \Waw(t)dt$ is divergent, we can find a  $\Psi\in\R$ such that
$$\|f\|_{H(\Waw)}^2\approx |f(0)|^2 +4\int_\BB |\Re f(z)|^2 \Psi(z)dV(z) <\infty.$$
For example,
$$
\Psi(t)=\left\{
\begin{array}{cc}
  \Waw(\frac{1}{2}), & t\in [0,\frac{1}{2}],  \\
  \Waw(t), & t\in[\frac{1}{2},1).
\end{array}
\right.
$$
So, we always  assume that $\Waw$ is a regular weight.

Theorem 2 in \cite{DjLsLxSy2019arxiv} shows that, if $\om$ is a radial weight,
\begin{align}\label{0605-2}
\|f-f(0)\|_{A_\om^2}^2=4\int_\BB \frac{|\Re f(z)|^2}{|z|^{2n}}\om^{n*}(z)dV(z)\approx \int_\BB |\Re f(z)|^2\om^*(z)dV(z).
\end{align}

\begin{Lemma}\label{0607-8}
Suppose $\om\in\hD$. Then
\begin{align}\label{0605-3}
A_\om^2=H(W_{0}^{\om}),\,\,\mbox{ and } A_{\om^*_{-\alpha-2}}^2\simeq  H(\Waw)\,\,\mbox{ when }\,\,\alpha<0.
\end{align}
Here, $\om^*_{-\alpha-2}(t)=(1-t)^{-\alpha-2}\om^*(t)$.
\end{Lemma}

\begin{proof} It is obvious that $A_\om^2=H(W_{0}^{\om})$.
For any $\om\in\hD$ and $\beta\in \RR$, let
$$\om_\beta(t)=(1-t)^\beta\om(t),\,\,\om^*_\beta(t)=(1-t)^\beta\om^*(t).$$

Assume $\alpha<0$. By Lemma \ref{1210-3}, we have $\om^*_{-\alpha-2}\in\R$. If $t>\frac{1}{2}$, using Lemma \ref{1210-3}, we obtain
\begin{align}\label{0605-4}
(\om^*_{-\alpha-2})^{n*}(t)
\approx  (\om^*_{-\alpha-2})^{*}(t)
\approx \om^*_{-\alpha}(t)
\approx \om^{n*}_{-\alpha}(t).
\end{align}
Since
$$\int_0^\frac{1}{2}r^{2n-1}(\om_{-\alpha-2}^*)^*(r)dr<\infty, \,\mbox{ and }\, \int_\frac{1}{2}^1 r^{2n-1}(\om_{-\alpha-2}^*)^*(r)dr>0,$$
we have
$$\int_0^\frac{1}{2}r^{2n-1}(\om_{-\alpha-2}^*)^*(r)dr<C(\om) \int_\frac{1}{2}^1 r^{2n-1}(\om_{-\alpha-2}^*)^*(r)dr.$$
By (\ref{0605-2}), the monotonicity of $M_2(r,\Re f)$ and (\ref{0605-4}), we have
\begin{align*}
\|f-f(0)\|_{A_{\om_{-\alpha-2}^*}}^2
\approx & \int_\BB |\Re f(z)|^2(\om_{-\alpha-2}^*)^*(z)dV(z)  \\
=&2n\left(\int_0^\frac{1}{2}+\int_\frac{1}{2}^1\right)r^{2n-1}(\om_{-\alpha-2}^*)^*(r)M_2^2(r,\Re f)dr\\
\lesssim & 2n \int_\frac{1}{2}^1 r^{2n-1}(\om_{-\alpha-2}^*)^*(r)M_2^2(r,\Re f)dr\\
\approx & 2n \int_\frac{1}{2}^1 r^{2n-1}\frac{\om_{-\alpha}^{n*}(r)}{r^{2n}}M_2^2(r,\Re f)dr\\
\leq &\int_\BB |\Re f(z)|^2 \Waw(z)dV(z).
\end{align*}
Therefore,
$$\|f\|_{A_{\om_{-\alpha-2}^*}}^2\lesssim |f(0)|^2\om(\BB)+ 4\int_\BB |\Re f(z)|^2 \Waw(z)dV(z).$$

For any $f\in H(\BB)$, let $f_r(z)=f(rz)$ for $r\in(0,1)$. If $|z|\leq\frac{1}{2}$, by Cauchy's formula, we have
\begin{align*}
f(z)=\int_\SS \frac{f_\frac{3}{4}(\eta)}{(1-\langle \frac{4}{3}z,\eta\rangle)^n}d\sigma(\eta).
\end{align*}
So, when $|z|\leq \frac{1}{2}$, since $M_1(r,f)\leq M_2(r,f)$, we have
\begin{align*}
|\Re f(z)|
&= n\left|\int_\SS \frac{f_\frac{3}{4}(\eta)}{(1-\langle \frac{4}{3}z,\eta\rangle)^{n+1}}\langle \frac{4}{3}z,\eta\rangle d\sigma(\eta)\right|
\lesssim |z| M_1(\frac{3}{4},f)\\
&\lesssim |z|\frac{2n\int_\frac{3}{4}^1 M_1(r,f)r^{2n-1}\om_{-\alpha-2}^*(r)dr}{2n\int_\frac{3}{4}^1 r^{2n-1}\om_{-\alpha-2}^*(r)dr}
\lesssim |z|\|f\|_{A_{\om_{-\alpha-2}^*}^2}.
\end{align*}
Thus, by (\ref{0605-2}) and (\ref{0605-4}), we have
\begin{align*}
\int_\BB |\Re f(z)|^2 \Waw(z)dV(z)
=&2n\left(\int_0^\frac{1}{2}+\int_\frac{1}{2}^1\right)r^{2n-1}\Waw(r)M_2^2(r,\Re f)dr\\
\lesssim &\|f\|_{A_{\om_{-\alpha-2}^*}^2}^2+2n\int_\frac{1}{2}^1 r^{2n-1}(\om_{-\alpha-2}^*)^{*}(r)M_2^2(r,\Re f)dr \\
\leq & \|f\|_{A_{\om_{-\alpha-2}^*}^2}^2+\|f-f(0)\|_{A_{\om_{-\alpha-2}^*}^2}^2.
\end{align*}
Therefore,
$$\|f\|_{H(\Waw)}\lesssim \|f\|_{A_{\om_{-\alpha-2}^*}^2}.$$
The proof is complete.
\end{proof}

\begin{Lemma}\label{0627-1}
Suppose $\om$ is continuous and regular.  For any $a,z\in\BB$, let
$$v_a(z)=\int_0^1 \left(B_{a}^\om(tz)-B_{a}^\om(0)\right)\frac{dt}{t}.$$
Then there exists $\delta=\delta(\om)\in(0,1)$ such that
\begin{align}
|v_a(z)|\gtrsim \frac{1}{(1-|a|)^{n-1}\hat{\om}(a)}, \label{0901-1}
\end{align}
for all $|a|\geq \frac{1}{2}$ and $|z-a|<\delta(1-|a|)$.

Therefore, there exists $r=r(\om)\in(0,1)$ such that (\ref{0901-1}) holds for all $|a|\geq \frac{1}{2}$ and $z\in\Delta(a,r)$.
\end{Lemma}

\begin{proof}
By Lemma \ref{0313-1}, we have
$$v_a(z)=\frac{1}{2n!}\sum_{k=1}^\infty \frac{(n-1+k)!}{k\, k!\om_{2n+2k-1}}\langle z,a\rangle^k.$$
Let $|a|>\frac{1}{2}$ and fix the integer $N=N(a)$ such that $1-\frac{1}{N}<|a|\leq 1-\frac{1}{N+1}$.
Using Stirling's estimate, Lemmas 1 and 2, we have
\begin{align*}
v_a(a)
&\approx\sum_{k=1}^\infty \frac{k^{n-2}}{\hat{\om}(1-\frac{1}{k})}|a|^{2k}
\gtrsim \frac{N^{n-2}}{\hat{\om}(1-\frac{1}{N})}\sum_{k=N}^\infty |a|^{2k}  \\
&= \frac{1}{(1-(1-\frac{1}{N}))^{n-2}\hat{\om}(1-\frac{1}{N})}\frac{|a|^{2N}}{1-|a|^2}
\approx\frac{1}{(1-|a|)^{n-1}\hat{\om}(a)}.
\end{align*}

If $|z-a|<\delta(1-|a|)$ and $|a|>\frac{1}{2}$, by Lemma \ref{0313-1}, more specifically, the proof of it in \cite{DjLsLxSy2019arxiv2}, we have
\begin{align*}
|v_a(z)-v_a(a)|
&=\frac{1}{2n!}\left|\sum_{k=1}^\infty \frac{(n-1+k)!}{k\, k!\om_{2n+2k-1}}\left(\langle z,a\rangle^k - \langle a,a\rangle^k\right)\right|\\
&\leq \frac{1}{2n!}\sum_{k=1}^\infty \frac{(n-1+k)!}{ k!\om_{2n+2k-1}}\left|\int_{\langle a,a\rangle}^{\langle z,a\rangle}\eta^{k-1}d\eta \right| \\
&\leq |z-a|\frac{1}{2n!}\sum_{k=1}^\infty \frac{(n-1+k)!}{ k!\om_{2n+2k-1}}|a|^{k-1}
\lesssim  \frac{\delta}{(1-|a|)^{n-1}\hat{\om}(a)}.
\end{align*}
So, when $\delta$ is small enough, we have
\begin{align*}
|v_a(z)|\geq |v_a(a)|-|v_a(a)-v_a(z)|\gtrsim \frac{1}{(1-|a|)^{n-1}\hat{\om}(a)}.
\end{align*}
The proof is complete.
\end{proof}

By Cauchy's formula, each point evaluation $L_z(f)=f(z)$ is a linear bounded functional on $H(\Waw)$ for $\alpha<2$.
So, there exists reproducing kernel $K_z^{\Waw}$ with $\|L_z\|=\|K_z^{\Waw}\|_{H(\Waw)}$ such that
$$f(z)=\langle f,K_z^{\Waw}\rangle_{H(\Waw)},\,\,\mbox{ for all }\,\,f\in H(\Waw).$$
Then, if $\mu$ is a positive Borel measure, the Toeplitz operator $\HT_\mu:H(\Waw)\to H(\Waw)$ is defined by
$$\HT_\mu f(z)=\int_\BB f(w)\ol{K_z^{\Waw}(w)}d\mu(w).$$

\begin{Lemma}\label{0607-1}
Suppose $\alpha<2$ and $\om\in\hD$. Then
$$K_z^\Waw(w)=\frac{1}{\om(\BB)}+\frac{1}{8n!}\sum_{k=1}^\infty \frac{(n-1+k)!}{k^2 k!(\Waw)_{2n+2k-1}}\langle w,z\rangle^k.$$
If $\alpha<1$, then
\begin{align}\label{0606-5}
\|K_z^\Waw\|_{H(\Waw)}^2\approx \frac{(1-|z|)^{\alpha-n+1}}{\om^*(z)},\,\,\mbox{ when }\,\,|z|>\frac{1}{2}.
\end{align}
Here, $(\Waw)_s=\int_0^1 t^s\Waw(t)dt$ for all $s\geq 2n$.
\end{Lemma}

\begin{proof}For any fixed $z\in\BB$, let
$$K_z^{\Waw}(w)=\sum_{m}a_m(z)w^m.$$
Let $f(z)=\sum\limits_{m}b_mz^m$ be a polynomial. By  (1.22) in \cite{Zk2005}, we have
\begin{align*}
f(z) &=\ol{a_0(z)}b_0\om(\BB)+4\int_\BB \left(\sum_{|m|>0}|m|b_m w^m\right)\ol{\left(\sum_{|m|>0}|m|a_m(z) w^m\right)}\Waw(w)dV(w)\\
&= \ol{a_0(z)}b_0\om(\BB)+4\sum_{|m|>0}2n |m|^2 \ol{a_m(z)}b_m   \frac{(n-1)!m!}{(n-1+|m|)!}(\Waw)_{2n+2|m|-1}.
\end{align*}
Then
$$a_0(z)=\frac{1}{\om(\BB)},\,\,\mbox{ and }\,\,\ol{a_m(z)}=\frac{(n-1+|m|)!}{8n! m! |m|^2(\Waw)_{2n+2|m|-1}}z^m,\,\,\mbox{ when }\,\,|m|>0.$$
Therefore,
\begin{align*}
K_z^{\Waw}(w)
&=\frac{1}{\om(\BB)}+\sum_{|m|>0}\frac{(n-1+|m|)!}{8n! m! |m|^2(\Waw)_{2n+2|m|-1}}  w^m \ol{z}^m   \\
&=\frac{1}{\om(\BB)}+\frac{1}{8n!}\sum_{k=1}^\infty \frac{(n-1+k)!}{k^2 k!(\Waw)_{2n+2k-1}}\sum_{|m|=k} \frac{|m|!}{m!} w^m \ol{z}^m   \\
&=\frac{1}{\om(\BB)}+\frac{1}{8n!}\sum_{k=1}^\infty \frac{(n-1+k)!}{k^2 k!(\Waw)_{2n+2k-1}}\langle w,z \rangle^k.
\end{align*}

Recall that
$$
(\Waw)_{2n+2k-1}
=\int_0^1 t^{2n+2k-1}\frac{(1-t)^{-\alpha}\om^{n*}(t)}{t^{2n}}dt=\int_0^1 t^{2k-1}(1-t)^{-\alpha}\om^{n*}(t)dt,$$
and
$$\,\,\om^{n*}(t)=\int_t^1 s^{2n-1}\log \frac{s}{t}\om(s)ds,\,\,\om^*(t)=\int_t^1 s\log\frac{s}{t}\om(s)ds.$$
By Lemmas \ref{0507-1} and (\ref{1210-3}), we have
\begin{align}
(\Waw)_{2n+2k-1}=\int_0^1 t^{2k-1}(1-t)^{-\alpha}\om^*(t)dt\approx \frac{\om^*(1-\frac{1}{2k})}{(2k)^{-\alpha+1}}
\approx \frac{\om^*(1-\frac{1}{2n+2k-2})}{(2n+2k-2)^{-\alpha+1}}.\label{0607-6}
\end{align}

When $|z|>\frac{1}{2}$ and $\alpha<1$, by Stirling's formula, we get
\begin{align}
\|K_z^{\Waw}\|_{H(\Waw)}^2
&=K_z^{\Waw}(z) \approx 1 + \sum_{k=1}^\infty \frac{k^{n-3}|z|^{2k}}{(\Waw)_{2n+2k-1}}  \nonumber \\
&\approx 1+\sum_{k=1}^\infty \frac{k^{n-1}|z|^{2k}}{k^{\alpha+2}\om_{2k-1}^*}
\approx  \sum_{k=0}^\infty \frac{(k+1)^{n-1}|z|^{2k+2}}{(k+1)^{\alpha+2}\om_{2k+1}^*}. \label{0606-1}
\end{align}
Using (4.5) in \cite{PjaRj2014book}, it is easy to see that $\{(k+1)^2\om_{2k+1}^*\}_{k=0}^\infty$ is decreasing.
Suppose $1-\frac{1}{N}\leq |z|<1-\frac{1}{N+1}$ for some $N\in\NN$ and $N>1$.
By Lemmas \ref{0507-1} and \ref{1210-3}, we have
\begin{align}
\sum_{k=0}^N \frac{(k+1)^{n-1}|z|^{2k+2}}{(k+1)^{\alpha+2}\om_{2k+1}^*}
&\leq \frac{1}{(N+1)^2\om_{2N+1}^*}\sum_{k=0}^\infty \frac{|z|^{2k+2}}{(k+1)^{\alpha+1-n}} \nonumber \\
&\approx \frac{(1-|z|)^{\alpha-n}}{(N+1)^2\om_{2N+1}^*}
\approx \frac{(1-|z|)^{\alpha-n+1}}{\om^*(z)},\label{0606-2}
\end{align}
and
\begin{align}
\sum_{k=N}^\infty \frac{(k+1)^{n-1}|z|^{2k+2}}{(k+1)^{\alpha+2}\om_{2k+1}^*}
&\geq \frac{1}{(N+1)^2\om_{2N+1}^*}\sum_{k=N}^{2N} \frac{|z|^{2k+2}}{(k+1)^{\alpha+1-n}} \nonumber \\
&\gtrsim \frac{1}{(N+1)^2\om_{2N+1}^*}\frac{(N+1)(1-\frac{1}{N})^{4N+2}}{(N+1)^{\alpha+1-n}} \nonumber \\
&\gtrsim \frac{1}{(N+1)^{\alpha+2-n}\om_{2N+1}^*}
\approx \frac{(1-|z|)^{\alpha-n+1}}{\om^*(z)}.\label{0606-3}
\end{align}

By Lemmas \ref{0507-1} and \ref{1210-3}, there is a $b>0$ such that $\frac{\om^*(t)}{(1-t)^b}$ is essential increasing.
Assume $n-\alpha+b-2>0$. Then
\begin{align}
\sum_{k=N+1}^\infty \frac{(k+1)^{n-1}|z|^{2k+2}}{(k+1)^{\alpha+2}\om_{2k+1}^*}
&\approx \sum_{k=N+1}^\infty \frac{(k+1)^{n-\alpha-2}|z|^{2k+2}}{\om^*(1-\frac{1}{2k+1})}
\approx \sum_{k=N+1}^\infty \frac{(k+1)^{n-\alpha+b-2}|z|^{2k+2}}{\frac{\om^*(1-\frac{1}{2k+1})}{\left(1-(1-\frac{1}{2k+1})\right)^b}}\nonumber\\
&\lesssim  \frac{1}{(2N+3)^b\om^*(1-\frac{1}{2N+3})}\sum_{k=0}^\infty (k+1)^{n-\alpha+b-2}|z|^{2k}\nonumber\\
&\approx  \frac{1}{(2N+3)^b\om^*(1-\frac{1}{2N+3})} \frac{1}{(1-|z|^2)^{n-\alpha+b-1}}  \nonumber\\
&\approx \frac{(1-|z|)^{\alpha-n+1}}{\om^*(z)}.  \label{0606-4}
\end{align}
By (\ref{0606-1})-(\ref{0606-4}), we obtain that (\ref{0606-5}) holds. The proof is complete.
\end{proof}

\begin{Lemma}\label{0607-7} Suppose $\alpha<1$ and $\om\in\hD$. There exists $C=C(\om)>0$, such that for all  $f(z)=\sum_m a_mz^n$, we have
$$\frac{1}{C}\|f\|_{H(\Waw)}^2\leq \sum\limits_{m}\frac{(|m|+1)^{\alpha+2} m!}{(n-1+|m|)!} \om_{2|m|+1}^*|a_m|^2 \leq C\|f\|_{H(\Waw)}^2 .$$
\end{Lemma}
\begin{proof}By (1.21) and Lemma 1.11 in \cite{Zk2005}, (\ref{0607-6}) and Lemma \ref{1210-3}, we have
\begin{align*}
\|f\|_{H(\Waw)}^2
&=|f(0)|^2 \om(\BB)+4\int_\BB |\Re f(z)|^2 \Waw(z)dV(z) \\
&\approx |a_0|^2+\sum_{|m|>0}|m|^2|a_m|^2\int_0^1 r^{2n+2|m|-1}\Waw(r)dr\int_\SS |\eta^m|^2d\sigma(\eta)\\
&\approx |a_0|^2+\sum_{|m|>0}\frac{|m|^{\alpha+2} m!}{(n-1+|m|)!}  \om^*(1-\frac{1}{2|m|}) |a_m|^2 \\
&\approx \sum_{m}\frac{(|m|+1)^{\alpha+2} m!}{(n-1+|m|)!} \om_{2|m|+1}^*|a_m|^2.
\end{align*}
The proof is complete.
\end{proof}

\begin{Theorem}\label{0828-1} Let $\om\in\hD$, $\mu$ be a positive Borel measure on $\BB$, $1\leq p<\infty$ and $-\infty<\alpha<1$ such that $p\alpha<1$. Let
$$\widehat{\mu_{r,\alpha}}=\frac{\mu(D(z,r))}{(1-|z|)^{-\alpha+n-1}\om^*(z)},\,\,\,\,d\lambda(z)=\frac{dV(z)}{(1-|z|^2)^{n+1}}.$$
Then the following statements are equivalent.
\begin{enumerate}[(i)]
  \item $\HT_\mu\in \mcS_p(H(\Waw))$;
  \item $$M_\mu=\sum\limits_{R_{k,j}\in\Upsilon}\left(\frac{\mu(R_{k,j})}{(1-|c_{k,j}|)^{-\alpha+n-1}\om^*(c_{k,j})}\right)^p<\infty;$$
  \item $\widehat{\mu_{r,\alpha}}\in L^p(\BB,d\lambda)$ for some (equivalently, for all) $r>0$.
\end{enumerate}
Moreover, $|\HT_\mu|_p^p\approx M_\mu\approx\|\widehat{\mu_{r,\alpha}}\|_{L^p_\lambda}^p.$
\end{Theorem}

\begin{proof}{\bf {\it (ii)}$\Rightarrow${\it (i)}}. The proof  will be divided into two steps.  Suppose $M_\mu<\infty$.

{\bf The first step.} Assume  $\mu$ is a compactly supported positive Borel measure. It is easy to know that $\HT_\mu$ is compact on $H(\Waw)$.
Then the canonical decomposition of $\HT_\mu$ is, see \cite{zhu} for example,
$$\HT_\mu g(z)=\sum_{k=} \lambda_k\langle g,e_k\rangle_{H(\Waw)}f_k,\,\,\mbox{ for   }\,\,g\in H(\Waw).$$
Here $\{e_k\}$, $\{f_k\}$ are orthogonal  sets in $H(\Waw)$ and  $\{\lambda_k\}$ is the singular values of $\HT_\mu$.
By Fubini's theorem, Cauchy-Schwarz's inequality, Lemmas \ref{0607-1} and  \ref{1210-3},  we get
\begin{align}
|\HT_\mu|_1^1
&=\sum_{k} |\lambda_k|=\sum_{k}\left|\langle\HT_\mu e_k, f_k\rangle_{H(\Waw)}\right|  \nonumber\\
&=\sum_{k} \left|\om(\BB)\HT_\mu e_k(0)\ol{f_k(0)}+4\int_\BB \Re(\HT_\mu e_k)(z)\ol{\Re f_k(z)}\Waw(z)   dV(z)\right|
\nonumber\\
&=\sum_k \left|\int_\BB\ol{\left( \om(\BB)f_k(0)\ol{K_w^\Waw(0)}+4\int_\BB \Re f_k(z)\ol{\Re K_w^\Waw(z)}  \Waw(z)dV(z)   \right)}e_k(w)d\mu(w)\right|
\nonumber\\
&=\sum_k\left|\int_\BB e_k(w)\ol{\langle f_k, K_w^\Waw\rangle_{H(\Waw)}}d\mu(w)\right|=\sum_k\left|\int_\BB e_k(w)\ol{f_k(w)}d\mu(w)\right|
\nonumber\\
&\leq \int_\BB \left(\sum_k |e_k(w)|^2\right)^\frac{1}{2}  \left(\sum_k |f_k(w)|^2\right)^\frac{1}{2} d\mu(w)
\leq \int_\BB \|K_w^\Waw\|_{H(\Waw)}^2  d\mu(w)
\nonumber\\
&\leq \int_\BB \frac{(1-|w|)^{\alpha-n+1}}{\om^*(w)}d\mu(w)
=\sum_{R_{k,j}\in \Upsilon}\int_{R_{k,j}} \frac{(1-|w|)^{\alpha-n+1}}{\om^*(w)}d\mu(w)
\nonumber\\
&\approx \sum_{R_{k,j}\in \Upsilon}  \frac{\mu(R_{k,j})}{(1-|c_{k,j}|)^{-\alpha+n-1}\om^*(c_{k,j})}.\nonumber
\end{align}
Thus we proved the assertion with the case  $p=1$.

Now we assume  $1<p<\infty$ and $p\alpha<1$. Take $\varepsilon>0$ such that $\alpha<2\varepsilon<\frac{1-\alpha}{p-1}$ and $\Lambda=\{\zeta\in\CC: 0\leq \Rp\zeta\leq 1\}$.
For any fixed $\zeta\in\Lambda$, define
$$G_\zeta(\sum_m a_m z^m)=\sum_m (|m|+1)^{\varepsilon(1-p\zeta)}a_m z^m,$$
and $\gamma=\alpha-2\varepsilon(1-p\Rp \zeta)$.
By Lemma \ref{0607-7}, we have
\begin{itemize}
  \item  for all $f\in H(\Waw)$, $\|G_\zeta f\|_{H(W_\gamma^\om)}=\|G_{\Rp \zeta} f\|_{H(W_\gamma^\om)}$;
  \item  $G_\zeta:H(\Waw)\to H(W_\gamma^\om)$ is uniformly bounded and invertible.
\end{itemize}
If $\Rp \zeta=0$, we have $H(W_\gamma^\om)\simeq A_{\om_{-(\alpha-2\varepsilon)-2}^*}^2$ by Lemma \ref{0607-8}.

Let
\begin{align}\label{0608-1}
\mu_\zeta=\sum_{R_{k,j}\in\Upsilon}\left(\frac{\mu(R_{k,j})}{(1-|c_{k,j}|)^{-\alpha+n-1}\om^*(c_{k,j})}\right)^{p\zeta-1}\chi_{R_{k,j}} \mu
\end{align}
and the operator $S_\zeta$ on $H(\Waw)$ be defined by
$$S_\zeta f(z)=\int_\BB G_\zeta f(w)\ol{G_{\ol{\zeta}}K_z^\Waw(w)}(1-|w|)^{2\varepsilon(1-p\zeta)}d\mu_\zeta(w),$$
here, $\chi_{R_{k,j}}=0$ if $\mu(R_{k,j})=0$ and  $\chi_{R_{k,j}}=1$ otherwise.
So, $\mu_\zeta$ is compactly supported and the series at the right side of (\ref{0608-1}) is a sum of finite terms.
Since $G_\zeta:H(\Waw)\to H(W_\gamma^\om)$ is uniformly bounded,  for any fixed $0<r<1$, there is a $C=C(r)$ such that
$$\sup_{0\leq |z|\leq r} |G_\zeta f(z)| \leq C(r)\|G_\zeta f\|_{H(W_\gamma^\om)}\lesssim C(r) \| f\|_{H(W_\alpha^\om)}.$$
So, $\|S_\zeta\|_{H(\Waw)\to H(\Waw)}$ is uniformly bounded on $\Lambda$.
If we can find $0<M_0, M_1<\infty$ such that
$$|S_\zeta|_\infty\leq M_0\,\,\mbox{ when }\,\,\Rp \zeta=0,
\,\,\mbox{ and }\,\,
|S_\zeta|_1\leq M_1\,\,\mbox{ when }\,\,\Rp \zeta=1,$$
by Theorem 13.1 in Chapter 3 of \cite{GiKm1969}, we have
\begin{align}\label{0608-2}
|S_\zeta|_p\leq M_0^{1-\frac{1}{p}}M_1^{\frac{1}{p}}\,\,\mbox{ when }\,\,\Rp \zeta=\frac{1}{p}.
\end{align}
Here, we should note that $|S_\zeta|_\infty=\|S_\zeta\|_{H(\Waw)\to H(\Waw)}$ and  $S_\frac{1}{p}=\HT_\mu$.

Next, we  estimate $\|S_\zeta\|_{H(\Waw)\to H(\Waw)}$ when $\Rp \zeta=0$. Let $f,g\in H(\Waw)$.
Since $\Rp \zeta=0$, then  $G_{\ol{\zeta}}g(0)=g(0)$ and
\begin{align*}
\int_\BB \Re g(z)\Re (G_{\ol{\zeta}} K_z^\Waw(w))\Waw(z)dV(z)
&=\int_\BB \Re (G_{\ol{\zeta}} g)(z)\Re ( K_z^\Waw(w))\Waw(z)dV(z)\\
&=\frac{1}{4}\left(  (G_{\ol{\zeta}} g)(w)-g(0)\right).
\end{align*}
Therefore, by Lemma \ref{0607-1} and Fubini's Theorem, we get
\begin{align*}
& \langle S_\zeta f,g \rangle_{H(\Waw)} \\
=&\om(\BB)\ol{g(0)}\int_\BB G_\zeta f(w)\ol{G_{\ol{\zeta}}K_0^\Waw(w)}(1-|w|)^{2\varepsilon(1-p\zeta)}d\mu_\zeta(w)\\
&+4\int_\BB \Re\left(\int_\BB G_\zeta f(w)\ol{G_{\ol{\zeta}}K_z^\Waw(w)}(1-|w|)^{2\varepsilon(1-p\zeta)}d\mu_\zeta(w)\right) \ol{\Re g(z)}\Waw(z) dV(z)  \\
=&\ol{g(0)}\int_\BB G_\zeta f(w)(1-|w|)^{2\varepsilon(1-p\zeta)}d\mu_\zeta(w)\\
&+4\int_\BB \left(\int_\BB G_\zeta f(w)\ol{\Re G_{\ol{\zeta}}K_z^\Waw(w)}(1-|w|)^{2\varepsilon(1-p\zeta)}d\mu_\zeta(w)\right) \ol{\Re g(z)}\Waw(z) dV(z)  \\
=&\ol{g(0)}\int_\BB G_\zeta f(w)(1-|w|)^{2\varepsilon(1-p\zeta)}d\mu_\zeta(w)\\
&+4\int_\BB \ol{\left(\int_\BB \Re G_{\ol{\zeta}}K_z^\Waw(w)\Re g(z)\Waw(z) dV(z)\right)}G_\zeta f(w)(1-|w|)^{2\varepsilon(1-p\zeta)}d\mu_\zeta(w)\\
&=\int_\BB G_\zeta f(w) \ol{G_{\ol{\zeta}}g(w)}(1-|w|)^{2\varepsilon(1-p\zeta)}d\mu_\zeta(w).
\end{align*}
Then, when $\Rp \zeta=0$,   H\"older's inequality induces that
{\small
 \begin{align}
 \left|\langle S_\zeta f,g \rangle_{H(\Waw)} \right|
& =\left|\int_\BB G_\zeta f(w) \ol{G_{\ol{\zeta}}g(w)}(1-|w|)^{2\varepsilon(1-p\zeta)}d\mu_\zeta(w) \right|\label{0619-1}\\
& \leq \left(\int_\BB |G_\zeta f(w)|^2(1-|w|)^{2\varepsilon}d|\mu_\zeta|(w)\right)^\frac{1}{2}
\left(\int_\BB |G_{\ol{\zeta}} g(w)|^2(1-|w|)^{2\varepsilon}d|\mu_\zeta|(w)\right)^\frac{1}{2}.\nonumber
 \end{align}
}

Since $G_\zeta:H(\Waw)\to H(W_\gamma^\om)$ is uniformly bounded and $H(W_\gamma^\om)\simeq A_{\om_{-(\alpha-2\varepsilon)-2}^*}^2$, when $\Rp \zeta =0$,
we have $\|G_\zeta f\|_{A_{\om_{-(\alpha-2\varepsilon)-2}^*}^2}\lesssim \|f\|_{H(\Waw)}$. Since $|G_\zeta f|^2$ is  subharmonicity and
$\om_{-(\alpha-2\varepsilon)-2}^*\in\R$, for any fixed $0<s<1$, by Lemma 1.23 in \cite{Zk2005} and Lemma \ref{1210-3}, we have
\begin{align*}
|G_\zeta f(w)|^2
&\leq \frac{\int_{\Delta(w,s)}|G_\zeta f(\eta)|^2\om_{-(\alpha-2\varepsilon)-2}^*(\eta)dV(\eta)}
{\int_{\Delta(w,s)}\om_{-(\alpha-2\varepsilon)-2}^*(\eta)dV(\eta)}\\
&\leq C(s,\om) \frac{\int_{\Delta(w,s)}|G_\zeta f(\eta)|^2\om_{-(\alpha-2\varepsilon)-2}^*(\eta)dV(\eta)}
{(1-|w|)^{2\varepsilon+n-\alpha-1}\om^*(w)}.
\end{align*}
Using Fubini's Theorem and Lemma \ref{0618-1}, we have
\begin{align*}
&\int_\BB |G_\zeta f(w)|^2(1-|w|)^{2\varepsilon}d|\mu_\zeta|(w)\\
\leq& C(s,\om) \int_\BB {\frac{\int_{\Delta(w,s)}|G_\zeta f(\eta)|^2\om_{-(\alpha-2\varepsilon)-2}^*(\eta)dV(\eta)}
{(1-|w|)^{n-\alpha-1}\om^*(w)}}d|\mu_\zeta|(w)  \\
=&C(s,\om) \int_\BB |G_\zeta f(\eta)|^2\om_{-(\alpha-2\varepsilon)-2}^*(\eta)dV(\eta)
\int_{\Delta(\eta,s)} \frac{d|\mu_{\zeta}|(w)}{(1-|w|)^{n-\alpha-1}\om^*(w)}\\
\lesssim &C(s,\om)  \left(\sup_{\eta\in\BB} \frac{|\mu_\zeta|(\Delta(\eta,s))}{(1-|\eta|)^{n-\alpha-1}\om^*(\eta)}\right)
\|G_\zeta f\|_{A_{\om_{-(\alpha-2\varepsilon)-2}^*}^2}^2  \\
\lesssim& C(s,\om)
\left(\sup_{\eta\in\BB} \frac{|\mu_\zeta|(\Delta(\eta,s))}{(1-|\eta|)^{n-\alpha-1}\om^*(\eta)}\right)
\| f\|_{H(\Waw)}^2\\
\lesssim & C(s,\om)
\left(\sup_{R_{k,j}\in\Upsilon} \frac{|\mu_\zeta|(R_{k,j})}{(1-|c_{k,j}|)^{n-\alpha-1}\om^*(c_{k,j})}\right)
\| f\|_{H(\Waw)}^2.
\end{align*}
Similarly, we have
$$
\int_\BB |G_{\ol{\zeta}} g(w)|^2(1-|w|)^{2\varepsilon}d|\mu_\zeta|(w)
\lesssim C(s,\om) \left(\sup_{R_{k,j}\in\Upsilon} \frac{|\mu_\zeta|(R_{k,j})}{(1-|c_{k,j}|)^{n-\alpha-1}\om^*(c_{k,j})}\right)\| g\|_{H(\Waw)}^2.
$$
Therefore, when $\Rp \zeta=0$,
$$
\left|\langle S_\zeta f,g \rangle_{H(\Waw)} \right|
\lesssim C(s,\om) \left(\sup_{R_{k,j}\in\Upsilon} \frac{|\mu_\zeta|(R_{k,j})}{(1-|c_{k,j}|)^{n-\alpha-1}\om^*(c_{k,j})}\right)
\| f\|_{H(\Waw)}  \| g\|_{H(\Waw)}.
$$
At the same time,  since
\begin{align}
 \frac{|\mu_\zeta|(R_{k,j})}{(1-|c_{k,j}|)^{n-\alpha-1}\om^*(c_{k,j})}
 &=\left(\frac{\mu(R_{k,j})}{(1-|c_{k,j}|)^{n-\alpha-1}\om^*(c_{k,j})}\right)^{p\Rp \zeta-1}
       \frac{\mu(R_{k,j})}{(1-|c_{k,j}|)^{n-\alpha-1}\om^*(c_{k,j})}  \nonumber\\
 &=\left(\frac{\mu(R_{k,j})}{(1-|c_{k,j}|)^{n-\alpha-1}\om^*(c_{k,j})}\right)^{p\Rp \zeta},  \label{0618-2}
\end{align}
when $\Rp \zeta=0$, there is a constant $0<M_0=M_0(s,\om)<\infty$, such that, for all $f,g\in H(\Waw)$, we have
$$
\left|\langle S_\zeta f,g \rangle_{H(\Waw)} \right|\leq M_0 \| f\|_{H(\Waw)}  \| g\|_{H(\Waw)}.
$$

Then, we will estimate the upper bound of $|S_\zeta|_1$ for $\Rp\zeta=1$. Assume $\Rp \zeta=1$.
Since   $G_\zeta$ and $G_{\ol{\zeta}}$ are bounded and invertible from $H(\Waw)$ to $H(W_\gamma^\om)$,
we can choose bounded and invertible operators $A$ and $B$ on  $H(\Waw)$ such that
$G_\zeta A$ and $G_{\ol{\zeta}}B$ are unitary operators from $H(\Waw)$ to $H(W_\gamma^\om)$.
Here we should note that $\gamma=\alpha+2\varepsilon(p-1)$ when $\Rp \zeta=1$.
Let $T=B^*S_\zeta A$. By (\ref{0619-1}), for all $f,g\in H(\Waw)$, we have
\begin{align*}
\langle T f,g \rangle_{H(\Waw)}&= \langle S_\zeta A f,Bg \rangle_{H(\Waw)}     \\
& =\int_\BB G_\zeta Af(w) \ol{G_{\ol{\zeta}}Bg(w)}(1-|w|)^{2\varepsilon(1-p\zeta)}d\mu_\zeta(w).
\end{align*}
Now, let $\{f_k\}$ and $\{g_k\}$ be orthogonal  sets on $H(\Waw)$ and
$$e_k=G_\zeta Af_k, \,\,e_k^\p=G_{\ol{\zeta}}Bg_k.$$
Since   $G_\zeta A$ and $G_{\ol{\zeta}}B$ are unitary operators from $H(\Waw)$ to $H(W_\gamma^\om)$,
$\{e_k\}$ and $e_k^\p$ are orthogonal  sets in $H(W_\gamma^\om)$. By   Lemma \ref{0607-1}, we get
$$\sum_k |e_k(z)|^2\leq \|K_z^{W_\gamma^\om}\|_{H(W_\gamma^\om)}^2\approx \frac{(1-|z|)^{\gamma-n}}{\hat{\om}(z)},$$
and
$$\sum_k |e_k^\p(z)|^2\leq \|K_z^{W_\gamma^\om}\|_{H(W_\gamma^\om)}^2\approx \frac{(1-|z|)^{\gamma-n}}{\hat{\om}(z)}.$$
Therefore, by Cauchy-Schwarz's inequality, we obtain
\begin{align*}
\sum_k\left|\langle T f_k,g_k \rangle_{H(\Waw)} \right|
&=\sum_k\left|\int_\BB e_k(w) \ol{e_k^\p(w)}(1-|w|)^{2\varepsilon(1-p\zeta)}d\mu_\zeta(w)\right|    \\
&\leq \int_\BB \left(\sum_k |e_k(w)|^2\right)^\frac{1}{2}  \left(\sum_k |e_k^\p(w)|^2\right)^\frac{1}{2}   (1-|w|)^{2\varepsilon(1-p\Rp\zeta)}d|\mu_\zeta|(w)  \\
&\leq \int_\BB \frac{(1-|w|)^{\gamma-n}}{\hat{\om}(w)}   (1-|w|)^{2\varepsilon(1-p)}d|\mu_\zeta|(w)\\
&=\sum_{R_{k,j}\in\Upsilon} \int_{R_{k,j}} \frac{(1-|w|)^{\alpha-n}}{\hat{\om}(w)}  d|\mu_\zeta|(w)   \\
&\approx \sum_{R_{k,j}\in\Upsilon}  \frac{(1-|c_{k,j}|)^{\alpha-n+1}}{\om^*(c_{k,j})}\int_{R_{k,j}}   d|\mu_\zeta|(w) \\
&=\sum_{R_{k,j}\in\Upsilon} \left( \frac{\mu(R_{k,j})}{(1-|c_{k,j}|)^{-\alpha+n-1}\om^*(c_{k,j})}\right)^p.
\end{align*}
By Theorem 1.27 in \cite{zhu}, we have
$$|T|_1\lesssim \sum_{R_{k,j}\in\Upsilon} \left( \frac{\mu(R_{k,j})}{(1-|c_{k,j}|)^{-\alpha+n-1}\om^*(c_{k,j})}\right)^p.$$
Lemma 1.36 in \cite{zhu} implies that
\begin{align*}
|S_\zeta|_1
&\leq \|(B^*)^{-1}\|_{H(\Waw)\to H(\Waw)} |T|_1  \cdot \|A^{-1}\|_{H(\Waw)\to H(\Waw)}  \\
&\lesssim \sum_{R_{k,j}\in\Upsilon} \left( \frac{\mu(R_{k,j})}{(1-|c_{k,j}|)^{-\alpha+n-1}\om^*(c_{k,j})}\right)^p.
\end{align*}
Using (\ref{0608-2}) and  $S_\frac{1}{p}=\HT_\mu$, we have
$$|\HT_\mu|_p\lesssim \left(\sum_{R_{k,j}\in\Upsilon} \left( \frac{\mu(R_{k,j})}{(1-|c_{k,j}|)^{-\alpha+n-1}\om^*(c_{k,j})}\right)^p\right)^\frac{1}{p}$$
for all $\mu$ with the compactly support subsets and $p>1$.

{\bf The second step.} Suppose $\mu$ is a positive Borel measure such that $M_\mu<\infty$. For $k=1,2,\cdots$ and all Borel set $E\subset\BB$, let
$$\mu_k(E)=\mu(E\cap (1-\frac{1}{k+1})\BB).$$
Then $\mu_k$ is compactly supported and $M_{\mu_k}\leq M_\mu<\infty.$ So, we have $|\HT_{\mu_k}|_p^p\lesssim M_\mu$ and
$$\|\HT_{\mu_k}\|_{H(\Waw)\to H(\Waw)}^p=|\HT_{\mu_k}|_{\infty}^p\leq |T_{\mu_k}|_p^p\lesssim M_\mu.$$

Consider the identity operator $Id_k:H(\Waw)\to L^2(\BB,d\mu_k)$.
Here, we should note that $Id_k$ can also be taken as a operator from $H(\Waw)$ to $L^2(\BB,d\mu)$, which is defined by
$$Id_k f(z)=\chi_{(1-\frac{1}{k+1}\BB)}(z)f(z),\,\,\mbox{for all }f\in H(\Waw).$$
So, we do not distinguish the operators  $Id_k:H(\Waw)\to L^2(\BB,d\mu_k)$ and $Id_k:H(\Waw)\to L^2(\BB,d\mu)$. For any $f,g\in H(\Waw)$, using Fubini's Theorem,
\begin{align}
\langle \HT_{\mu_k} f,g \rangle_{H(\Waw)}&=\HT_{\mu_k} f(0)\ol{g(0)}\om(\BB)+4\int_\BB \Re \HT_{\mu_k} f(z)\ol{\Re g(z)} \Waw(z)dV(z)   \label{0626-1}\\
&=\int_\BB f(w)\ol{g(w)}d\mu_k(w)=\langle Id_k f,Id_k g\rangle_{L^2(\BB,d\mu_k)}=\langle (Id_k)^*Id_k f,g\rangle_{H(\Waw)}.\nonumber
\end{align}
Then we have $\HT_{\mu_k}=(Id_k)^* Id_k$ and
$$\|(Id_k)^* Id_k\|_{H(\Waw)\to H(\Waw)}=\|\HT_{\mu_k}\|_{H(\Waw)\to H(\Waw)}\lesssim M_\mu^\frac{1}{p}.$$
So, $\| Id_k\|_{H(\Waw)\to L^2(\BB,\mu_k)}\lesssim M_\mu^\frac{1}{2p}$.
Let $Id$ be the  identity operator from $H(\Waw)$  to $L^2(\BB,\mu)$.
By Banach-Steinhaus's Theorem, we have
$$\lim_{k\to\infty}\|Id-Id_k\|_{H(\Waw)\to L^2(\BB,d\mu)}=0,\,\,\mbox{and}\,\,\| Id\|_{H(\Waw)\to L^2(\BB,\mu)}\lesssim M_\mu^\frac{1}{2p}.$$
Then,
$$\lim_{k\to\infty}\|(Id)^*-(Id_k)^*\|_{ L^2(\BB,d\mu)\to H(\Waw)}=0.$$
Thus,
$$\lim_{k\to\infty}\|(Id)^*Id-(Id_k)^*Id_k\|_{H(\Waw)\to H(\Waw)}=0.$$
Similar to get (\ref{0626-1}), for any $f,g\in H(\Waw)$ we have
\begin{align}\label{0626-2}
\langle \HT_\mu f,g \rangle_{H(\Waw)}
&=\int_\BB f(w)\ol{g(w)}d\mu(w).
\end{align}
Therefore,
$$\lim_{k\to\infty}\|\HT_\mu-\HT_{\mu_k}\|_{H(\Waw)\to H(\Waw)}
=\lim_{k\to\infty}\|(Id)^* Id-(Id_k)^*Id_k\|_{H(\Waw)\to H(\Waw)}=0,$$
 and
$$\|\HT_\mu\|_{H(\Waw)\to H(\Waw)}\leq \|(Id)^* Id\|_{H(\Waw)\to L^2(\BB,\mu)}\lesssim M_\mu^\frac{1}{p}.$$
Since $\HT_{\mu_k}(k=1,2,\cdots)$ are compact operators on $H(\Waw)$,  $\HT_\mu$ is compact on $H(\Waw)$.

If $T$ is a bounded operator on $H(\Waw)$, then
\begin{align*}
\|\HT_\mu-T\|_{H(\Waw)\to H(\Waw)}
&=\lim_{k\to\infty}\|\HT_{\mu_k}-T\|_{H(\Waw)\to H(\Waw)}.
\end{align*}
So, $\lambda_j(\HT_\mu)=\lim_{k\to\infty}\lambda_j(\HT_{\mu_k}).$  Fatou's Lemma deduces
\begin{align*}
|\HT_\mu|_p^p
&=\sum_{j=0}^\infty (\lambda_j(\HT_\mu))^p=\sum_{j=0}^\infty \lim_{k\to\infty}(\lambda_j(\HT_{\mu_k}))^p  \\
&\leq \liminf_{k\to\infty}\sum_{j=0}^\infty (\lambda_j(\HT_{\mu_k}))^p  \leq \limsup_{k\to\infty}|\HT_{\mu_k}|_p^p\lesssim M_\mu.
\end{align*}
We finish the proof of {\it (ii)}$\Rightarrow${\it (i)}.

  {\it (i)}$\Rightarrow${\it (ii)}.  Suppose  $\HT_\mu\in \mcS_p(H(\Waw))$. Let  $\{a_j\}$ be $s$ pseudo-hyperbolic separated and $\BB=\cup_{j=1}^\infty \Delta(a,t)$ for some   small enough  $t$ and $0<s<t<1$. Suppose there is no origin in $\{a_j\}$.

Let $\{e_j\}$ be an orthogonal  set in $H(\Waw)$ and $\mathscr{E}$ denote the subspace which is generated by $\{e_j\}$ and equipped with the norm of $H(\Waw)$.
Consider the linear operator $\widetilde{J}:\mathscr{E}\to H(\Waw)$, which is defined by
$$\widetilde{J}(e_j)(z)=\frac{1}{\|B_{a_j}^\Wav\|_{A_\Wav^2}}\int_0^1 \left(B_{a_j}^\Wav(tz)-B_{a_j}^\Wav(0)\right)\frac{dt}{t}.$$
For any $\{c_j\}\in l^2$ and $f\in H(\Waw)$, by Cauchy-Schwarz's inequality, we have
\begin{align*}
\left|\left\langle \widetilde{J}\left(\sum_j c_je_j\right),f  \right\rangle_{H(\Waw)}\right|
&=\left|\sum_j \frac{c_j}{\|B_{a_j}^\Wav\|_{A_\Wav^2}}\left\langle B_{a_j}^\Wav-B_{a_j}^\Wav(0),\Re f \right\rangle_{A_\Waw^2}\right|   \\
&=\left|\sum_j \frac{c_j}{\|B_{a_j}^\Wav\|_{A_\Wav^2}}\left\langle B_{a_j}^\Wav,\Re f \right\rangle_{A_\Waw^2}\right|
=\left|\sum_j \frac{c_j\ol{\Re f(a_j)}}{\|B_{a_j}^\Wav\|_{A_\Wav^2}}\right|   \\
&\leq \|\{c_j\}\|_{l^2} \left(\sum_{j}\frac{|\Re f(a_j)|^2}{\|B_{a_j}^\Wav\|_{A_\Wav^2}^2}\right)^\frac{1}{2}.
\end{align*}
By (\ref{0412-1}) and $\Wav\in\R$,  we have
$$
\|B_{a_j}^\Wav\|_{A_\Wav^2}^2 \approx \int_0^{|a_j|}\frac{1}{\widehat{\Wav}(t)(1-t)^{n+1}}dt
\approx \frac{1}{(1-|a_j|)^{n+1}\Wav(a_j)}.
$$
Therefore, by Lemma 2.24 in \cite{Zk2005} we get
\begin{align*}
\left|\left\langle \widetilde{J}\left(\sum_j c_je_j\right),f  \right\rangle_{H(\Waw)}\right|
&\lesssim \|\{c_j\}\|_{l^2} \left(\sum_{j}(1-|a_j|)^{n+1}\Wav(a_j)|\Re f(a_j)|^2\right)^\frac{1}{2}\\
&\approx  \|\{c_j\}\|_{l^2} \left(\sum_{j}\Wav(a_j)\int_{\Delta(a_j,\frac{s}{2})}|\Re f(\xi)|^2dV(\xi)\right)^\frac{1}{2}\\
&\approx \|\{c_j\}\|_{l^2} \left(\sum_{j}\int_{\Delta(a_j,\frac{s}{2})}|\Re f(\xi)|^2\Wav(\xi)dV(\xi)\right)^\frac{1}{2}\\
&\leq \|\{c_j\}\|_{l^2} \|f\|_{H(\Waw)}.
\end{align*}
So,  $\widetilde{J}:\mathscr{E}\to H(\Waw)$ is bounded.

Let $P$ be the orthogonal projection operator from $H(\Waw)$  to $\mathscr{E}$ and $J=\widetilde{J} P$.
Then $J$ is bounded on $H(\Waw)$.
By \cite[p. 27]{zhu}, $\mathcal{S}_p(H(\Waw))$ is a two-sided ideal in the space of bounded linear operators on $H(\Waw)$.
So, $J^*\HT_\mu J\in\mathcal{S}_p(H(\Waw))$ and $|J^*\HT_\mu J|_p\lesssim |\HT_\mu|_p.$
Theorem 1.27 in \cite{zhu} deduce that
\begin{align*}
\sum_j |\langle  (\HT_\mu J)e_j,Je_j\rangle_{H(\Waw)}|^p=\sum_j |\langle  (J^*\HT_\mu J)e_j,e_j\rangle_{H(\Waw)}|^p<|\HT_\mu|_p^p.
\end{align*}
By (\ref{0626-2}), we get $\sum\limits_{j}\|Je_j\|_{L^2(\BB,d\mu)}^2<|\HT_\mu|_p^p.$ By (\ref{0412-1}) and Lemma \ref{0627-1}, we have
\begin{align*}
\|Je_j\|_{L^2(\BB,d\mu)}^{2p}
&=\|\widetilde{J}e_j\|_{L^2(\BB,d\mu)}^{2p}= \frac{1}{\|B_{a_j}^\Waw\|_{A_\Waw^2}^{2p}}\int_\BB  \left|\int_0^1 \left(B_{a_j}^\Wav(tz)-B_{a_j}^\Wav(0)\right)\frac{dt}{t}\right|^{2p}d\mu(z) \\
&\gtrsim {(1-|a_j|)^{(n+1)p}\Waw(a_j)^p}\int_{\Delta(a_i,t)}  \left|\int_0^1
\left(B_{a_j}^\Wav(tz)-B_{a_j}^\Wav(0)\right)\frac{dt}{t}\right|^{2p}d\mu(z) \\
&\gtrsim \left(\frac{\mu(\Delta(a_j,t))}{(1-|a_j|)^{n-1}\Waw(a_j)}\right)^p.
\end{align*}
If $E,F\subset\BB$, let $\chi(E,F)=1$ when $E\cap F\neq \O$ and  $\chi(E,F)=0$ otherwise.
By  Lemma \ref{0618-1}, we have
\begin{align}
&\sum_{R_{k,j}\in\Upsilon}\left(\frac{\mu(R_{k,j})}{(1-|c_{k,j}|)^{-\alpha+n-1}\om^*(c_{k,j})}\right)^p
 \lesssim \sum_{R_{k,j}\in\Upsilon}\left(\sum_{i}\frac{\mu(\Delta(a_i,t))\chi(R_{k,j},\Delta(a_i,t))}{(1-|a_i|)^{n-1}\Waw(a_i)}\right)^p  \nonumber\\
&\approx  \sum_{i}\sum_{R_{k,j}\in\Upsilon}\left(\frac{\mu(\Delta(a_i,t))\chi(R_{k,j},\Delta(a_i,t))}{(1-|a_i|)^{n-1}\Waw(a_i)}\right)^p  \approx  \sum_{i}\left(\sum_{R_{k,j}\in\Upsilon}\frac{\mu(\Delta(a_i,t))\chi(R_{k,j},\Delta(a_i,t))}{(1-|a_i|)^{n-1}\Waw(a_i)}\right)^p  \nonumber\\
&\lesssim   \sum_{i}\left(\frac{\mu(\Delta(a_i,t))}{(1-|a_i|)^{n-1}\Waw(a_i)}\right)^p  \lesssim \sum_{i}\|\widetilde{J}e_j\|_{L^2(\BB,d\mu)}^{2p} \lesssim |\HT_\mu|_p^p.  \nonumber
\end{align}
So, {\it (ii)} holds.

({\it ii})$\Rightarrow$ ({\it iii}). Let $\{a_i\}_{i=1}^\infty$ be a $r$-lattice.
By Lemma \ref{0618-1}, we get
\begin{align*}
\|\widehat{\mu_{r,\alpha}}\|_{L^p_\lambda}^p
&\leq \sum_{i=1}^\infty \int_{D(a_i,5r)}\left(\frac{\mu(D(z,r))}{(1-|z|)^{-\alpha+n-1}\om^*(z)}\right)^p d\lambda(z)
\lesssim \sum_{i=1}^\infty\left(\frac{\mu(D(a_i,6r))}{(1-|a_i|)^{-\alpha+n-1}\om^*(a_i)}\right)^p  \\
&\lesssim \sum_{i=1}^\infty\sum_{R_{k,j}\in\Upsilon}  \left(\frac{\mu(R_{k,j})}{(1-|c_{k,j}|)^{-\alpha+n-1}\om^*(c_{k,j})}\right)^p\chi(R_{k,j},D(a_i,6r))\\
&\lesssim \sum_{R_{k,j}\in\Upsilon}  \left(\frac{\mu(R_{k,j})}{(1-|c_{k,j}|)^{-\alpha+n-1}\om^*(c_{k,j})}\right)^p,
\end{align*}
as desired.

 {\it (iii)} $\Rightarrow$  {\it (ii)}. Let $\{a_i\}$ be a $\frac{r}{6}$-lattice. For any $z\in D(a_i,\frac{r}{60})$, we have $D(a_i,\frac{5r}{6})\subset D(z,r)$.
By Lemma \ref{0618-1}, we obtain
\begin{align*}
\|\widehat{\mu_{r,\alpha}}\|_{L^p_\lambda}^p
&\geq \sum_{i=1}^\infty \int_{D(a_i,\frac{r}{60})}\left(\frac{\mu(D(z,r))}{(1-|z|)^{-\alpha+n-1}\om^*(z)}\right)^p d\lambda(z)
\geq \sum_{i=1}^\infty\left(\frac{\mu(D(a_i,\frac{5r}{6}))}{(1-|a_i|)^{-\alpha+n-1}\om^*(a_i)}\right)^p  \\
&\approx \sum_{i=1}^\infty\sum_{R_{k,j}\in\Upsilon}  \left(\frac{\mu(D(a_i,\frac{5r}{6}))}{(1-|a_i|)^{-\alpha+n-1}\om^*(a_i)}\right)^p\chi(R_{k,j},D(a_i,\frac{5r}{6}))\\
&\approx \sum_{R_{k,j}\in\Upsilon}  \left(\sum_{i=1}^\infty\frac{\mu(D(a_i,\frac{5r}{6}))}{(1-|a_i|)^{-\alpha+n-1}\om^*(a_i)}\chi(R_{k,j},D(a_i,\frac{5r}{6}))\right)^p\\
&\gtrsim \sum_{R_{k,j}\in\Upsilon}  \left(\frac{\mu(R_{k,j})}{(1-|c_{k,j}|)^{-\alpha+n-1}\om^*(c_{k,j})}\right)^p.
\end{align*}
 The proof is complete.
\end{proof}

\begin{Lemma}\label{0916-1}
 Suppose $\om\in\hD$, $-\infty<\alpha<1$, $r\in(0,1)$, $s$ is a large enough integer, and $\{e_k\}_{k=1}^\infty$ be a orthonormal basis of a Hilbert space $H(\Waw)$.
If $\{b_k\}_{k=1}^\infty\subset\BB\backslash\{0\}$ is a $r$-lattice ordered by increasing module, and
$$Je_k=\left(\frac{(1-|b_k|^2)^{s}}{\Waw(b_k)(1-\langle z,b_k\rangle)^{n-1+s}}\right)^\frac{1}{2},\,\,k=1,2,\cdots.$$
Then $J:H(\Waw)\to H(\Waw)$ is bounded and  onto.
\end{Lemma}
\begin{proof}
Firstly, we prove $J$ is surjective.
Let
$$
 \Psi(t)=\left\{
\begin{array}{cc}
  \Waw(\frac{1}{2}), & t\in [0,\frac{1}{2}],  \\
  \Waw(t), & t\in[\frac{1}{2},1).
\end{array}
\right.
$$
By (\ref{0605-1}), we have
\begin{itemize}
  \item $\Psi\in \R$;\,\, $\Psi(b_k)\approx \Waw(b_k)$ for $k=1,2,\cdots$;
  \item $\|f\|_{H(\Waw)}^2\approx |f(0)|^2+\int_\BB |\Re f(z)|^2\Psi(z) dV(z)$.
\end{itemize}
Let the operators $R^{\beta,t}$ and $R_{\beta,t}$ be defined as (1.33) and (1.34) in \cite{Zk2005}.
Assume the homogeneous expansion of $f$ is $f(z)=\sum\limits_{k=0}^\infty f_k(z)$. Then, by Strling's formula, for any given $\beta>0$, we have
\begin{align*}
\|f\|_{H(\Waw)}^2
&\approx  |f(0)|^2+\int_\BB |\Re f(z)|^2\Psi(z) dV(z)\\
&= |f(0)|^2+\sum_{k=1}^\infty k^2\int_\BB |f_k(z)|^2 \Psi(z)dV(z)
\approx \|R^{\beta,1}f\|_{A_{\Psi}^2}^2.
\end{align*}
Here,
$$\|f\|_{A_\Psi^2}^2=\int_\BB |f(z)|^2 \Psi(z)dV(z).$$
By Proposition 2.4 in \cite{LxHz2010amsc} (also see Theorem 3.2 in \cite{ZxXlFhLj2014amsc}), $R^{\beta,1}f\in A_{\Psi}^2$ if and only if
there exists $\{c_k\}_{k=1}^\infty\in l^2$ such that $\|R^{\beta,1}f\|_{A_{\Psi}^2}\lesssim \|\{c_k\}\|_{l^2}$ and
$$R^{\beta,1}f(z)=\sum_{k=1}^\infty c_k\left(\frac{(1-|b_k|^2)^s}{\Psi(b_k)(1-\langle z,b_k\rangle)^{n+1+s}}\right)^\frac{1}{2}.$$
Let $\beta=-\frac{n+3}{2}+\frac{s}{2}$. By Proposition 1.14 in \cite{Zk2005}, we have
\begin{align*}
f(z)
&=R_{\beta,1}\left(\sum_{k=1}^\infty c_k\left(\frac{(1-|b_k|^2)^s}{\Psi(b_k)(1-\langle z,b_k\rangle)^{n+1+s}}\right)^\frac{1}{2}\right)  \\
&=\sum_{k=1}^\infty c_k\left(\frac{(1-|b_k|^2)^s}{\Psi(b_k)(1-\langle z,b_k\rangle)^{n-1+s}}\right)^\frac{1}{2}\\
&=\sum_{k=1}^\infty \frac{c_k\sqrt{\Waw(b_k)}}{\sqrt{\Psi(b_k)}}\left(\frac{(1-|b_k|^2)^s}{\Waw(b_k)(1-\langle z,b_k\rangle)^{n-1+s}}\right)^\frac{1}{2}.
\end{align*}
Let $d_k=\frac{c_k\sqrt{\Waw(b_k)}}{\sqrt{\Psi(b_k)}}$. Then we have $\|\{d_k\}\|_{l^2}\approx \|\{c_k\}\|_{l^2}$ and $f=J(\sum_{k=1}^\infty d_ke_k)$. So, $J$ is surjective.

Next, we prove $J$ is bounded on $H(\Waw)$. Let $g(z)=\sum\limits_{k=1}^\infty d_ke_k(z)$. Then $\|g\|_{H(\Waw)}=\|\{d_k\}\|_{l^2}$.
Let $\beta=-\frac{n+3}{2}+\frac{s}{2}$. By Proposition 1.14 in \cite{Zk2005} and Proposition 2.4 in \cite{LxHz2010amsc} (also see Theorem 3.2 in \cite{ZxXlFhLj2014amsc}), we have

\begin{align*}
\|Jg\|_{H(\Waw)}
&= \left\|  \sum_{k=1}^\infty \frac{d_k\sqrt{\Psi(b_k)}}{\sqrt{\Waw(b_k)}}\left(\frac{(1-|b_k|^2)^s}{\Psi(b_k)(1-\langle z,b_k\rangle)^{n-1+s}}\right)^\frac{1}{2}  \right\|_{H(\Waw)}   \\
&\approx \left\|  \sum_{k=1}^\infty \frac{d_k\sqrt{\Psi(b_k)}}{\sqrt{\Waw(b_k)}}R^{\beta,1}\left(\frac{(1-|b_k|^2)^s}{\Psi(b_k)(1-\langle z,b_k\rangle)^{n-1+s}}\right)^\frac{1}{2}  \right\|_{A_{\Psi}^2}\\
&= \left\|  \sum_{k=1}^\infty \frac{d_k\sqrt{\Psi(b_k)}}{\sqrt{\Waw(b_k)}}\left(\frac{(1-|b_k|^2)^s}{\Psi(b_k)(1-\langle z,b_k\rangle)^{n+1+s}}\right)^\frac{1}{2}  \right\|_{A_{\Psi}^2}\\
&\lesssim \left\|\left\{\frac{d_k\sqrt{\Psi(b_k)}}{\sqrt{\Waw(b_k)}}\right\}\right\|_{l^2}
\approx \|\{d_k\}\|_{l^2}=\|f\|_{H(\Waw)}.
\end{align*}
So, $J:H(\Waw)\to H(\Waw)$ is bounded. The proof is complete.
\end{proof}

\begin{Lemma}\label{0916-2} Suppose $\varphi$ is regular and continuous, $0<p<\infty$, $s$ is large enough. If there exist $-1<a<\infty$ and $\delta\in(0,1)$ such that
\begin{align*}
n-1+a>0,
\,\,\mbox{ and }\,\,\frac{{\varphi}(t)}{(1-t)^a}\searrow 0,\,\,\mbox{ when }\,\,\delta\leq t<1 ,
\end{align*}
then
\begin{align}\label{0918-1}
I(z)=\int_\BB \frac{(1-|w|^2)^{sp-(n+1)}dV(w)}{\varphi(w)^p|1-\langle z,w\rangle|^{(n-1+s)p}}\lesssim \frac{1}{\varphi(z)^p(1-|z|^2)^{p(n-1)}},
\end{align}
and for any given $\varepsilon>0$, there exists $0<r<1$, such that for all $z\in\BB$,
\begin{align}\label{0917-1}
 I(z,r)=\int_{\BB\backslash \Delta(z,r)}\frac{(1-|w|^2)^{sp-(n+1)}dV(w)}{\varphi(w)^p|1-\langle z,w\rangle|^{(n+1+s)p}}\leq \frac{\varepsilon}{\varphi(z)^p(1-|z|^2)^{p(n-1)}}.
 \end{align}
\end{Lemma}
\begin{proof}By Lemma \ref{0507-1}, there exists $a<b<+\infty$  such that $\frac{{\varphi}(t)}{(1-t)^b} \nearrow\infty.$
Without loss of generality, let $\delta=0$.

Let
$$I(z)=\left(\int_{|w|\leq |z|}+\int_{|z|<|w|<1} \right)\frac{(1-|w|^2)^{sp-(n+1)}dV(w)}{\varphi(w)^p|1-\langle z,w\rangle|^{(n-1+s)p}} =I_1(z)+I_2(z).$$
When $n-1+a>0$, by Theorem 1.12 in \cite{Zk2005}, we have
$$
I_1(z)\lesssim \frac{(1-|z|)^{pa}}{\varphi(z)^p}\int_{|w|\leq |z|}\frac{(1-|w|^2)^{sp-(n+1)-pa}dV(w)}{|1-\langle z,w\rangle|^{(n-1+s)p}}
\lesssim \frac{1}{\varphi(z)^p(1-|z|)^{(n-1)p}},
$$
and
$$I_2(z)\leq \frac{(1-|z|)^{pb}}{\varphi(z)^p}\int_{|z|\leq |w|<1}\frac{(1-|w|^2)^{sp-(n+1)-pb}dV(w)}{|1-\langle z,w\rangle|^{(n-1+s)p}}
\lesssim \frac{1}{\varphi(z)^p(1-|z|)^{(n-1)p}}.$$
Hence, (\ref{0918-1}) holds.

In order to obtain (\ref{0917-1}), we prove that, if $n<c<\beta$, for any given $\varepsilon>0$, there exists $r\in(0,1)$, such that, for all $z\in\BB$,
\begin{align}\label{0917-2}
\int_{\BB\backslash\Delta(z,r)}\frac{(1-|w|^2)^{c-(n+1)}}{|1-\langle z,w\rangle|^\beta}dV(w)\lesssim \frac{\varepsilon}{(1-|z|^2)^{\beta-c}}.
\end{align}
Letting $\eta=\vp_z(w)$ and $|\eta|=t$, by Lemmas 1.2 and 1.3  in \cite{Zk2005}, we have
\begin{align*}
\int_{\BB\backslash\Delta(z,r)}\frac{(1-|w|^2)^{c-(n+1)}}{|1-\langle z,w\rangle|^\beta}dV(w)
&=\int_{r\leq |\eta|<1}\frac{(1-|\vp_z(\eta)|^2)^{c}}{|1-\langle \vp_z(0),\vp_z(\eta)\rangle|^\beta}\frac{dV(\eta)}{(1-|\eta|^2)^{n+1}}\\
&=\int_{r\leq |\eta|<1} \frac{(1-|z|^2)^{c-\beta}(1-|\eta|^2)^{c-(n+1)}}{|1-\langle z,\eta\rangle|^{2c-\beta}}dV(\eta)  \\
&\approx \frac{1}{(1-|z|^2)^{\beta-c}} \int_r^1 (1-t)^{c-(n+1)}dt\int_\SS\frac{d\sigma(\xi)}{|1-\langle tz,\xi\rangle|^{2c-\beta}}.
\end{align*}
So, when $2c-\beta\leq 0$, for any given $\varepsilon>0$, (\ref{0917-2}) holds for some $r\in(0,1)$.

When $2c-\beta>0$, let $\lambda=\frac{2c-\beta}{2}$.  By the proof of Theorem 1.12 in \cite{Zk2005}  and  Stirling's estimate, we get
\begin{align*}
\int_\SS\frac{d\sigma(\xi)}{|1-\langle tz,\xi\rangle|^{2c-\beta}}
&= \sum_{k=0}^\infty \left(\frac{\Gamma(k+\lambda)}{k!\Gamma(\lambda)}\right)^2\int_\SS |\langle tz,\xi \rangle|^{2k}d\sigma(\xi)  \\
&= \sum_{k=0}^\infty \left(\frac{\Gamma(k+\lambda)}{k!\Gamma(\lambda)}\right)^2 \frac{(n-1)!k!}{(n-1+k)!} |tz|^{2k}
\approx  \sum_{k=0}^\infty k^{2\lambda-n-1}|tz|^{2k}.
\end{align*}
Hence,
\begin{align*}
\int_r^1 (1-t)^{c-(n+1)}dt\int_\SS\frac{d\sigma(\xi)}{|1-\langle tz,\xi\rangle|^{2c-\beta}}
\approx \sum_{k=0}^\infty k^{2\lambda-n-1}|z|^{2k}\int_r^1 (1-t)^{c-(n+1)}t^{2k}dt.
\end{align*}
As $N\to \infty$,
\begin{align*}
\sum_{k=N+1}^\infty k^{2\lambda-n-1}|z|^{2k}\int_r^1 (1-t)^{c-(n+1)}t^{2k}dt
&\leq  \sum_{k=N+1}^\infty k^{2\lambda-n-1}\int_0^1 (1-t)^{c-(n+1)}t^{2k}dt\\
&\approx \sum_{k=N+1}^\infty k^{c-\beta-1}\to 0.
\end{align*}
Therefore, for any given $\varepsilon>0$, there exists $N_*\in\NN$, such that
\begin{align*}
\int_r^1 (1-t)^{c-(n+1)}dt\int_\SS\frac{d\sigma(\xi)}{|1-\langle tz,\xi\rangle|^{2c-\beta}}
\lesssim \sum_{k=0}^{N_*} k^{2\lambda-n-1}\int_r^1 (1-t)^{c-(n+1)}t^{2k}dt +\varepsilon.
\end{align*}
So, for any given $\varepsilon>0$, there exists $r\in(0,1)$, such that for all $z\in\BB$, (\ref{0917-2}) holds.

Let $\BB_{|z|}=\{w\in\BB:|w|<|z|\}$. By (\ref{0917-2}), for any given $\varepsilon>0$, there exists $r\in(0,1)$ such that
\begin{align*}
I(z,r)&=\left(\int_{(\BB\backslash \Delta(z,r))\cap \BB_{|z|}} + \int_{\BB\backslash (\Delta(z,r) \cup \BB_{|z|})}\right)
\frac{(1-|w|^2)^{sp-(n+1)}dV(w)}{\varphi(w)^p|1-\langle z,w\rangle|^{(n-1+s)p}}\\
&\lesssim  \frac{\varepsilon}{\varphi(z)^p(1-|z|^2)^{p(n-1)}}.
\end{align*}
The proof is complete.
\end{proof}

\begin{Theorem}\label{0919-1} Let $0< p<1$, $-\infty<\alpha<1$, $\om\in\hD$ and $\mu$ be a positive Borel measure on $\BB$.  Let
$$\widehat{\mu_{r,\alpha}}=\frac{\mu(D(z,r))}{(1-|z|)^{-\alpha+n-1}\om^*(z)},\,\,\,\,d\lambda(z)=\frac{dV(z)}{(1-|z|^2)^{n+1}}.$$
Assume that there exist $-1<a<\infty$ and $\delta\in(0,1)$ such that
\begin{align*}
n-1+a>0,
\,\,\mbox{ and }\,\,\frac{(1-t)^{-\alpha}\om^*(t)}{(1-t)^a}\searrow 0,\,\,\mbox{ when }\,\,\delta\leq t<1 .
\end{align*}
Then the following statements are equivalent.
\begin{enumerate}[(i)]
  \item $\HT_\mu\in \mcS_p(H(\Waw))$;
  \item $M_\mu=\sum\limits_{R_{k,j}\in\Upsilon}\left(\frac{\mu(R_{k,j})}{(1-|c_{k,j}|)^{-\alpha+n-1}\om^*(c_{k,j})}\right)^p<\infty$;
  \item $\widehat{\mu_{r,\alpha}}\in L^p(\BB,d\lambda)$ for some (equivalently, for all) $r>0$.
\end{enumerate}
Moreover, $|\HT_\mu|_p^p\approx M_\mu\approx\|\widehat{\mu_{r,\alpha}}\|_{L^p_\lambda}^p.$
\end{Theorem}
\begin{proof}
For convenience, let
$$
\Psi(t)=\left\{
\begin{array}{cc}
  \Waw(\frac{1}{2}), & t\in [0,\frac{1}{2}],  \\
  \Waw(t), & t\in[\frac{1}{2},1).
\end{array}
\right.
$$

({\it ii})$\Rightarrow$({\it i}). Suppose $M_\mu<\infty$. By Theorem \ref{0828-1}, $\HT_\mu\in S_1(H(\Waw))$. So, $\HT_\mu:H(\Waw)\to H(\Waw)$ is compact.
Then, by (\ref{0626-2}), we have
\begin{align*}
\langle \HT_\mu f,g\rangle_{H(\Waw)}=&\langle f,g \rangle_{L_\mu^2}.
\end{align*}
Let $\{b_j\}_{j=1}^\infty$ be a $r$-lattice, $s$ be a large enough positive integer,    and $\{e_j\}_{j=1}^\infty$ be a fixed  orthonormal basis of $H(\Waw)$. Define $J:H(\Waw)\to H(\Waw)$ as  in Lemma \ref{0916-1}. Then $J$ is bounded and onto.
Let  $J(e_j)=h_j,j=1,2,\cdots$.
By Proposition 1.30 in \cite{zhu}, we have
\begin{align*}
|J^*\HT_\mu J|_p^p
&\leq \sum_{x=1}^\infty\sum_{y=1}^\infty |\langle J^*\HT_\mu Je_x,e_y\rangle_{H(\Waw)}|^p
=\sum_{x=1}^\infty\sum_{y=1}^\infty |\langle  h_x,h_y\rangle_{L_\mu^2}|^p          \\
&\leq \sum_{x=1}^\infty \sum_{y=1}^\infty\left|\sum_{k=1}^\infty \int_{D(b_k,5r)}|h_x(z)h_y(z)d\mu(z)\right|^p\\
&\leq \sum_{k=1}^\infty (\mu(D(b_k,5r))^p\left(\sum_{x=1}^\infty |h_x(z_{x,k})|^p\right)\left(\sum_{y=1}^\infty|h_y(z_{y,k})|^p\right).
\end{align*}
Here, $z_{x,k}\in \ol{D(b_k,5r)}$ such that $|h_x(z_{x,k})|=\sup_{z\in D(b_k,5r)}|h_x(z)|$.
Since $s$ is large enough, using subharmonicity, by Lemma \ref{0916-2}, we have
\begin{align*}
\sum_{x=1}^\infty |h_x(z_{x,k})|^p
&\lesssim \sum_{x=1}^\infty  \frac{1}{(1-|z_{x,k}|)^{n+1}}\int_{D(z_{x,k},r)}|h_x(z)|^p dV(z)\\
&\lesssim \frac{1}{(1-|b_k|)^{n+1}} \int_{D(b_k,6r)}\left(\sum_{x=1}^\infty |h_x(z)|^p\right)dV(z),
\end{align*}
and
\begin{align*}
\sum_{x=1}^\infty |h_x(z)|^p
&=\sum_{x=1}^\infty \frac{(1-|b_x|^2)^\frac{sp}{2}}{\Waw(b_x)^\frac{p}{2}|1-\langle z,b_x\rangle|^\frac{(n-1+s)p}{2}}  \\
&\lesssim \sum_{x=1}^\infty \frac{(1-|b_x|^2)^{\frac{sp}{2}-(n+1)}}{\Waw(b_x)^\frac{p}{2}}\int_{D(b_x,\frac{r}{10})}\frac{1}{|1-\langle \eta, z\rangle|^{\frac{(n-1+s)p}{2}}}dV(\eta)  \\
&\lesssim \int_\BB \frac{(1-|\eta|^2)^{\frac{sp}{2}-(n+1)}}{\Psi(\eta)^\frac{p}{2}|1-\langle \eta, z\rangle|^{\frac{(n-1+s)p}{2}}}dV(\eta)
\approx \frac{1}{\Psi(z)^\frac{p}{2}(1-|z|)^{\frac{(n-1)p}{2}}}.
\end{align*}
Therefore,
\begin{align*}
\sum_{x=1}^\infty |h_x(z_{x,k})|^p
&\lesssim \frac{1}{(1-|b_k|)^{n+1}} \int_{D(b_k,6r)}  \frac{1}{\Psi(z)^\frac{p}{2}(1-|z|)^{\frac{(n-1)p}{2}}}  dV(z)\\
&\approx \frac{1}{\Waw(b_k)^\frac{p}{2}(1-|b_k|)^{\frac{(n-1)p}{2}}}.
\end{align*}
So, we obtain
\begin{align*}
|J^*\HT_\mu J|_p^p &\lesssim \sum_{k=1}^\infty \frac{(\mu(D(b_k,5r)))^p}{\Waw(b_k)^p(1-|b_k|)^{(n-1)p}}
\approx \sum_{k=1}^\infty \left(\frac{\mu(D(b_k,5r))}{(1-|b_k|)^{-\alpha+n-1}\om^*(b_k)}\right)^p.
\end{align*}
By Lemma \ref{0618-1},   $|J^*\HT_\mu J|_p^p\lesssim M_\mu$. By Proposition 1.30 in \cite{zhu}, $\HT_\mu\in \mathcal{S}_p(H(\Waw))$  and $|\HT_\mu|_p^p\lesssim M_\mu$. So, ({\it i}) holds.

({\it i})$\Rightarrow$({\it ii}). Suppose $\HT_\mu\in \mathcal{S}_p(H(\Waw))$. Let $\{\tau_j\}_{j=1}^\infty$ be a $r$-lattice, $s$ be a large enough positive integer,    and $\{e_j\}$ be a fixed orthonormal basis of $H(\Waw)$.

For any given $R>10r$, $\{\tau_j\}$ can be divided into $N_R$ subsequences such that the Bergman metric between any two points in each subsequence is at least $2R$. Let $\{b_j\}$ be such a subsequence and define
$$d\mu_*(z)=\sum_{j=1}^\infty \chi_{D(b_j,5r)}(z)d\mu(z).$$
Then we have $|\HT_{\mu_*}|_p\leq |\HT_\mu|_p$.

Define $J:H(\Waw)\to H(\Waw)$ as we did in Lemma \ref{0916-1} and let $J(e_j)=h_j,j=1,2,\cdots$. Then $J$ is bounded.
Let $  \|\cdot\|=\|\cdot\|_{H(\Waw\to H(\Waw))} $ for short.  So,
$$|J^*\HT_{\mu_*}J|_p\leq \|J\|^2 |\HT_{\mu_*}|_p\leq \|J\|^2 |\HT_{\mu}|_p.$$

For any $f\in H(\Waw)$, let
\begin{align*}
D(f)=\sum_{k=1}^\infty \langle \HT_{\mu_*} h_k,h_k\rangle_{H(\Waw)}\langle f,e_k\rangle_{H(\Waw)}e_k,
\end{align*}
and
$$E(f)=\sum_{j=1}^\infty\sum_{k\neq j} \langle \HT_{\mu_*}(h_k),h_j\rangle_{H(\Waw)}\langle f,e_k\rangle_{H(\Waw)}e_j.$$
Since $$\langle J^*\HT_{\mu_*} J e_k ,e_j\rangle_{H(\Waw)}=\langle \HT_{\mu_*}(h_k),h_j\rangle_{H(\Waw)},\,\,j,k=1,2,\cdots,$$
we get $J^*\HT_{\mu_*} J=D+E$.

By (\ref{0626-2}), there exists a constant $C_1>0$ independent on $R$, such that
\begin{align*}
|D|_p^p  &=\sum_{k=1}^\infty |\langle \HT_{\mu_*} h_k ,h_k\rangle_{H(\Waw)}|^p  =\sum_{k=1}^\infty \left(\int_\BB |h_k(z)|^2d{\mu_*}(z)\right)^p\\
&\geq \sum_{k=1}^\infty \left(\int_{D(b_k,6r)} |h_k(z)|^2d{\mu_*}(z)\right)^p
\geq C_1\sum_{k=1}^\infty \left(\frac{\mu_*(D(b_k,6r))}{(1-|b_k|)^{-\alpha+n-1}\om^*(b_k)}\right)^p.
\end{align*}
By Proposition 1.29 in \cite{zhu} and   (\ref{0626-2}), we have
{\small
\begin{align}
|E|_p^p
&\leq \sum_{j=1}^\infty\sum_{k\neq j}|\langle \HT_{\mu_*} h_k,h_j\rangle_{H(\Waw)}|^p
= \sum_{j=1}^\infty\sum_{k\neq j}\left(\int_{\BB}|h_k(z)||h_j(z)|d\mu_*(z)\right)^p  \nonumber\\
&= \sum_{j=1}^\infty\sum_{k\neq j}\left(\sum_{i=1}^\infty \int_{D(b_i,5r)}|h_k(z)||h_j(z)|d\mu_*(z)\right)^p   \nonumber \\
&\leq \sum_{j=1}^\infty\sum_{k\neq j}\sum_{i=1}^\infty \left(\int_{D(b_i,5r)}|h_k(z)||h_j(z)|d\mu_*(z)\right)^p \nonumber\\
&\leq \sum_{i=1}^\infty (\mu_*(D(b_i,5r)))^p
\left(\sum_{j=1}^\infty\sum_{k\neq j} |h_k(z_{k,i})|^p |h_j(z_{j,i})|^p\right).\label{0911-1}
\end{align}
}Here,  $z_{j,i}\in\ol{D(b_i,5r)}$ such that $|h_j(z_{j,i})|=\sup_{z\in D(b_i,5r)}|h_j(z)|$.
Using subharmonicity,
we have
{\small
\begin{align*}
\sum_{k\neq j}|h_k(z_{k,i})|^p
\lesssim &\sum_{k\neq j} \frac{(1-|b_k|^2)^{\frac{sp}{2}}}{\Waw(b_k)^\frac{p}{2}}
 \int_{D(b_k,r)}\int_{D(z_{k,i},r)} \left|\frac{1}{1-\langle z, u\rangle}\right|^\frac{(n-1+s)p}{2}d\lambda(z)d\lambda(u)    \\
 \lesssim &\sum_{k\neq j} \frac{(1-|b_k|^2)^{\frac{sp}{2}}}{\Psi(b_k)^\frac{p}{2}}
 \int_{D(b_k,r)}\int_{D(b_i,6r)} \left|\frac{1}{1-\langle z, u\rangle}\right|^\frac{(n-1+s)p}{2}d\lambda(z)d\lambda(u)    \\
\lesssim &\sum_{k\neq j}\int_{D(b_k,r)} \frac{(1-|u|^2)^\frac{sp}{2}}{\Psi(u)^\frac{p}{2}}\int_{D(b_i,6r)} \left|\frac{1}{1-\langle z, u\rangle}\right|^\frac{(n-1+s)p}{2} d\lambda(z)d\lambda(u).
\end{align*}
}
So, by Fubini's theorem, the double sums $\sum\limits_{j=0}^\infty\sum\limits_{k\neq j}$ in (\ref{0911-1}) are dominated by a constant (dependent on $R$) times
\begin{align*}
&\iint\limits_{\cup_{j}\cup_{k\neq j} D(b_j,r)\times D(b_k,r) } \frac{(1-|u|^2)^\frac{sp}{2}}{\Psi(u)^\frac{p}{2}}\frac{(1-|v|^2)^\frac{sp}{2}}{\Psi(v)^\frac{p}{2}}  \\
&\,\,\,\,\,\,\,\,\,\,\,\,\,\,\,\,\cdot \left(\iint\limits_{D(b_i,6r)\times D(b_i,6r)}\left|\frac{1}{(1-\langle z,u\rangle)(1-\langle w,v\rangle)}\right|^\frac{(n-1+s)p}{2}
d\lambda(z)d\lambda(w)\right)d\lambda(u)d\lambda(v)  \\
&\leq
\iint\limits_{G_R} \frac{(1-|u|^2)^\frac{sp}{2}}{\Psi(u)^\frac{p}{2}}\frac{(1-|v|^2)^\frac{sp}{2}}{\Psi(v)^\frac{p}{2}}    \\
&\,\,\,\,\,\,\,\,\,\,\,\,\,\,\,\,\,\cdot\left(\iint\limits_{D(b_i,6r)\times D(b_i,6r)}\left|\frac{1}{(1-\langle z,u\rangle)(1-\langle w,v\rangle)}\right|^\frac{(n-1+s)p}{2}
d\lambda(z)d\lambda(w)\right)d\lambda(u)d\lambda(v)  \\
&\lesssim  \frac{1}{(1-|b_i|)^{2(n+1)}}\iint\limits_{D(b_i,6r)\times D(b_i,6r)}\left(\iint\limits_{G_R} Y(u,v,z,w)dV(u)dV(v)\right)dV(z)dV(w),
\end{align*}
where
$$Y(u,v,z,w)= \frac{(1-|u|^2)^{\frac{sp}{2}-(n+1)}(1-|v|^2)^{\frac{sp}{2}-(n+1)}}{\Psi(u)^\frac{p}{2}\Psi(v)^\frac{p}{2}|(1-\langle z,u\rangle)(1-\langle w,v\rangle)|^\frac{(n-1+s)p}{2}}$$
and
$$G_R=\left\{(u,v):\beta(u,v)>2R-2r\right\}\supset \cup_{j}\cup_{k\neq j} D(b_j,r)\times D(b_k,r). $$

Since $z,w\in D(b_i,6r)$, by (2.20) in \cite{Zk2005} and Lemma \ref{0916-2}, for any given $\varepsilon>0$, there exists $R_\varepsilon>10r$ such that
\begin{align*}
K_{i,1}(z,w)&=   \int_{\BB\backslash D(b_i,R_\varepsilon)}\int_{\BB\backslash D(v,2R_\varepsilon-2r)} Y(u,v,z,w) dV(u)dV(v) \\
&\leq \int_{\BB\backslash D(b_i,R_\varepsilon)}\int_{\BB} Y(u,v,z,w) dV(u)dV(v)  \\
&\lesssim \frac{1}{\Psi(z)^\frac{p}{2}(1-|z|^2)^\frac{(n-1)p}{2}}\int_{\BB\backslash D(b_i,R_\varepsilon)}
\frac{(1-|v|^2)^{\frac{sp}{2}-(n+1)}}{|1-\langle w,v\rangle|^\frac{(n-1+s)p}{2}\Psi(v)^\frac{p}{2}}dV(v)\\
&\approx \frac{1}{\Psi(b_i)^\frac{p}{2}(1-|b_i|^2)^\frac{(n-1)p}{2}}\int_{\BB\backslash D(b_i,R_\varepsilon)}
\frac{(1-|v|^2)^{\frac{sp}{2}-(n+1)}}{|1-\langle b_i,v\rangle|^\frac{(n-1+s)p}{2}\Psi(v)^\frac{p}{2}}dV(v)\\
&\lesssim\varepsilon\left( \frac{1}{\Psi(b_i) (1-|b_i|)^{n-1}}\right)^p
\end{align*}
and
\begin{align*}
K_{i,2}(z,w)&=\int_{D(b_i,R_\varepsilon)}\int_{\BB\backslash D(v,2R_\varepsilon-2r)} Y(u,v,z,w) dV(u)dV(v)  \\
&\leq \int_{\BB}\int_{\BB\backslash D(b_i,R_\varepsilon-2r)} Y(u,v,z,w) dV(u)dV(v)  \\
&\lesssim\varepsilon\left( \frac{1}{\Psi(b_i) (1-|b_i|)^{n-1}}\right)^p.
\end{align*}
Let $R=R_\varepsilon$. Since
\begin{align*}
\iint\limits_{G_{R_\varepsilon}} Y(u,v,z,w)dV(u)dV(v)
=\int_\BB\int_{\BB\backslash D(v,2R_\varepsilon-2r)} Y(u,v,z,w) dV(u)dV(v),
\end{align*}
the double sums $\sum\limits_{j=0}^\infty\sum\limits_{k\neq j}$ in (\ref{0911-1}) are dominated by a constant (dependent on $R$) times
\begin{align*}
 & \frac{1}{(1-|b_i|)^{2(n+1)}}\iint\limits_{D(b_i,6r)\times D(b_i,6r)}\left(\iint\limits_{G_R} Y(u,v,z,w)dV(u)dV(v)\right)dV(z)dV(w)  \\
 \lesssim &
 \frac{1}{(1-|b_i|)^{2(n+1)}}\iint\limits_{D(b_i,6r)\times D(b_i,6r)} (K_{i,1}(z,w)+K_{i,2}(z,w))dV(z)dV(w)  \\
 \lesssim&
 \frac{\varepsilon}{\Psi (b_i)^p(1-|b_i|)^{(n-1)p}}
 \approx \frac{\varepsilon}{(1-|b_i|)^{(-\alpha+n-1)p}\om^*(b_i)^p}.
\end{align*}
Therefore,
$$|E|_p^p\lesssim \varepsilon\sum_{i=1}^\infty \left(\frac{\mu_*(D(b_i,5r))}{(1-|b_i|)^{-\alpha+n-1}\om^*(b_i)}\right)^p.$$
So, there exist $C_1,C_2$, such that
\begin{align*}
|\HT_{\mu^*}|_p^p
\geq |D|_p^p-|E|_p^p
\gtrsim (C_1-C_2\varepsilon) \sum_{i=1}^\infty \left(\frac{\mu_*(D(b_i,5r))}{(1-|b_i|)^{-\alpha+n-1}\om^*(b_i)}\right)^p
\end{align*}
for any given $\varepsilon>0$. Let $\varepsilon=\frac{C_1}{2C_2}$. Then we can choose $R_\varepsilon$ and $N_{R_\varepsilon}$ such that
 \begin{align*}
N_{R_\varepsilon}|\HT_{\mu}|_p^p
\geq \frac{C_1}{2} \sum_{i=1}^\infty \left(\frac{\mu_*(D(\tau_i,5r))}{(1-|\tau_i|)^{-\alpha+n-1}\om^*(\tau_i)}\right)^p.
\end{align*}
By Lemma \ref{0618-1}, $M_\mu\lesssim |\HT_\mu|_p^p$.

({\it ii})$\Leftrightarrow$({\it iii}) This can be obtained by the proof of Theorem \ref{0828-1}.

The proof is complete.
\end{proof}

If $\om\in\hD$, by Lemma \ref{1210-3}, $(1-t)^{-1}\om^*(t)\in\R$. By Lemma \ref{0507-1}, there exist $a>-1$ and $\delta\in(0,1)$ such that \begin{align*}
\frac{(1-t)^{-1}\om^*(t)}{(1-t)^a}\searrow 0,\,\,\mbox{ when }\,\,\delta\leq t<1.
\end{align*}
  Therefore, by Lemma \ref{0607-8},  Theorems  \ref{0828-1} and \ref{0919-1}, we get two  characterizations of  $\T_\mu\in S_p(A_\om^2)$ for positive Borel Measure $\mu$ and $\om\in\hD$ as follows.
\begin{Theorem}\label{0829-2}
Let $0<p<\infty$, $\om\in\hD$ and $\mu$ be a positive Borel measure on $\BB$.
Let
$$\widehat{\mu_r}=\frac{\mu(D(z,r))}{(1-|z|)^{n-1}\om^*(z)},\,\,\mbox{ and }\,\,d\lambda(z)=\frac{dV(z)}{(1-|z|^2)^{n+1}}.$$
Then the following statements are equivalent.
\begin{enumerate}[(i)]
  \item $\T_\mu\in \mcS_p(A_\om^2)$;
  \item $$ \sum_{R_{k,j}\in\Upsilon}\left(\frac{\mu(R_{k,j})}{(1-|c_{k,j}|)^{n-1}\om^*(c_{k,j})}\right)^p<\infty;$$
  \item $\widehat{\mu_r}\in L^p(\BB,d\lambda)$ for some (equivalently for all) $r>0$.
\end{enumerate}
\end{Theorem}

\section{Schatten class  Volterra operator}

Suppose $0<p,q<\infty$ and $\om\in\hD$. For any $g\in H(\BB)$, the Volterra integral operator on $H(\BB)$ is defined by
$$T_g f(z)=\int_0^1 f(tz)\Re g(tz)\frac{dt}{t}, \,\,f\in H(\BB),\,\,z\in\BB.$$
The operator $T_g$ was   introduced in \cite{hu}. See  \cite{DjLsLxSy2019arxiv, hu, hu3,  pau}, for example, for the study of this operator.  In \cite{DjLsLxSy2019arxiv}, we characterized the boundedness and compactness of $T_g:A_\om^p\to A_\om^q$ when $\om\in\hD$ and $0<p\leq q<\infty$ by using two kinds of function spaces $\mathcal{C}^1(\om^*)$ and $\mathcal{C}_0^1(\om^*)$, respectively.  Recall that $\mathcal{C}_0^1(\om^*)$ consists of all the functions $g\in H(\BB)$ such that
$$\lim_{|a|\to 1}\frac{\int_{S_a}|\Re g(z)|^2\om^*(z)dV(z)}{\om(S_a)}=0.$$

When $p>0$, the Besov space $B_p$ on   $\BB$ is the space consisting of all $g\in H(\BB)$ such that
$$\int_\BB |\Re g(z)|^p(1-|z|^2)^{p-(n+1)}dV(z)<\infty.$$
As we know, when $p\leq n$, $B_p=\CC$.

Let $\B_0$ and $VMOA$ denote the little Bloch space and  vanishing mean oscillation holomorphic function space, respectively, see \cite{Zk2005} for example.

\begin{Lemma}\label{0829-1}
Suppose  $p>0$ and $\om\in\hD$. For any $g\in H(\BB)$, $g\in B_p$ if and only if
\begin{align}
\sum_{R_{k,j}\in\Upsilon}\left(\int_{R_{k,j}}|\Re g(z)|^2\frac{dV(z)}{(1-|z|)^{n-1}}\right)^\frac{p}{2}<\infty.\label{0829-3}
\end{align}
\end{Lemma}
\begin{proof}
By Theorem 1.1 in \cite{HzTx2010pams}, $g\in B_p$ if and only if $T_g\in \mathcal{S}_p(A^2)$.

Suppose $g\in B_p$. Since $B_p\subset\B_0$, $T_g:A^2\to A^2$ is compact. By Theorem 1.26 in \cite{zhu}, $T_g^*T_g\in \mathcal{S}_{\frac{p}{2}} (A^2)$.
For any $f,h\in A^2$, by Theorem 2 in \cite{DjLsLxSy2019arxiv}, we have
\begin{align*}
\langle (T_g)^*T_g f,h\rangle_{A^2} &=\langle T_g f,T_g h\rangle_{A^2}  \\
&=4\int_\BB f(z)\ol{h(z)}|\Re g(z)|^2\left( \frac{1}{2n}\log\frac{1}{|z|}-\frac{1}{4n^2}(1-|z|^{2n})  \right)\frac{dV(z)}{|z|^{2n}}.
\end{align*}
Let
\begin{align*}
d\mu_g(z)=4|\Re g(z)|^2\left( \frac{1}{2n}\log\frac{1}{|z|}-\frac{1}{4n^2}(1-|z|^{2n})  \right) \frac{dV(z)}{|z|^{2n}}.
\end{align*}
Then, $\mu_g(\BB)<\infty$ and
\begin{align*}
d\mu_g(z)\approx |\Re g(z)|^2(1-|z|)^2dV(z), \mbox{  as  }|z|\to 1.
\end{align*}
Let $B_w$ be the reproducing kernel of $A^2$ and $h=B_w$. We have $(T_g)^*T_g f=\T_{\mu_g}f$.  So, $\T_{\mu_g}\in S_{\frac{p}{2}}(A^2)$.
By Corollary \ref{0829-2}, (\ref{0829-3}) holds.

Suppose (\ref{0829-3}) holds.  For any given $\varepsilon>0$, there exists  $N_\varepsilon\in\NN$ such that
\begin{align}\label{0830-1}
\sum_{k>N_\varepsilon}\left(\int_{R_{k,j}}|\Re g(z)|^2\frac{dV(z)}{(1-|z|)^{n-1}}\right)^\frac{p}{2}<\varepsilon.
\end{align}
Since $|\Re g|^2$ is subharmonic, for any fixed $0<r<1$ and any  $z\in\BB$, we have
\begin{align*}
(1-|z|)^2|\Re g(z)|^2  \lesssim (1-|z|)^2\frac{\int_{\Delta(z,r)}|\Re g(\eta)|^2dV(\eta)}{V(\Delta(z,r))}
\approx \int_{\Delta(z,r)}  \frac{|\Re g(\eta)|^2}{(1-|\eta|)^{n-1}}dV(\eta).
\end{align*}
By \ref{0618-1}  and (\ref{0830-1}), there exist  $C=C(r,p)$ and $\delta\in(0,1)$ such that
$$\sup_{|z|>\delta} (1-|z|)^2|\Re g(z)|^2<C\varepsilon^\frac{2}{p}.$$
That is to say, $g\in\B_0$. Then $T_g:A^2\to A^2$ is compact, see  \cite{hu3,DjLsLxSy2019arxiv} for example.
Using Theorem 2 in \cite{DjLsLxSy2019arxiv}, we have $\mu_g(\BB)=\|T_g(1)\|_{A_2^2}<\infty.$
From the above proof, we have $T_g^*T_g=\T_{\mu_g}$. By Corollary \ref{0829-2} and (\ref{0829-3}), $\T_{\mu_g}\in \mathcal{S}_\frac{p}{2}(A^2)$.
Therefore, $T_g\in \mathcal{S}_p(A^2)$ and $g\in B_p$.
The proof is complete.
\end{proof}

\begin{Theorem}\label{0830-2}
Suppose $\om\in\hD$. If $p>n$, $T_g\in \mathcal{S}_p$ if and only if $g\in B_p$. If\, $0<p\leq n$, $T_g\in \mathcal{S}_p$ if and only if $g$ is a constant.
\end{Theorem}

\begin{proof}The case of $n=1$ is Theorem 6.1 in \cite{PjaRj2014book}. So, we can assume $n\geq 2$.

Suppose $g\in B_p$ and $p>n$. Since $B_p\subset VMOA$, by \cite[Proposition 4]{DjLsLxSy2019arxiv} and  \cite[Theorem 5]{DjLsLxSy2019arxiv}, $T_g$ is compact on $A_\om^2$.  Let
$$d\mu_g^\om(z)=\frac{4|\Re g(z)|^2\om^{n*}(z)dV(z)}{|z|^{2n}}.$$
By Theorem 2 in \cite{DjLsLxSy2019arxiv}, for all $f,g\in A_\om^2$, we get
\begin{align*}
\langle (T_g)^*T_g f,h\rangle_{A_\om^2} &=\langle T_g f,T_g h\rangle_{A_\om^2}  \\
&=4\int_\BB f(z)\ol{h(z)}\frac{|\Re g(z)|^2\om^{n*}(z)dV(z)}{|z|^{2n}} =\int_\BB f(z)\ol{h(z)}d\mu_g^\om(z).
\end{align*}
Let $B_w^\om$ be the reproducing kernel of $A_\om^2$ and $h=B_w^\om$. We have $(T_g)^*T_g f=\T_{\mu_g^\om}f$.
Using the fact $\mu_g^\om(\BB)<\infty$ and $\frac{\om^{n*}(z)}{|z|^{zn}}\approx \om^*(z)$ as $|z|\to 1$,
by Lemma \ref{1210-3} we have
$$\mu_g^\om (R_{k,j})\approx  \om^{*}(c_{k,j})(1-|c_{k,j}|)^{n-1} \int_{R_{k,j}}|\Re g(z)|^2\frac{dV(z)}{(1-|z|)^{n-1}}.$$
By Corollary \ref{0829-2} and Lemma \ref{0829-1}, $\T_{\mu_g^\om}\in \mathcal{S}_{\frac{p}{2}}(A_\om^2)$.
By Theorem 1.26 in \cite{zhu},  $T_g\in \mathcal{S}_p(A_\om^2)$.
When $0< p\leq n$ and $g$ is a constant, it is obvious that $T_g\in \mathcal{S}_p(A_\om^2)$.

Conversely, we assume that $T_g\in \mathcal{S}_p(A_\om^2)$.
Using Corollary \ref{0829-2} and Lemma \ref{0829-1}, we get that $g\in B_p$. Moreover, when $ p\leq n$, $B_p=\CC$.
The proof is complete.
\end{proof}

\noindent {\bf Acknowledgments.} 
 The corresponding author was supported by the Macao Science and Technology Development Fund (No.186/2017/A3) and NNSF of China (No. 11720101003).

\end{document}